\documentclass{siamltex}

\usepackage{amsmath}
\usepackage{amsfonts}
\usepackage{amssymb}

\usepackage[T1]{fontenc}
\usepackage[latin1]{inputenc}
\usepackage[greek,francais,english]{babel}
\numberwithin{equation}{section}

\usepackage{vmargin}
\setmarginsrb{1in}{1in}{1in}{1in}{0cm}{0cm}{0cm}{1.5cm}

\newcommand{\R}{\mathbb R}
\newcommand{\N}{\mathbb N}
\newcommand{\Z}{\mathbb Z}

\newcommand{\T}{\mathbb T}
\newcommand{\bE}{\mathbb E}
\newcommand{\cL}{\mathcal L}
\newcommand{\cEl}{\mathcal E_\lambda}
\newcommand{\cE}{\mathcal E}

\newcommand{\cH}{\mathcal H}
\newcommand{\Proj}{\mathbb P}
\newcommand{\e}{\varepsilon}

\newcommand{\ug}{u^{\varepsilon,\nu}}

\newcommand{\pg}{p^{\varepsilon,\nu}}

\newcommand{\om}{\omega}

\newcommand{\be}{\begin{equation}}
\newcommand{\ee}{\end{equation}}
\newcommand{\ds}{\displaystyle}
\newcommand{\p}{\partial}
\newcommand{\dv}{\mathrm{div}}
\newcommand{\rot}{\mathrm{rot}}

\newcommand{\sgn}{\mathrm{sgn}}
\newcommand{\uint}{u^{\text{int}}}
\newcommand{\ustat}{u^{\text{BL,res}}}
\newcommand{\vint}{v^{\text{int}}}
\newcommand{\buint}{{\bar u}^{\text{int}}}
\newcommand{\wnd}{w_{n,N}^\delta}
\newcommand{\ubl}{u^{\text{BL}}}

\newcommand{\mean}[1]{\left\langle #1\right\rangle}
\newcommand{\uapp}{u^{\text{app}}}
\newcommand{\wrm}{w^{\text{rem}}}
\newcommand{\bw}{\bar w}
\newcommand{\cFa}{\mathcal F_\alpha}
\newcommand{\cO}{\mathcal O}
\newcommand{\U}{\Upsilon}

\newtheorem{remark}{Remark}[section]

\title{Asymptotic behaviour of a rapidly rotating fluid with random stationary surface stress}

\author{Anne-Laure Dalibard\thanks{CNRS/D\'epartement de math\'ematiques et applications, UMR 8553, \'Ecole normale sup\'erieure, 45 rue d'Ulm, F-75005 Paris, France.}}

\begin{document}

\maketitle

\begin{abstract}
The goal of this paper is to describe in mathematical terms the effect on the ocean circulation of a random stationary wind stress at the surface of the ocean. In order to avoid singular behaviour, non-resonance hypotheses are introduced, which ensure that the time frequencies of the wind-stress are different from that of the Earth rotation. We prove a convergence result for a three-dimensional Navier-Stokes-Coriolis system in a bounded domain, in the asymptotic of fast rotation and vanishing vertical viscosity, and we exhibit some random and stationary boundary layer profiles. At last, an average equation is derived for the limit system in the case of the non-resonant torus.
\end{abstract}



\section{Introduction}
The goal of this paper is to study mathematically a problem arising in ocean dynamics, namely the behaviour of ocean currents under stimulation by the wind. Following the books by Pedlosky \cite{Pedlosky1,Pedlosky2} and Gill \cite{Gill}, the velocity of the fluid in the ocean, denoted by $u$, is described by the incompressible Navier-Stokes equations in three dimensions, in rotating coordinates, with Coriolis force
\begin{eqnarray*}
\rho(\p_t u + u\cdot \nabla u+ 2 \Omega e\wedge u) - A_h\Delta_h u - A_z\p_z^2 u &=&\nabla p,\quad t>0,\ (x,y,z)\in U(t)\subset\R^3,\\
\dv u&=&0.
\end{eqnarray*}
In the above equation, $A_h$ and $A_z$ are respectively the horizontal and vertical turbulent viscosities, $p$ is the pressure inside the fluid, $\rho$ is the homogeneous and constant density, and $\Omega e$ is the rotation vector of the Earth ($\Omega>0$ and $e$ is a unitary vector, parallel to the pole axis, oriented from South to North). $U(t)$ is an open set in $\R^3$, depending  on the time variable $t$: indeed, the interface between the ocean and the atmosphere may be moving, and is described in general by a free surface $z=h(t)$. 

In order to focus on the influence of the wind, let us now make a series of crude modeling hypotheses on the boundary conditions: first, we assume that the lateral boundaries of the ocean are flat, and that the velocity $u$ satisfies periodic boundary conditions in the horizontal variable. We also neglect the fluctuations of the free surface, namely, we assume that $h(t)\equiv a D$, with $a, D$  positive constants. This approximation, although highly unrealistic, is justified by the fact that the behaviour of the fluid around the surface is in general very turbulent. Hence, as emphasized in \cite{DesjardinsGrenier}, only a modelization is tractable and meaningful. Let us also mention that the justification of this rigid lid approximation starting from a free surface is mainly open from a mathematical point of view: we refer to \cite{AP} for the derivation of Navier-type wall laws for the Laplace equation, under general assumptions on the interface, and to \cite{LT} for some elements of justification in the case of the great lake equations.
At last, we assume that the bottom of the ocean is flat;  the case of a nonflat bottom has already been investigated by several authors, and we refer to  \cite{DesjardinsGrenier,dgv,MasmoudiCPAM} for more details regarding that point.

As a consequence, we assume that $U(t)=[0,a_1 H)\times [0,a_2 H )\times [0,a D)$, where $H>0$ is the typical horizontal lengthscale,  and $u$ satisfies the following boundary conditions
\begin{eqnarray*}
&& u\text{ is periodic in the horizontal variable  with period }[0,a_1 H)\times [0,a_2 H ),\\
&&u_{|z=0}=0\quad\text{(no slip condition at the bottom of the ocean),}\\
&&\p_z u_{h|z=aD}=A_0\sigma\quad\text{(influence of the wind),}\\
&& u_{3|z=aD}=0\quad\text{(no flux condition at the surface)}.
\end{eqnarray*}

Let us now reduce the problem by scaling arguments. First, we neglect the effect of the horizontal component of the rotation vector $e$, which is classical in a geophysical framework (see \cite{CheminDGG}). Furthermore, we assume that the motion occurs at midlatitudes (far from the equator), and on a ``small'' geographical zone, meaning $H\ll R_0$, where $R_0$ is the earth radius. In this setting, it is legitimate to use the so-called {\it $f$-plane approximation} (see \cite{SaintRayGallagherHbk}), and to neglect the fluctuations of the quantity $e_3\cdot e$ with respect to the latitude. In rescaled variables, the equation becomes
\be
\p_t \ug + \ug\cdot \nabla \ug+ \frac{1}{\e}e_3\wedge \ug - \nu_h\Delta_h \ug -\nu_z\p_z^2 \ug  + \begin{pmatrix}
                                                                                                   \nabla_h p\\\frac{1}{\gamma^2}\p_z p
                                                                                                  \end{pmatrix}
=0,\label{eq:shallow}
\ee
where
$$
\e:=\frac{U}{2H \Omega}, \ \nu_h:=\frac{A_h}{\rho u H}, \ \nu_z:=\frac{L A_z}{\rho U D^2},\ \gamma:=\frac{D}{H},
$$
and $U$ is the typical horizontal relative velocity of the fluid. We are interested in the limit
$$
\nu_z\ll 1,\quad \e\ll 1,\quad \nu_h\sim 1.
$$Such a { scaling of parameters}
  seems convenient for instance for the mesoscale eddies that have been observed in western Atlantic (see \cite{Pedlosky1}). One has indeed
 $$U\sim  5\,\mathrm{cm\cdot s}^{-1} ,\quad H\sim 100\,\mathrm{km}, \quad D\sim 4\,\mathrm{km} \hbox{ and }\Omega \sim 10^{-4} \mathrm s^{-1}$$
 which leads to $\e \sim 5\times 10^{-3} .$
 Possible values for the turbulent viscosities given in \cite {Pedlosky1} are
 $$A_h\sim 10^6\,\mathrm{kg\cdot m^{-1} \cdot s^{-1}} \hbox{ and } A_z \sim 1\,\mathrm{kg\cdot m^{-1} \cdot s^{-1}} $$
 so that
 $\nu_z\sim 10^{-3},\ \nu_h\sim 1.$ In the rest of the article, we denote by $\nu$ the small parameter $\nu_z$, and we assume that $\nu_h=1$.
Additionally, we do not take into account the shallow water effect, and thus we take $\gamma=1$, even though this is not consistent with the values of $H$ and $D$ given above. Indeed, the thin layer effect, which corresponds to $\gamma\ll 1$, is expected to substantially complicate the analysis, but without modifying the definition of boundary layers. Thus, in order to focus on the influence of a random forcing, we study the classical rotating fluids equation (see for instance \cite{CheminDGG}), that is
\be
\p_t \ug + \ug\cdot \nabla \ug+ \frac{1}{\e}e_3\wedge \ug - \Delta_h \ug -\nu\p_z^2 \ug + \nabla p=0.\label{eq:depart}
\ee
Moreover, the amplitude of the wind stress at the surface of the ocean may be very large; thus we set
$$
\beta:=\frac{A_0S_0D}{U},
$$
where $S_0$ is the amplitude of the wind velocity, and we study the limit $\beta\to \infty$. Equation \eqref{eq:depart} is thus supplemented with the boundary conditions
\be\label{CL}
\begin{aligned}
&&\ug_{|z=0}=0,\\
 &&\p_z \ug_{h|z=a}={\beta}\sigma^{\e},\\
&&\ug_{3|z=a}=0.
\end{aligned}
\ee
Additionally, $\ug$ is assumed to be $\T^2$-periodic in the horizontal variable $x_h$, where $\T^2:=\R^2/[0,a_1)\times [0,a_2)$. 
In the rest of the paper, we set $\U:=\T^2\times (0,a)$.
The assumptions on the wind-stress $\sigma^\e$ will be made clear later on.

\subsection{General results on rotating fluids}

\label{ssec:gal}
Let us now explain heuristically what is the expected form of $u^{\e,\nu}$ at the limit. Assume for instance that $\nu=\e$ and that the family $u^{\e,\nu}$ behaves in $L^2([0,T]\times \U)$ like some function $u^0(t,t/\e, x)$, with $u^0\in L^\infty([0,T]\times[0,\infty)\times \U)$ sufficiently smooth. Then it is natural to expect that $u^0$ satisfies
\be\label{eq:w-lim}
\left\{
\begin{array}{l}
 \p_\tau u^0 + e_3\wedge u^0=0,\\
\dv u^0=0,\\
u^0_{3|z=0}=u^0_{3|z=a}=0.
\end{array}
\right.
\ee

In fact, the above equation can be derived rigorously from \eqref{eq:depart} if the dependence of the function $u^0$ with respect to the fast time variable $\tau$ is known \textit{a priori}; the goal of the two-scale convergence theory, formalized by Gr\'egoire Allaire in \cite{Allaire} after an idea of Gabriel N'Guetseng (see \cite{Ng}), is precisely to justify such derivations in the context of periodic functions. However, in this paper, we do not need to resort to such techniques; our aim is merely to build an approximate solution thanks to formal computations.

In view of \eqref{eq:w-lim}, we introduce the vector space
$$
\mathcal H:=\left\{ u\in L^2(\U)^3,\ \dv u=0,\  u_{3|z=0}=u_{3|z=a}=0\right\}.
$$
We denote by $\Proj$ the orthogonal projection on $\mathcal H$ in $L^2(\U)^3$, and we set $L:=\Proj(e_3\wedge \cdot)$. Notice that $\Proj$ differs from the Leray projector in general, because of the no-flux conditions at the bottom and the surface of the fluid. It is known (see for instance \cite{CheminDGG}) that there exists a hilbertian basis $(N_k)_{k\in \Z^3\setminus \{0\}}$ of $\cH$ such that for all $k$,
$$
\Proj(e_3\wedge N_k)  = i\lambda_k N_k \hbox{ with }\lambda_k =  -\frac{k_3'}{|k'|},
$$
where $k'=(2\pi k_1/a_1,2\pi k_2/a_2, \pi k_3/a ).$
The vector $N_k$ is given by
$$
N_k(x_h,z)=e^{ik_h'\cdot x_h}\begin{pmatrix}
                              \cos (k_3' z)n_1(k)\\ \cos (k_3' z)n_2(k)\\ \sin(k_3'z) n_3(k)
                             \end{pmatrix}
$$
where
$$\left\{\begin{array}{l}
 \ds n_1(k)=\frac{1}{\sqrt{a_1a_2a}|k'_h|}(ik'_2+k'_1\lambda_k)\\
 \ds n_2(k)=\frac{1}{\sqrt{a_1a_2a}|k'_h|}(-ik'_1+k'_2\lambda_k)\\
 \ds n_3(k)=i\frac{|k'_h|}{\sqrt{a_1a_2a}|k'|}
 \end{array}\right.
\quad\text{if }k_h\neq 0, $$
and
$$
\left\{\begin{array}{l}
 \ds n_1(k)=\frac{\sgn(k_3)}{\sqrt{a_1a_2a}}\\
 \ds n_2(k)=\frac{i}{\sqrt{a_1a_2a}}\\
\ds n_3(k)=0 \end{array}\right.\text{ else}.
$$
Let $\cL(\tau):\cH\to\cH$ be the semi-group associated with equation \eqref{eq:w-lim}, i.e. $\cL(\tau)=\exp(-\tau L)$ for $\tau\geq 0$. 
We infer from  equation \eqref{eq:w-lim} that $u^0(t,\tau)\in\cH$ almost everywhere, and that there exists a function $u^0_L$ such that
$$
u^0=\cL(\tau)u^0_L=\sum_{k}e^{-i\lambda_k\tau}\mean{N_k,u^0_L} N_k.
$$
Consequently the main effect of the Coriolis operator $L$ is to create waves, propagating at frequencies of order $\e^{-1}.$ The goal is now to identify the function $u^0_L$, which in general depends on the slow time variable $t$. This is achieved thanks to filtering methods, which were introduced by
S.~Schochet in \cite{schochet}, and further developed by E.~Grenier in \cite{grenier} in the context of rotating fluids. Precisely, setting
$$
u^{\e,\nu}_L= \exp\left( \frac{t}{\e} L\right) u^{\e,\nu}
$$
it is proved in \cite{CheminDGG,MasmoudiCPAM} in the case of Dirichlet boundary conditions at $z=0$ and $z=a$ that $u^{\e,\nu}_L$ converges strongly  in $L^2_\text{loc}([0,\infty)\times \U)$ towards a function $u^0_L$. Moreover, the function $u^0_L$ satisfies a nonlinear equation of the type
 \be\label{eq:envelope_gal}
\p_t u^0_L + \bar Q(u^0_L,u^0_L)- \Delta_h u^0_L= \bar S,
\ee
where the quadratic term $\bar Q(u^0_L,u^0_L)$ corresponds to the filtering of oscillations in the non-linear term $\ug \cdot \nabla\ug$, and the source term $\bar S$ to the filtering of oscillations in lower order terms in $u^{\e,\nu}.$ The quadratic term $\bar Q$ is defined as follows (see \cite{CheminDGG}, Proposition 6.1 and \cite{MasmoudiCPAM}): for $w_1,w_2\in \cH \cap H^1(\U)$, $\bar Q(w_1,w_2)$ is the weak limit as $\e\to 0$ of the quantity
\begin{eqnarray*}
&& \frac{1}{2} \exp\left( \frac{t}{\e} L\right)\Proj \left(\exp\left( -\frac{t}{\e} L\right)w_1\cdot \nabla \exp\left( -\frac{t}{\e} L\right)w_2 \right)\\
&+&\frac{1}{2}   \exp\left( \frac{t}{\e} L\right)\Proj \left(\exp\left( -\frac{t}{\e} L\right)w_2\cdot \nabla \exp\left( -\frac{t}{\e} L\right)w_1\right).
\end{eqnarray*}
Hence
\be
\bar Q(w_1,w_2)=\sum_{m\in \Z^3}\sum_{(k,l)\in\mathcal K_m}\mean{N_k, w_1}\mean{N_l,w_2}\alpha_{k,l,m} N_m,\label{def:Q}
\ee
where
the resonant set $\mathcal K_m$ is defined  for $m\in\Z^3\setminus\{0\}$, by
$$
\mathcal K_m:=\left\{ (k,l)\in \Z^6,\begin{array}{l} k_h+ l_h=m_h,\\ \lambda_k+\lambda_l=\lambda_m\\\end{array}
\text{ and }\exists \eta\in\{-1,1\}^2,\ \eta_1 k_3 + \eta_2 l_3 = m_3 \right\}
$$
and the coefficient $\alpha_{k,l,m}$ by
$$
\alpha_{k,l,m}=\frac{1}{2}\left(\mean{N_m,(N_k\cdot \nabla) N_l } + \mean{N_m,(N_l\cdot \nabla) N_k } \right).
$$
In order that the equation on $u^0_L$ is defined unambiguously, the value of the source term $\bar S$ has to be specified. In the present case, we have
$$
\bar S= -\sqrt\frac{\nu}{\e}S_B(u^0_L) - \nu\beta S_T(\sigma),
$$
where $S_B:\cH\to \cH$ is a linear continuous non-negative operator (see \cite{CheminDGG,forcage-resonant,MasmoudiCPAM}) recalled in formula \eqref{def:Sekman} below, and $S_T(\sigma)$ depends on the time oscillations in the wind-stress $\sigma.$ Thus, in the next paragraph, we precise the assumptions on the wind-stress $\sigma^{\e}$, and we define the source term $S_T$. In the above formula and throughout the article, the subscripts $B$ and $T$ refer to top and bottom, respectively.

\subsection{Definition of the limit equation}

Let us first introduce the hypotheses on the dependance of the wind velocity $\sigma^\e$ with respect to the time variable. Since the Coriolis operator generates oscillations at frequencies of order $\e^{-1}$, it seems natural to consider functions $\sigma^\e$ which depend on the fast time variable $t/\e$. The case where this dependance is periodic, or almost periodic, has been investigated by N. Masmoudi in \cite{MasmoudiCPAM} in the non-resonant case, that is, when the frequencies of the wind-stress are different from $\pm 1$. The results of \cite{MasmoudiCPAM} were then extended by the author and Laure Saint-Raymond in \cite{forcage-resonant}. In fact, it is proved in \cite{forcage-resonant} that when the wind-stress  oscillates with the same frequency as the rotation of the Earth (i.e. $\pm 1$), the typical size of the boundary layers is much larger than the one of the classical Ekman layers. Moreover, a resonant forcing may overall destabilize the whole fluid for large times. Here, we wish to avoid these singular behaviours, and thus to consider a more general non-resonant setting.

Let
$(E, \mathcal A,m_0)$ be a probability space, and let
$(\theta_\tau)_{\tau\in\R}$ be a measure preserving group transformation acting on
$E$. We assume that the function $\sigma^\e$ can be written
$$
\sigma^\e(t,x_h)=\sigma\left( t,\frac{t}{\e},x_h;\om \right),\quad t>0,\ x_h\in\T^2,\ \om\in E,
$$
and that the function $\sigma$ is stationary, i.e.
$$
\sigma(t,\tau+s,x_h;\om)=\sigma(t,\tau,x_h;\theta_s\om)
$$
almost everywhere.

The periodic setting can be embedded in the stationary (ergodic)
setting in the following way (see \cite{PapanicolaouVaradhan}): 
take $E=\R / \Z \simeq [0,1)$, and let $m_0$ be the Lebesgue measure on $E$. Define the group transformation $(\theta_\tau)_{\tau\in\R}$ by
$$
\theta_\tau s = s + \tau \quad \text{mod } \Z\quad \forall (\tau,s)\in \R\times E.
$$
Then it is easily checked that $\theta_\tau$ preserves the measure $m_0$ for all $\tau\in\R$. Thus the periodic setting is a particular case of the stationary setting; the almost periodic setting can also be embedded in the stationary setting, but the construction is more involved, and we refer the interested reader to \cite{PapanicolaouVaradhan}.

The interest of the stationary setting, in addition to its generalization of the almost periodic one, lies in the introduction of some randomness in equation \eqref{eq:depart}. Hence, we also expect to recover a random function in the limit $\e,\nu\to 0$. In fact, we will prove rigorously a strong convergence result of this kind; additionally, we will characterize the average behaviour of $\ug$ in the limit.  Thus, one of the secondary goals of this paper is to derive some averaging techniques adapted to highly rotating fluids, which may be of interest in the framework of a mathematical theory of weak turbulence.

Since the function $\sigma$ is not an almost periodic function, we now introduce a notion of approximate spectral decomposition of $\sigma$. For $\alpha>0$, define the operator $\cFa$ by
\be
\cFa\sigma:\lambda\in\R\mapsto \frac{1}{2\pi}\int_{\R} \exp(-\alpha |\tau|) e^{-i \lambda \tau}\sigma(\tau)\:d\tau,
\ee
and define the family of functions $(\sigma_\alpha)_{\alpha>0}$ by the formula
\be
\sigma_\alpha(\tau):= \int_{\R} \exp(-\alpha |\lambda|) e^{i \lambda \tau}\cFa\sigma(\lambda)\:d\lambda.
\ee
It is  proved in Appendix A (see Lemma A.1) that the family $(\sigma_\alpha)_{\alpha>0}$ converges towards $\sigma$, as $\alpha\to 0$, in $L^\infty([0,T_1]\times [0,T_2]\times E, L^2(\T^2))$ for all $T_1,T_2>0$. In order to simplify the presentation, we assume from now on that \textbf{$\sigma$ only has a finite number of horizontal Fourier modes}. This is not crucial in the convergence proof, but the non-resonance conditions on $\sigma$ are somewhat simpler to formulate in this case. We refer to Remark \ref{rem:regul_sigma} for more general assumptions.

From now on, we assume that for all $\alpha>0$, $T>0$, the function $\cFa \sigma $ belongs to $L^\infty([0,T]\times E\times \T^2, L^1\cap L^\infty(\R_\lambda))$, and that
the following non-resonance hypotheses hold:
\begin{itemize}
 \item[\bf(H1)] For all $T>0$, 
$$
\forall T>0,\quad\sup_{\alpha>0}||\cFa \sigma||_{L^\infty( [0,T]\times E\times \T^2, L^1(\R_\lambda))}<+\infty.
$$

\item[\bf(H2)] There exist neighbourhoods $V_\pm$ of $\pm 1$, independent of $\alpha>0$,  such that
$$
\forall T>0,\quad\lim_{\alpha\to 0}\sup_{\lambda\in V_{+}\cup V_-}\|\cFa \sigma(\lambda)\|_{L^\infty([0,T]\times E\times \T^2)}=0.
$$

\end{itemize}

We refer to Remark \ref{rem:H1H2} below for some details about  the meaning of hypotheses {\bf (H1)-(H2)} for almost periodic functions. The interested reader may also consult  \cite{Lannes} for a treatment a resonance phenomena for functions with a continuous spectrum; notice however that the context in \cite{Lannes} is somewhat different, since it deals with functions whose Fourier transform is well-defined. 

Let us now explain how random oscillations are filtered:
\begin{proposition}
Let $\phi \in L^\infty(\R_\tau, L^2(E))$ be stationary, and let $\lambda\in\R$. Then the family
$$
\phi^{\lambda}_\theta:\om\in E\mapsto\frac{1}{\theta}\int_0^\theta \phi(\tau,\om) e^{-i\lambda\tau}d\tau,\quad \theta>0
$$
converges, almost surely and in $L^2(E)$, towards a function denoted by $\cEl [\phi]\in L^2(E)$  as $\theta\to\infty$. Moreover, $\cEl[\phi]$ satisfies the identity
$$
\cEl[\phi](\theta_\tau\om)= \cEl[\phi](\om) e^{i\lambda \tau}
$$
almost surely in $\om$, for all $\tau\in\R$.

Additionally, if $\sigma $ satisfies {\bf (H1)-(H2)}, then
\be
\cEl[\sigma]=0\label{nonres_cEL}
\ee
for $\lambda$ in a neighbourhood of $\pm 1.$

\label{prop:ergodic}\end{proposition}
Proposition \ref{prop:ergodic} is proved in Appendix B, except property \eqref{nonres_cEL}, which will be proved in the course of the proof page \pageref{nonres_barsigma}.

With the above definition of $\cEl$, the source term $S_T$ is defined by
$$
S_T(\sigma)(t)=-\frac{1}{\sqrt{aa_1a_2}}\sum_{k\in\Z^3}\mathbf 1_{k_h\neq 0}\frac{(-1)^{k_3}}{|k_h'|}\left( \lambda_k k'_h +i(k_h')^\bot\right)\cdot \cE_{-\lambda_k}\left[ \hat \sigma(t,\cdot,k_h)  \right] N_k,
$$
where
$$
\hat \sigma(t,\tau,k_h;\om)=\int_{\T^2} \sigma(t,x_h;\om)e^{-ik_h'\cdot x_h}\:dx_h.
$$
Notice that $S_T(\sigma)$ is a random function in general, and is well-defined in $L^\infty_\text{loc}([0,\infty)\times E,L^2(\U))$ thanks to {\bf (H1)-(H2)} provided $\sigma\in L^\infty([0,T]\times[0,\infty)\times \T^2\times E)$ for all $T>0$.

\vskip1mm

$\bullet$ We now state an existence result for the limit system, based on the analysis in \cite{CheminDGG}. To that end, we introduce the anisotropic Sobolev spaces $H^{s,s'}$ by
$$
H^{s,s'}:=\left\{ u\in L^2(\U),\ \forall \alpha\in\N^3,\  |\alpha_h|\leq s, |\alpha_3|\leq s',\ \nabla_h^{\alpha_h} \p_z^{\alpha_3} u\in L^2(\U)\right\}.
$$
Then the following result holds:

\begin{proposition}Let $\nu,\e,\beta>0$ be arbitrary.

 Let $u_0\in \cH \cap H^{0,1},$ and let $\sigma \in L^\infty_\text{loc}([0,\infty)_t,L^\infty( [0,\infty)_\tau\times \T^2\times E)).$

Assume that the hypotheses {\bf (H1)-(H2)} hold. 

Then $S_T(\sigma)\in L^\infty_\text{loc}([0,\infty)_t, L^\infty(E , H^{0,1})),$ and consequently, the equation
\be
\begin{array}{l}
\ds \p_t w + \bar Q(w,w) - \Delta_h w + \sqrt{\frac{\nu}{\e}}S_B(w) + \nu\beta S_T(\sigma)=0,\\
\ds w_{|t=0}=u_0
\end{array}
\label{eq:lim}\ee
has a unique  solution $w\in L^\infty(E, \mathcal C([0,\infty), \cH\cap H^{0,1}))$ such that $\nabla_h u$ belongs to $ L^\infty(E, L^2_\text{loc}([0,\infty), H^{0,1})) $.

\label{prop:eqlim}
\end{proposition}

\begin{remark} {\bf (i) }  Notice that the function $w$ is random in general because of the source term $S_T$.\\
{\bf (ii)} In \cite{CheminDGG}, Proposition \ref{prop:eqlim} is proved for $S_T=0$ (see Proposition 6.5 p. 145). As stressed by the authors, the result is non trivial since the system \eqref{eq:lim} is similar to a three-dimensional Navier-Stokes equation, with a vanishing vertical viscosity. The proof relies on two arguments: first,  a careful analysis of the structure of the quadratic term $\bar Q$ shows that the limit equation is in fact close to a two-dimensional one. Second, the divergence-free property enables one to recover estimates on the vertical derivatives on the third component of the velocity field, and thus to bypass the difficulties caused by the lack of smoothing in the vertical direction.

In fact, the proof of Proposition \ref{prop:eqlim} can easily be adapted from the one of Proposition 6.5 in \cite{CheminDGG}, and is therefore left to the reader. The method remains exactly the same, the only difference being the presence of the source term $S_T$ in the energy estimates. This does not rise any particular difficulty, thanks to the assumptions on~$\sigma$.

\end{remark}

\subsection{Convergence result}~

\begin{theorem}Assume that $\nu=\mathcal O(\e)$, and that $\sqrt{\e\nu}\beta=\mathcal O(1)$.

Let $\sigma \in W^{1,\infty}([0,\infty)_\tau\times [0,\infty)_t, L^\infty( \T^2\times E))$ such that {\bf(H1)-(H2)} are satisfied.

Let $\ug\in L^\infty(E, \mathcal C([0,\infty), L^2)\cap L^2_\text{loc}([0,\infty), H^1))$ be a weak solution of \eqref{eq:depart}, supplemented with the conditions \eqref{CL} and the initial data $\ug_{|t=0}=u_0\in \cH\times H^{0,1}.$
Let $w$ be the solution of \eqref{eq:lim}. Then for all $T>0$,
$$
\ug - \exp\left( -\frac{t}{\e}L \right) w \to 0 
$$
in $L^2 ([0,T]\times E, H^{1,0})\cap L^\infty([0,T], L^2(E\times \U)).$

\label{thm:conv}

\end{theorem}

In the case of the nonresonant torus (see \eqref{hyp:tor_nonres} below), it is likely that the hypothesis $\nu=\mathcal O(\e)$ can be relaxed. Indeed, in this case, the equation on $w$ decouples between a nonlinear equation on the vertical average of $w$ on the one hand, and a linear equation on the vertical modes of $w$ on the other  (see paragraph \ref{ssec:mean_bhv} below, together with Section \ref{sec:mean_bhv}). Moreover, it can be proved that the purely horizontal modes of $w$ decay exponentially in time at a rate $\exp(-t\sqrt{\nu/\e})$, and the rate of decay does not depend on the particular horizontal mode considered. Thus, in this particular case, the regime $\nu\gg\e$ may be investigated, using arguments similar to those developed in \cite{forcage-resonant}.

\begin{remark}
 \label{rem:regul_sigma}
In fact, the above Theorem remains true even when the number of horizontal Fourier modes of $\sigma$ is infinite. In this case, it can be checked that the non-resonance assumptions \textbf{(H1)-(H2)} have to be replaced by the following: there exists $s>4$ such that
\begin{itemize}
 \item[\bf(H1')] For all $\alpha>0$, $T>0$, $\cFa \sigma\in L^\infty( [0,T]\times E, L^1(\R_\lambda, H^s(\T^2)))$, and
$$
\forall T>0,\quad\sup_{\alpha>0}||\cFa \sigma||_{L^\infty( [0,T]\times E, L^1(\R, H^s(\T^2)))}<+\infty.
$$

\item[\bf(H2')] There exist neighbourhoods $V_\pm$ of $\pm 1$, independent of $\alpha>0$,  such that
$$
\forall T>0,\quad\lim_{\alpha\to 0}\sup_{\lambda\in V_{+}\cup V_-}\|\cFa \sigma(\lambda)\|_{L^\infty([0,T]\times E, H^s(\T^2))}=0.
$$

\end{itemize}
The $H^4$ regularity assumption stems from the regularity required in the stopping Lemma A.2 in the Appendix.

Furthermore, the regularity assumptions on the function $\sigma $ become 
$$
\sigma \in L^\infty_{\text{loc}}([0,\infty)_t, L^\infty([0,\infty)_\tau\times E, H^{3/2}(\T^2)),\quad \p_\tau \sigma\in H^1(\T^2, L^\infty([0,\infty)_t\times [0,\infty)_\tau\times E)).
$$

\end{remark}

\begin{remark}\label{rem:H1H2}
Let us now explain the meaning of hypotheses {\bf (H1)-(H2)} for almost periodic functions. Let $k_h\in\Z^2$, and let $\phi\in L^\infty([0,\infty)\times \T^2)$ such that
$$
\phi(\tau,x_h)=e^{ik_h'\cdot x_h}\sum_{\mu\in M }\phi_\mu e^{i\mu\tau},
$$
where $M$ is a countable set. The fact that $\phi$ as only one horizontal Fourier mode is not crucial, but merely helps  focusing on the time spectrum.
Then it can be checked easily that for all $\alpha>0$,
$$
\cFa \phi(\lambda,x_h)=\frac{1}{2\pi}e^{ik_h'\cdot x_h}\sum_{\mu\in M}\phi_\mu\frac{2 \alpha}{\alpha^2 + (\mu-\lambda)^2}.
$$
In particular,  there exists a constant $C>0$ such that 
\begin{eqnarray*}
\| \cFa \phi\|_{L^\infty(\T^2, L^1(\mathbb R_\lambda))}&\leq  &C \sum_{\mu\in M} \left|\phi_\mu  \right|\int_{\R}\frac{2\alpha}{\alpha^2 + (\mu-\lambda)^2}d\lambda\\
&\leq & C\sum_{\mu \in M} \left|\phi_\mu  \right|.
\end{eqnarray*}
Thus hypothesis {\bf (H1)} is satisfied provided $\sum_{\mu\in M}\left|\phi_\mu  \right|<\infty.$

On the other hand, assume that 
\be
\eta:=d(M,\{-1,1\})>0,\label{hyp:freq_phi}
\ee
i.e. that there are no frequencies in a neighbourhood of $\pm 1$. Then if $\lambda \in (-1-\eta/2, -1 + \eta/2) \cup(1-\eta/2, 1 + \eta/2) $, we have
$$
|\lambda-\mu|\geq \frac{\eta}{2}\quad\forall \mu\in M,
$$
and consequently, setting $V^\pm:=\left(\pm1-{\eta}/{2}, \pm 1 + {\eta}/{2}\right)$, we have,
$$
\sup_{\lambda\in V^- \cup V^+}
\left\|\cFa\phi(\lambda)\right\|_{L^\infty(\T^2)}\leq C \frac{\alpha}{\eta}.
$$
Thus hypothesis \eqref{hyp:freq_phi} entails {\bf (H2)}. Additionally, hypothesis \eqref{hyp:freq_phi} cannot be easily relaxed, as shows the following construction: consider the sequence $\mu_n:=1- 1/n, $ and choose a sequence of positive numbers $\phi_n$ such that
$$
\sum_n \phi_n<\infty.
$$
For $\tau\in \R$, set
$$
\phi(\tau):= \sum_n \phi_n e^{i\mu_n \tau}.
$$
Then for all $\alpha>0$, for all $k>0$
$$
\cFa \phi (\mu_k)=\sum_n \phi_n\frac{2\alpha}{\alpha^2 + \left( \frac{1}{n}- \frac{1}{k} \right)^2}\geq \frac{2\phi_k}{\alpha}.
$$
In particular,
$$
\lim_{\alpha\to 0} \cFa \phi (\mu_k)=+\infty
$$
for all $k$, and thus condition {\bf (H2)} is not satisfied.
\end{remark}

\subsection{Average behaviour at the limit}
\label{ssec:mean_bhv}
We have already stressed that the solution $w$ of equation \eqref{eq:lim} is, in general, a random function. Thus one may wonder whether the average behaviour of $w$ at the limit can be characterized. In general, the nonlinearity of equation \eqref{eq:lim} prevents us from deriving an equation, or a system of equations, on the expectation of $w$, which we denote by $\bE[w]$. However, when the torus is non resonant, equation \eqref{eq:lim} decouples, and in this case we are able to exhibit a system of equations satisfied by $\bE[w].$

Let us first recall a few definitions:

\begin{definition}[Non-resonant torus] The torus $\T^3:= \T^2\times [-a,a)$ is said to be {\bf non-resonant} if the following property holds: for all $(k,n)\in \Z^3\setminus\{ 0\}\times\Z^3\setminus\{ 0\},$
 \be
\left( \exists \eta\in\{-1,1\}^{3}, \eta_1\lambda_k + \eta_2\lambda_{n-k} -\eta_3\lambda_n=0 \right)\Rightarrow k_3n_3(n_3-k_3)=0.
\label{hyp:tor_nonres}
\ee
\end{definition}

We refer to \cite{BMN} for a discussion of hypothesis \eqref{hyp:tor_nonres} and its consequences. Let us mention that \eqref{hyp:tor_nonres} holds for almost all values of $(a_1,a_2,a)\in (0,\infty)^3.$
When the torus is non-resonant,  the structure  of the quadratic form $\bar Q$ defined by \eqref{def:Q} is particularly simple, and the system \eqref{eq:lim} can be decoupled into a two-dimensional Navier-Stokes equation on the vertical average of $w$, and a linear equation on the $z$-dependent part (see \cite{CheminDGG}). The advantage of this decomposition in our case is that the vertical average of $\bar S_T(\sigma)$ is deterministic, at least when the group transformation $(\theta_\tau)_{\tau\geq 0}$ acting on $E$ is ergodic (see \cite{sinai}).

\begin{definition}[Ergodic transformation group] Let $(\theta_\tau)_{\tau\in \R}$ be a group of invariant transformations acting on the probability space $(E, \mathcal A, m_0)$. The group is said to be ergodic if for all $A\in \mathcal A$, 
$$
\left( \theta_\tau A\subset A\quad\forall \tau \in \R \right)\Rightarrow m_0(A)=0\text{ or }m_0(A)=1.
$$

\end{definition}

We now state the result on the average behaviour at the limit:
\begin{proposition}
Assume  that the transformation group $(\theta_\tau)_{\tau \in\R}$ is ergodic.

 \noindent Let $u_0\in \cH \cap H^{0,1},$ and let $\sigma \in L^\infty([0,\infty)_t,\times [0,\infty)_\tau\times E\times \T^2)$  such that
the hypotheses of Theorem \ref{thm:conv} are satisfied. Let $w\in L^\infty(E, \mathcal C([0,\infty), \cH\cap H^{0,1})\cap L^2_\text{loc}([0,\infty), H^{1,0})) $ be the unique solution of equation \eqref{eq:lim}.

Let $\bar w=(\bar w_h,0)\in\mathcal C([0,\infty), L^2(\T^2))\cap L^2_\text{loc}([0,\infty), H^1(\T^2))$ be the solution of the 2D-Navier-Stokes equation
$$
\begin{aligned}
&\p_t \bw_h + \bw_h \cdot \nabla_h \bw_h -\Delta_h \bw_h + \frac{1}{a\sqrt{2}}\sqrt{\frac{\nu}{\e}}\bw_h +\nu\beta  \bE\left[ S_{T}(\sigma)\right]_h=\nabla_h \bar p,\\
&\dv_h \bw_h=0,\\
&\bw_{h|t=0}(x_h)=\frac{1}{a}\int_0^{a}u_{0,h}(x_h,z)\:dz.
\end{aligned}
$$
Then the following properties hold:

\begin{enumerate}
 \item As $\e,\nu\to 0$ as in Theorem \ref{thm:conv}, we have
$$
\ug \rightharpoonup \bar w \quad \text{in }L^2([0,T]\times \U \times E).
$$
In particular, the weak limit of $\ug$ is a deterministic function.
\item Assume additionally that the torus $\T^3$ is non resonant. Then
$$
\bE [w]= \bar w + \tilde w,
$$
where $\tilde w$ solves a linear deterministic equation
$$
\begin{aligned}
&\ds\p_t\tilde w +2 \bar Q(\bar w, \tilde w)-\Delta_h\tilde w + \sqrt{\frac{\nu}{\e}} S_{B}(\tilde w) =0,\\
&\tilde w_{|t=0}=u_0- \bw_{|t=0}.
\end{aligned}$$
\end{enumerate}

\label{prop:meanbhv}
\end{proposition}

\subsection{Strategy of proof of Theorem \ref{thm:conv}}

The proof relies on the construction of an approximate solution, obtained as the sum of some interior terms - the largest of which is $\exp(-TL/\e ) w(t)$ - and some boundary layer terms which restore the horizontal boundary conditions violated by the interior terms. We refer to the works by N. Masmoudi \cite{Masmoudi1,MasmoudiCPAM}, N. Masmoudi and E. Grenier \cite{MasmoudiGrenier}, N. Masmoudi and F. Rousset \cite{MasmoudiRousset}, and F. Rousset \cite{Rousset} for an extensive study of boundary layers in rotating fluids, or in incompressible fluids  with vanishing vertical viscosity for \cite{Masmoudi1}. We emphasize that {\it in fine}, all terms will be small in $L^2$ norm, except $\exp(-TL/\e ) w(t)$. 

Following \cite{CheminDGG} (Chapter 7), let us assume that as $\e,\nu\to 0$,
\be
\begin{array}{l}
\ds 
\ug\approx \uint+ \ubl, \\
\ds\pg\approx \frac{1}{\e}p^{\text{int}}+ \frac{1}{\e}p^{\text{BL}} + p^{\text{int},0},
\end{array}
\label{Ansatz}\ee
where 
\begin{eqnarray*}
&& u^{\text{int}}(t,x_h,z)= U\left( t,\frac{t}{\e},x_h,z \right),\  p^{\text{int}}(t,x_h,z)= P\left( t,\frac{t}{\e},x,y,z \right), \\
&&u^{\text{BL}}(t,x_h,z)= u_T\left( t,\frac{t}{\e},x_h,\frac{a-z}{\eta} \right) +  u_B\left( t,\frac{t}{\e},x_h,\frac{z}{\eta} \right),\\
&&p^{\text{BL}}(t,x_h,z)= p_T\left( t,\frac{t}{\e},x_h,\frac{a-z}{\eta} \right)  + p_B\left( t,\frac{t}{\e},x_h,\frac{z}{\eta} \right).
\end{eqnarray*}
Above, $\eta$ is a small parameter that will be chosen later on. The function $u_T(t,\tau,x_h,\zeta)$ is assumed to vanish as $\zeta\to \infty$ (same  for $p_T,p_B,u_B$).

We then plug the Ansatz \eqref{Ansatz} into equation \eqref{eq:depart}, and identify the different powers of $\e$. In general, there is a \textbf{coupling between $\uint$ and $\ubl$}: indeed, we have seen that it is natural to expect that
$$
U(t,\tau)=\exp(-\tau L) w(t),
$$
at first order, and thus $u^{\text{int}}$ does not match the horizontal boundary conditions in general. As a consequence,
 the value of $\uint$ at the boundary has to be taken into account when constructing the boundary layer term $\ubl$. On the other hand, because of the divergence-free constraint, the third component of $\ubl$ does not vanish at the boundary, which means that a small amount of fluid may enter or leave the interior of the domain. This phenomenon is called \textbf{Ekman suction}, and gives rise to a source term (called the Ekman pumping term) in the equation satisfied by $\uint$. This leads to some sort of ``loop'' construction, in which the boundary layer and interior terms are constructed one after the other.

The first step of this  construction lies in the definition of boundary layer terms. In the periodic case, this is well-understood (see \cite{CheminDGG,MasmoudiCPAM,forcage-resonant}); thus the main contribution of this article in this regard lies in the definition of boundary layers in the random stationary case. Hence, Section \ref{sec:BL} is entirely devoted to that topic. Section \ref{sec:interior} is concerned with the definition of first and second order interior terms; in particular, we derive in paragraph \ref{ssec:envelope} the limit equation for the system \eqref{eq:depart}. In Section \ref{sec:CVproof}, we prove the convergence result, after defining some additional corrector terms. At last, we prove Proposition \ref{prop:meanbhv} in Section \ref{sec:mean_bhv}.

\section{Construction of random  boundary layer terms}
\label{sec:BL}

The goal of this section is to construct approximate solutions of equation \eqref{eq:depart}, which satisfy the horizontal boundary condition at $z=a$, and which are localized in the vicinity of the surface. Such a construction has already been achieved in the case when the function $\sigma$ is quasi-periodic with respect to the fast time variable (see \cite{MasmoudiCPAM} in the non-resonant case, and \cite{forcage-resonant} in the resonant case). Thus our goal is to extend this construction to a random forcing. The main result of this section is the following:

\begin{lemma}Assume that $\beta \sqrt{\e\nu}=\cO(1)$.
Let $\sigma\in W^{1,\infty}([0,\infty)_t\times [0,\infty)_\tau, L^\infty(\T^2\times E))$ be such that \textbf{(H1)-(H2)} are satisfied. Then for all $\delta>0$, there exists a function $u^{\text{BL},\delta}_T \in L^\infty([0,\infty)_t\times \U\times E)$ which satisfies the system
\begin{eqnarray*}
&&\p_t u^{\text{BL},\delta}_T + \frac{1}{\e} e_3\wedge u^{\text{BL},\delta}_T-\nu\p_z^2 u^{\text{BL},\delta}_T- \Delta_h u^{\text{BL},\delta}_T= \cO\left( \left(1+\frac{\delta}{\e}\right)(\e\nu)^{1/4}\|\sigma\| \right)_{L^\infty([0,\infty)\times  E, L^2(\U))},\\
&&\p_z u^{\text{BL},\delta}_{T, h|z=a}=\beta \sigma,\\
&&\dv u^{\text{BL},\delta}_T=0,
\end{eqnarray*}
and such that
$$
\sup_{\delta>0}\| u^{\text{BL},\delta}_T\|_\infty<\infty,\quad \sup_{\delta>0} (\e\nu)^{-1/4}\| u^{\text{BL},\delta}_T\|_{L^\infty([0,\infty)\times E, L^2(\U))}<+\infty.
$$

Moreover, $u^{\text{BL},\delta}_{T|z=0}$ is exponentially small.
\label{lem:def-uTBL}

\end{lemma}

The above Lemma entails in particular that for all $\delta>0$, $ u^{\text{BL},\delta}_T$ is an approximate solution of \eqref{eq:depart}, which satisfies the appropriate horizontal boundary condition at $z=a$. The Lemma is proved in the two next paragraphs: we first explain how $ u^{\text{BL},\delta}_T$ is defined, and then we derive the $L^2$ and $L^\infty$ estimates.

\subsection{Construction of the boundary term at the surface}

As explained in the Introduction, the idea is to consider an Ansatz of the form $$u^{\text{BL},\delta}_T(t,x_h,z,\om)=
u_{T}\left( t,\frac{t}{\e}, x_h, \frac{a-z}{\eta};\om \right),$$
where $\eta$ is a small parameter (whose size has to be determined) and 
$$
\lim_{\zeta\to\infty}u_T(t,\tau,x_h,\zeta;\om)=0\quad \forall\ t,\tau,x_h,\om.
$$
Hence we expect $\ubl_h$ to be of order ${\eta}\beta||\sigma||_{\infty}$ in $L^\infty$. Moreover, 
the divergence-free condition entails that the third component of $u_T$ is given by 
$$
 u_{T,3}\left(\zeta\right)=-\eta\int_{\zeta}^\infty \dv_h u_{T,h}(\zeta')d\zeta';$$
thus $ u_{T,3}=\mathcal O(\beta \eta^2||\sigma||_{W^{1,\infty}})$.
At last, in order to be consistent with \eqref{Ansatz}, we assume that the pressure inside the boundary layer is given by
$$
p(t,x_h,z,\om)\underset{z\sim a}{\approx}  \frac{1}{\e}p_T \left( \frac{a-z}{\eta} \right)
$$
where $p_T=\mathcal O(\beta{\eta}||\sigma||_{\infty})$. 
Then the pressure term in the third component of \eqref{eq:depart} is of order $\beta||\sigma||_{\infty}/\e$, whereas the lowest order term in the left-hand side is of order $\eta^2\beta||\sigma||_{W^{1,\infty}}/\e$. Thus, since $\eta$ is small, we infer
$$
\p_{\zeta}p_T=0,
$$
and since $p_T$ vanishes at infinity, we have $p_T=0$: at first order, the pressure does not vary in the boundary layer. Thus, we now focus on the horizontal component of $u_T$, which is a solution of
\begin{eqnarray}
 &&\p_\tau \begin{pmatrix}
            u_{{T,1}}\\
u_{{T,2}}
           \end{pmatrix}
 - \frac{\nu\e}{\eta^2}\p_{\zeta}^2\begin{pmatrix}
            u_{{T,1}}\\
u_{{T,2}}
           \end{pmatrix} + \begin{pmatrix}
            -u_{{T,2}}\\
u_{{T,1}}
           \end{pmatrix} =0,\label{eq:ubl0}\\
&&\p_{\zeta}u_{T, h|\zeta=0}=-\eta\beta \sigma(\tau,x,y,\om),\label{CL1:ubl0}\\
&&u_{T, h|\zeta=+\infty}=0.\label{CL2:ubl0}
\end{eqnarray}
We now choose $\eta$ so that all the terms in \eqref{eq:ubl0} are of the same order, that is,
$$
\eta=\sqrt{\nu\e}.
$$
Moreover, since $ \sigma$ is a stationary function of time, it seems natural to look for stationary solutions of \eqref{eq:ubl0}, and thus for fundamental solutions $\varphi_1,\varphi_2$ of \eqref{eq:ubl0} in the following sense: $\varphi_i$ ($i=1,2$) is a solution of \eqref{eq:ubl0} in the sense of distributions and satisfies \eqref{CL2:ubl0}, and
$$
\p_\zeta \varphi_{1|\zeta=0}= \delta_0(\tau)\begin{pmatrix}
                                         1\\0
                                        \end{pmatrix},\qquad \p_\zeta \varphi_{2|\zeta=0}= \delta_0(t)\begin{pmatrix}
                                         0\\1
                                        \end{pmatrix}
$$
where $\delta_0$ denotes the Dirac mass at $\tau=0$. 
If we can construct $\varphi_1$ and $\varphi_2$ satisfying the above conditions, then a good candidate for $u_T$ is
$$
u_{T, h}(t,\tau,x_h,\zeta;\om)=-\sqrt{\nu\e}\beta\sum_{j\in\{1,2\}}\int_0^\infty \sigma_j(t,\tau-s,x_h;\om)\varphi_j(s,\zeta) ds.
$$
Hence we now define $\varphi_{1}, \varphi_{2}$. Since the fundamental solution of the heat equation is known, let us make the following change of unknow function (see \cite{MasmoudiCPAM}):
$$
H^{\pm}_j=\p_{\zeta}\left[ e^{\pm i\tau}\begin{pmatrix}
									\varphi_{j,1}\pm i \varphi_{j,2}\\
									\varphi_{j,2}\mp i \varphi_{j,1}
                                                                 \end{pmatrix}
 \right],\quad j=1,2.
$$
Then, setting $e_1^\pm:=(1,\mp i)$, $e_2^\pm:=(\pm i,1)$, we infer that $H_j^\pm=G e_j^\pm$, where $G$ satisfies
\be
\left\{ 
\begin{array}{l}
\ds\p_\tau G - \p_{\zeta}^2 G=0,\quad\tau>0,\zeta>0,\\
G_{|\zeta=0}(\tau)=\delta_0(\tau),\\
G_{|\zeta=+\infty}=0.
\end{array}
 \right.\label{eq:H}
\ee
The boundary condition at $\zeta=0$ should be understood as follows: for all $\varphi\in\mathcal C_b(\R),$ for all $\tau> 0$
$$
\lim_{\zeta\to 0^+}\left[\int_0^\infty \varphi(\tau-s) G(s,\zeta)ds\right]=\varphi(\tau).
$$
It can be checked (see Chapter 4, section 1 in \cite{ladyparabolic}) that 
$$
G(\tau,\zeta):=  \frac{\zeta}{\sqrt{4\pi}\tau^{3/2}} \exp\left({- \frac{\zeta^2}{4\tau} }\right)\quad\text{for } \tau>0,\ \zeta> 0,
$$
is a solution of \eqref{eq:H}, which leads to
\begin{eqnarray*}
\p_\zeta \varphi_j(\tau,\zeta)&:=&\frac{1}{2}\left[e^{-i\tau}  H_j^+ (\tau,\zeta) + e^{+i\tau} H_j^-(\tau,\zeta) \right]\\
&=&\frac{1}{2}G(\tau,\zeta)\left[e^{-i\tau} e_j^+  +e^{+i\tau} e_j^- \right].
\end{eqnarray*}
Unfortunately, when we integrate this formula with respect to $\zeta$ in order to obtain an explicit expression for $u_{{T,h}}$, the convolution kernel thus obtained is
$$
\varphi_j(\tau,\zeta)=-\frac{1}{\sqrt{4\pi \tau}}\exp\left({- \frac{\zeta^2}{4\tau} }\right)\left[e^{-i\tau} e_j^+  +e^{+i\tau} e_j^- \right],
$$
and is not integrable near $\tau=+\infty$. Hence, in the spirit of \cite{MasmoudiCPAM}, we consider an approximate corrector in the boundary layer: for $\delta>0$, we set
$$
G_{\delta}(\tau,\zeta)= \frac{\zeta}{\sqrt{4\pi}\tau^{3/2}} \exp\left({- \frac{\zeta^2}{4\tau} }-\delta \tau\right).
$$
Then the corresponding corrector is given by
\begin{eqnarray}
u_{T,h}^{\delta}(t,\tau,x_h, \zeta,\om)
&=&-{\beta\sqrt{\e\nu}}\sum_{j\in\{1,2\}}\int_0^\infty\:\: \varphi_j(s,\zeta)\exp(-\delta s)\sigma_j(t,\tau-s,x_h;\om)\:ds\label{expr:ubl0}\\
&=&\frac{\beta\sqrt{\nu\e}}{\sqrt{4\pi}}\sum_{\pm}\!\! \int_0^\infty \frac{1}{\sqrt s} \exp\left(-\frac{\zeta^2}{4 s}\right)(\sigma \pm i \sigma^\bot)(t, \tau-s,x_h,\om)e^{-\delta s\pm i s}ds.\nonumber
\end{eqnarray}
The approximate corrector $u_{T}^{\delta}$ satisfies the exact boundary conditions at $\zeta=0$, and equation \eqref{eq:ubl0} up to an error term of order $\delta$
$$
\p_\tau u_{T,h}^{\delta} -\p_\zeta^2 u_{T,h}^{\delta} + \left( u_{T,h}^{\delta}\right)^\bot + \delta  u_{T,h}^{\delta}=0.
$$

The third component of $u_T^\delta$ is then given by
$$
u_{T,3}^{\delta}(\zeta)= -\sqrt{{\nu}{\e} }\int_\zeta^\infty \dv_h u_{T,h}^{\delta},
$$
which yields
$$
u_{T,3}^{\delta}(\cdot,\tau,\cdot, \zeta,\om)
=\frac{{\nu\e}\beta}{\sqrt{4\pi}}\sum_{\pm}\int_0^\infty  \varphi\left(\frac{\zeta}{\sqrt s}\right)(\dv_h \sigma \mp i \rot_h \sigma)(\cdot, \tau-s,\cdot,\om)e^{-\delta s\pm i s}ds,$$
where $\varphi $ is defined by $\varphi'(\zeta)= \exp\left(-\frac{\zeta^2}{4}\right)$, $\varphi(+\infty)=0$.

In horizontal Fourier variables, we have
\be
u_{T,3}^{\delta}(t,\tau,x_h, \zeta,\om)=\frac{{\nu}{\e}\beta }{\sqrt{4\pi}a_1a_2}\sum_{k_h\in\Z^2}\sum_{\pm}e^{i k'_h\cdot x_h}\int_0^\infty \varphi\left(\frac{\zeta}{\sqrt s}\right) 
\hat \sigma^\pm(t,\tau-s,k_h,\om)e^{-\delta s\pm i s}ds
\label{expr:u1T3}
\ee
where
$$\hat \sigma^\pm(k_h)= ik_h' \cdot \hat \sigma(k_h) \pm (k_h')^\bot \cdot \hat \sigma(k_h).$$

Now, set
$$
 u^{\text{BL},\delta}_T(t,x_h,z;\om):= u_T^\delta \left( t,\frac{t}{\e}, x_h,\frac{a-z}{\sqrt{\e\nu}};\om \right).
$$
It can be readily checked that
$$
\p_t  u^{\text{BL},\delta}_T + \frac{1}{\e}e_3\wedge  u^{\text{BL},\delta}_T - \nu \p_z^2  u^{\text{BL},\delta}_T - \Delta_h u^{\text{BL},\delta}_T = \begin{pmatrix}
															(\p_t -\Delta_h + \delta)u_{T,h}^\delta\left( t,\frac{t}{\e}, x_h,\frac{a-z}{\sqrt{\e\nu}};\om \right)\\
															\left(\p_t+ \frac{1}{\e}\p_\tau -\Delta_h - \frac{1}{\e}\p_\zeta^2 + \delta\right) u_{T,3}^\delta \left( t,\frac{t}{\e}, x_h,\frac{a-z}{\sqrt{\e\nu}};\om \right)
                                                                                                                   \end{pmatrix}.
$$

There remains to evaluate $u_T^\delta$ in $L^\infty$ and $L^2$.

\subsection{Continuity estimates}

This paragraph is devoted to the proof of the following Proposition:
\begin{proposition}
Assume that $\sigma\in L^\infty(E\times[0,\infty)\times \T^2, \mathcal C_b(\R_\tau))$, and that $\sigma$ satisfies {\bf (H1)-(H2)}. Then for all $T>0$, there exists a constant $C_T>0$, such that for all $\delta,\nu,\e,\beta>0$, 
\begin{eqnarray}
\left|\left| u_T^\delta ,\ \p_\zeta u_T^\delta,\ \zeta\p_\zeta u_T^\delta\right|  \right|_{L^\infty([0,T]\times\R_\tau\times\T^2\times[0,\infty)_\zeta\times E)} &\leq& C_T{\sqrt{\e\nu}}{\beta},\\
\left|\left| u_T^\delta,\ \p_\zeta u_T^\delta,\ \zeta\p_\zeta u_T^\delta \right|  \right|_{L^\infty([0,T]\times\R_\tau\times E, L^2([0,\infty)_\zeta\times\T^2))}&\leq& C_T{\sqrt{\e\nu}}{\beta}.
\end{eqnarray}

\label{prop:estuT0}
\end{proposition}

\begin{remark}
With the assumptions of Theorem \ref{thm:conv}, the same bounds also hold for all the derivatives of $u_{T,\delta}$ with respect to the macroscopic time variable $t$ and the horizontal space variable $x_h$.
\end{remark}

\begin{proof}
We focus on the horizontal component of $u_T^\delta$; the vertical one is treated with similar arguments. Recall that $u_{T,h}^\delta$ is given by \eqref{expr:ubl0}; in order to simplify we set
$\sigma^\pm:=\sigma\pm i \sigma^\bot$.

First, we write 
\begin{eqnarray}
u_{T,h}^\delta (\cdot,\tau,\cdot, \zeta,\cdot)&=&\frac{\sqrt{\nu\e}\beta}{\sqrt{4\pi}}\sum_\pm\int_0^\infty \frac{1}{\sqrt s} \exp\left(-\frac{\zeta^2}{4 s}-\delta s\right)
\sigma^\pm (\cdot, \tau-s,\cdot)e^{\pm is}\:ds\nonumber\\
&=&\frac{\sqrt{\nu\e}\beta}{\sqrt{4\pi}}\sum_\pm\int_0^\infty \frac{1}{\sqrt s} \exp\left(-\frac{\zeta^2}{4 s}\right)
\sigma_\alpha^\pm (\cdot, \tau-s,\cdot)e^{(-\delta\pm i)s}\:ds\label{c_alpha}\\
&+& \frac{\sqrt{\nu\e}\beta}{\sqrt{4\pi}}\sum_\pm\int_0^\infty \frac{1}{\sqrt s} \exp\left(-\frac{\zeta^2}{4 s}\right)
\left(\sigma^\pm-\sigma_\alpha^\pm\right)(\cdot, \tau-s,\cdot)e^{(-\delta\pm i)s}\:ds.\label{c-c_alpha}
\end{eqnarray}
The  term \eqref{c-c_alpha} can easily be evaluated thanks to Lemma A.1 in the Appendix; notice that since the convergence given in Lemma A.1 is not uniform with respect to $\tau\in[0,\infty)$, we cannot derive an estimate in $L^\infty([0,\infty)_\tau)$ at this stage. Hence we keep the variable $\tau$ for the time being; there exists a constant $C>0$ such that for all $\tau\geq 0,$ $R>0$, 
\begin{eqnarray}
 &&\left\|\int_0^\infty \frac{1}{\sqrt s} \exp\left(-\frac{\zeta^2}{4 s}-\delta s\right)
\left(\sigma^\pm-\sigma_\alpha^\pm\right)(\cdot, \tau-s,\cdot,\om)e^{\pm is}\:ds\right\|_{L^\infty([0,T]\times E, L^\infty(\T^2))}\nonumber\\
&\leq & C||\sigma-\sigma_\alpha||_{L^\infty([0,T]\times E\times \left[ \tau-R,\tau \right]\times \T^2)}\int_0^R \frac{1}{\sqrt s} \exp\left(-\frac{\zeta^2}{4 s}-\delta s\right)\:ds\label{est_diff_sigma1}\\
&+&C\left\| \sigma \right\|_{L^\infty([0,T]\times\R_\tau\times \T^2\times E)}\int_R^\infty \frac{1}{\sqrt s} \exp\left(-\frac{\zeta^2}{4 s}-\delta s\right)\:ds\nonumber\\
&\leq & \frac{C}{\delta}||\sigma-\sigma_\alpha||_{L^\infty([0,T]\times E\times \left[ \tau-R,\tau \right]\times\T^2)}\nonumber\\
&+& C\left\| \sigma \right\|_{L^\infty([0,T]\times\R_\tau\times \T^2\times E)}\frac{\exp(-\delta R)}\delta.\nonumber
\end{eqnarray}
Choosing $R=\delta^{-2},$ we deduce that
\begin{multline*}
 \left\|\int_0^\infty \frac{1}{\sqrt s} \exp\left(-\frac{\zeta^2}{4 s}-\delta s\right)
\left(\sigma^\pm-\sigma_\alpha^\pm\right)(\cdot, \tau-s,\cdot,\om)e^{\pm is}\:ds\right\|_{L^\infty([0,T]\times E, L^\infty(\T^2))}\\
\leq \frac{C}{\delta}||\sigma-\sigma_\alpha||_{L^\infty([0,T]\times E\times \left[ \tau-\frac{1}{\delta^2},\tau \right]\times\T^2)} + C\frac{\exp\left(-\frac{1}{\delta}\right)}\delta.
\end{multline*}

As for the  term \eqref{c_alpha}, recalling the definition of $\sigma_\alpha$, we have
\begin{eqnarray}
 &&\int_0^\infty \frac{1}{\sqrt s} \exp\left(-\frac{\zeta^2}{4 s}-\delta s\right)
\cFa\sigma^\pm(\cdot, \tau-s,\cdot,\om)e^{\pm is}\:ds\\
&=&\int_0^\infty\int_\R \frac{1}{\sqrt s} e^{-\alpha|\lambda|}\exp\left(-\frac{\zeta^2}{4 s}-\delta s\right)
\cFa\sigma^\pm(\cdot,\lambda,\cdot,\om)e^{i\lambda(\tau-s)}e^{\pm is}\:d\lambda\:ds.\label{eq:alpha_1}
\end{eqnarray}
We first evaluate
$$
\int_0^\infty\frac{1}{\sqrt s}\exp\left(-\frac{\zeta^2}{4 s}\right)e^{-(\delta + i(\lambda\pm 1))s}\:ds.$$
We split the integral into two parts, one going from $s=0$ to $s=1$, and the other from $s=1$ to $s=\infty$. It is obvious that for all $\zeta>0,\delta>0,\lambda\in\R$,
\be
\left|\int_0^1\frac{1}{\sqrt s}\exp\left(-\frac{\zeta^2}{4 s}\right)e^{-(\delta + i(\lambda\pm 1))s}\:ds\right|\leq \int_0^1\frac{1}{\sqrt s}\exp\left(-\frac{\zeta^2}{4 s}\right)\:ds\leq \frac{1}{2}.\label{int01_2}
\ee
Integrating by parts the second integral, we obtain
\begin{eqnarray}
 &&\int_1^\infty\frac{1}{\sqrt s}\exp\left(-\frac{\zeta^2}{4 s}\right)e^{-(\delta + i(\lambda\pm 1))s}\:ds\nonumber\\
&=&\frac{1}{\delta + i(\lambda\pm 1)}\exp\left(-\frac{\zeta^2}{4 }\right)\nonumber\\
&&-\frac{1}{2(\delta + i(\lambda\pm1))}\int_1^\infty\frac{1}{s^\frac{3}{2}}\left[ 1- \frac{\zeta^2}{2s} \right]\exp\left(-\frac{\zeta^2}{4 s}\right)e^{-(\delta + i(\lambda\pm1))s}\:ds.\label{int1infty}
\end{eqnarray}
We are now ready to derive the $L^\infty$ estimate; the function
$$
x\mapsto \left(1-\frac{x^2}2\right) e^{-\frac{x^2}4}
$$
is bounded on $\R$. Hence, gathering \eqref{int01_2} and \eqref{int1infty}, we deduce that there exists a constant $C$ such that for all $\zeta>0,\delta>0,\lambda\in\R$,
$$\left| \int_0^\infty\frac{1}{\sqrt s}\exp\left(-\frac{\zeta^2}{4 s}\right)e^{-(\delta + i(\lambda\pm 1))s}\:ds \right|
\leq  C\left[ 1 +\frac{1}{|\delta + i(\lambda\pm 1)|}   \right].
$$
Inserting this inequality in \eqref{eq:alpha_1}, we obtain
\begin{eqnarray*}
 &&\left\|\int_0^\infty \frac{1}{\sqrt s} \exp\left(-\frac{\zeta^2}{4 s}-\delta s\right)
\cFa\sigma^\pm(\cdot, \tau-s,\cdot,\om)e^{\mp is}\:ds\right\|_{L^\infty(\T^2)}\\
&\leq& C \int_\R e^{-\alpha|\lambda|} \left[ 1 +\frac{1}{|\delta + i(\lambda\mp 1)|}   \right]\left\|\hat \cFa\sigma^\pm(\cdot,\lambda,\cdot,\om)\right\|_{L^\infty(\T^2)}\:d\lambda\\
&\leq& C\left[ \sup_\alpha ||\hat{\sigma}_{+,\alpha}||_{L^\infty(E\times \T^2, L^1(\R))} + \int_{V_\pm}  \frac{1}{|\delta + i(\lambda\mp 1)|}   \left\|\cFa\sigma(\cdot,\lambda,\cdot,\om)\right\|_{L^\infty(\T^2)}\:d\lambda\right] \\&&+C \int_{\R\setminus V_\pm} \left\|\cFa\sigma(\cdot,\lambda,\cdot,\om)\right\|_{L^\infty(\T^2)}\:d\lambda\\
&\leq & C\left[ \sup_\alpha ||\cFa\sigma||_{L^\infty(E\times \T^2, L^1(\R))}  + \sup_{\lambda\in V_\pm}\left\| \cFa\sigma(\lambda) \right\|_{L^\infty(\T^2)}\ln(\delta)\right].
\end{eqnarray*}
Above, we have used the following facts: there exists a constant $c_1>0$ such that
$$
\left| \delta + i(\lambda \mp 1) \right|\geq |\lambda \mp 1|\geq c_1\quad \forall \lambda\in\R\setminus V_\pm,
$$
and there exists another constant $c_2>0$ such that
$$
\int_{V_\pm} \frac{1}{|\delta + i(\lambda\mp 1)|} \leq \int_{\pm 1-c_2}^{\pm 1 + c_2}\frac{1}{\sqrt{\delta^2 + (1+\lambda)^2}}d\lambda\leq C\ln(\delta).
$$
We deduce that for all $\alpha>0$, for all $\delta>0$, $\tau\geq 0,$
\begin{eqnarray*}
 &&\left|\left|u_{T,h}^\delta(\tau)\right|\right|_{L^\infty([0,T]\times \T^2\times [0,\infty)\zeta\times E)}\\&\leq  &C\sqrt{\e\nu}\beta\left[ 1 +\frac{\exp\left(-\frac{1}{\delta}  \right)}{\delta}\right]\\
&+& C\sqrt{\e\nu}\beta\left[ \frac{1}{\delta}||\sigma-\sigma_\alpha||_{L^\infty([0,T]\times \left[ \tau-\delta^{-1},\tau \right]\times E\times\T^2)} +\sup_{\lambda\in V_+\cup V_-}\left\| \cFa \sigma (\lambda) \right\|\ln(\delta) \right].
\end{eqnarray*}
Taking the infimum with respect to $\alpha$ of the right-hand side, with $\delta>0$ fixed, we deduce that
$$
\sup_{\delta>0}\left|\left|u_{T,h}^\delta\right|\right|_{L^\infty([0,T]\times [0,\infty)\tau \times \T^2\times [0,\infty)\zeta\times E)}\leq C\sqrt{\e\nu}\beta.
$$
\vskip1mm

We now turn to the derivation of the $L^2$ estimate, which is similar to the above computations. The main difference lies in the fact that we need to integrate by parts \eqref{int1infty} yet another time, which yields
\begin{eqnarray*}
 &&\int_1^\infty\frac{1}{\sqrt s}\exp\left(-\frac{\zeta^2}{4 s}\right)e^{-(\delta + i(\lambda\pm 1))s}\:ds\nonumber\\
&=&\frac{1}{\delta + i(\lambda\pm 1)}\exp\left(-\frac{\zeta^2}{4 }\right)-\frac{1}{2(\delta + i(\lambda \pm 1))^2}\left[ 1- \frac{\zeta^2}{2}\right]\exp\left(-\frac{\zeta^2}{4 }\right)\\
&&-\frac{1}{2(\delta + i(\lambda\pm 1))^2}\int_1^\infty\frac{1}{s^\frac{5}{2}}\phi\left( \frac{\zeta}{\sqrt{s}} \right)e^{-(\delta + i(\lambda\pm 1))s}\:ds,
\end{eqnarray*}
where
$$
\phi(x)=-\left( \frac{x^4}{8} - \frac{3 x^2}{2} + \frac{3}{2}\right) \exp\left(-\frac{x^2}{4}\right).
$$
Consequently, remembering \eqref{int01_2}, we have
\begin{eqnarray*}
&&\left| \int_0^\infty\frac{1}{\sqrt s}\exp\left(-\frac{\zeta^2}{4 s}\right)e^{-(\delta + i(\lambda\pm 1))s}\:ds \right|
\\&\leq & \int_0^1\frac{1}{\sqrt s}\exp\left(-\frac{\zeta^2}{4 s}\right)\:ds + \frac{1}{|\delta + i(\lambda\pm1)|}\exp\left(-\frac{\zeta^2}{4 }\right)\\
&&+\frac{1}{2|\delta + i(\lambda\pm1)|^2}\left| 1- \frac{\zeta^2}{2}\right|\exp\left(-\frac{\zeta^2}{4 }\right)\\
&&+\frac{1}{2|\delta + i(\lambda\pm1)|^2}\int_1^\infty\frac{1}{s^\frac{5}{2}}\left|\phi\left( \frac{\zeta}{\sqrt{s}} \right)\right|\:ds.
\end{eqnarray*}
Plugging this estimate into \eqref{eq:alpha_1} and using {\bf (H1)-(H2)}, we infer that for all $\zeta>0$, for all $s>0$,
\begin{eqnarray*}
&&\left\| \int_0^\infty \frac{1}{\sqrt s} \exp\left(-\frac{\zeta^2}{4 s}-\delta s\right)
\sigma_\alpha^\pm(\cdot, \tau-s,\cdot,\om)e^{\pm is}\:ds \right\|_{L^\infty(\T^2)}\\
&\leq & C\left[\int_0^1\frac{1}{\sqrt s}\exp\left(-\frac{\zeta^2}{4 s}\right)\:ds\right]   \\
&&+ C\exp\left(-\frac{\zeta^2}{4 }\right)\left( 1 + \sup_{\lambda\in V_\pm}\left\|\cFa\sigma(\lambda) \right\|_{L^\infty([0,T]\times E, L^\infty(\T^2)))}\ln(\delta) \right)\\
&&+C\left| 1- \frac{\zeta^2}{2}\right|\exp\left(-\frac{\zeta^2}{4 }\right)\left( 1 + \sup_{\lambda\in V_\pm}\left\| \cFa\sigma(\lambda) \right\|_{L^\infty([0,T]\times E, L^\infty(\T^2)))}\frac{1}{\delta} \right)\\
&&+C\left[\int_1^\infty\frac{1}{s^\frac{5}{2}}\left|\phi\left( \frac{\zeta^2}{s} \right)\right|\:ds\right]\left( 1 + \sup_{\lambda\in V_\pm}\left\| \cFa\sigma(\lambda) \right\|_{L^\infty([0,T]\times E, L^\infty(\T^2)))}\frac{1}{\delta} \right).
\end{eqnarray*}
Here, we have used the inequality
$$
\int_{V_\mp}\frac{d\lambda}{|\delta + i(\lambda\pm 1)|^2}\leq \int_{\mp 1- c_2}^{\mp 1 + c_2}\frac{d\lambda}{\delta^2 + (\lambda\pm 1)^2}\leq \frac{C}{\delta}.
$$
There only remains to prove that each term of the right-hand side has a finite $L^2$ norm. First, thanks to Jensen's inequality, we have
\begin{eqnarray*}
\int_0^\infty \left( \int_0^1\frac{2}{\sqrt s}\exp\left(-\frac{\zeta^2}{4 s}\right)\:ds \right)^2 d\zeta&\leq & \int_0^\infty \int_0^1\frac{2}{\sqrt s}\exp\left(-\frac{\zeta^2}{2 s}\right)\:ds  d\zeta\\
&\leq & 2\int_0^1\: ds\int_0^\infty e^{-\frac{x^2}{2}}\:dx<\infty.
\end{eqnarray*}
Similarly, 
\begin{eqnarray*}
\int_0^\infty \left( \int_1^\infty\frac{1}{s^\frac{5}{2}}\left|\phi\left( \frac{\zeta}{\sqrt{s}} \right)\right|\:ds \right)^2 d\zeta&\leq & C\int_0^\infty \int_1^\infty\frac{1}{s^\frac{5}{2}}\left|\phi\left(  \frac{\zeta}{\sqrt{s}}\right)\right|^2\:ds  d\zeta\\
&\leq & C\left( \int_1^\infty\frac{1}{s^2}\ds \right)\left(\int_0^\infty\left|\phi\left( x \right)\right|^2\:dx\right)<\infty.
\end{eqnarray*}
We also have to evaluate the $L^2$ norm of the integral in \eqref{est_diff_sigma1}; we have
\begin{eqnarray*}
 &&\int_0^\infty\left[\int_0^\infty \frac{1}{\sqrt s} \exp\left(-\frac{\zeta^2}{4 s}-\delta s\right)\:ds\right]^2d\zeta\\
&\underset{\substack{x=\sqrt{\delta}\zeta,\\u=\delta s}}{\leq} & \frac{1}{\delta^{\frac{3}{2}}}\int_0^\infty\left[\int_0^\infty \frac{1}{\sqrt u} \exp\left(-\frac{x^2}{4 u}-u\right)\:du\right]^2dx\\
&\leq & \frac{1}{\delta^{\frac{3}{2}}}\int_0^\infty\int_0^\infty \frac{1}{ u} \exp\left(-\frac{x^2}{2 u}-u\right)\:du\:dx\\
&\leq &  \frac{1}{\delta^{\frac{3}{2}}}\int_0^\infty\int_0^\infty \frac{1}{ \sqrt u} \exp\left(-\frac{x^2}{2 }-u\right)\:du\:dx\\
&\leq &\frac{C}{\delta^{\frac{3}{2}}}.
\end{eqnarray*}
Gathering all the terms, we obtain, for all $\alpha, \delta>0$, for all $\tau>0,$
\begin{eqnarray*}
&&\left\|u_{T,h}^\delta(\tau)\right\|_{L^\infty([0,T]\times E, L^2([0,\infty)_\zeta, L^\infty(\T^2)))}^2\\&\leq & C\beta^2{\e\nu}\frac{||\sigma-\sigma_\alpha||_{L^\infty([0,T]\times [\tau-\delta^{-1},\tau]\times E, L^\infty(\T^2))}}{\delta^{\frac{3}{2}}} \\&+& C\beta^2{\e\nu}\left(\frac{\exp\left( -\frac{1}{\delta} \right)}{\delta^{\frac{3}{2}}}+ \sup_{\lambda\in V_-}\left| \cFa \sigma^+(\lambda) \right|\left(\frac{1}{\delta}+ \ln(\delta) \right)\right). 
\end{eqnarray*}
Taking the infimum of the above inequality with respect to $\alpha$, we infer the $L^2$ estimate on $u_{T,h}^\delta.$ The estimates on $u_{T,3}^\delta$, $\p_\zeta u_T^\delta$, and $\zeta\p_\zeta u_T^\delta$ are derived in a similar fashion.
\end{proof}

\begin{remark}
Stationary boundary layer terms relative to Dirichlet boundary conditions can also be defined: consider for instance the boundary condition
$$
u_{h|z=0}(t,x_h)=c_{B,h}\left( t,\frac{t}{\e}, x_h;\om \right).
$$
The construction is the same as for Neumann boundary conditions, and is in fact more simple because we need not integrate with respect to the variable $\zeta$. Thus, with the same notations as above, the boundary layer term at the bottom is given by
$$u_{{B,h}}^\delta(t,\tau,x_h, \zeta,\om)
=\frac{1} 2\sum_{j\in\{1,2\}}\int_{0}^\infty G_\delta(s,\zeta)\left[e^{-is} e_j^+  +e^{+is} e_j^- \right]c_{B,j}(t,\tau-s,x,y;\om)\:ds,$$
and
$$
 u_{{B,3}}^\delta(t,\tau,x_h, \zeta,\om)
=\frac{{\nu}{\e} }{\sqrt{4\pi}}\sum_{k_h\in\Z^2}\sum_\pm e^{i k'_h\cdot x_h}\int_0^\infty \frac{1}{\sqrt s} \exp\left(-\frac{\zeta^2}{4 s}\right)
\hat c_{B,h}^\pm(\cdot,\tau-s,k_h,\om)e^{-\delta s\pm i s}ds.
$$
\label{rem:BL-Dir-stat}

\end{remark}

\subsection{Previous results in the quasi-periodic case}

For the reader's convenience, we have gathered here previous results appearing in \cite{CheminDGG,MasmoudiCPAM}, in which the authors compute the boundary layer term at the bottom of the fluid.
We recall that it can be expected that the solution in the interior of the domain  behaves like some function $\exp (-tL/\e) w(t)$, with $w\in L^\infty(E, \mathcal C([0,\infty), \cH))$. In general, the horizontal component of such a function does not vanish at $z=0$, and thus a boundary layer has to be added in order to restore the Dirichlet boundary condition. Consequently, we seek for a boundary layer term $u^{\text{BL}}_B$ which is an approximate solution of equation \eqref{eq:depart} and which satisfies
\be u^{\text{BL}}_{B,h|z=0}(t,x_h)=\sum_{k\in\Z^3,k\neq 0}\hat c_{B,h} (t,k) e^{ik'_h\cdot x_h} e^{-i\lambda_kt/\e};
\label{CL:uB0_0}\ee
for the boundary layer term at the first order, the coefficient $\hat c_{B,h} (t,k)$ will be given by the formula
$$
\hat c_{B,h} (t,k):=-\mean{N_k, w(t)} \begin{pmatrix}
                                      n_1(k)\\n_2(k)
                                     \end{pmatrix}.
$$
However, we will also use this construction for the lower order boundary layer terms, and thus we keep an arbitrary value for $\hat c_{B,h} (t,k)$ for the time being.

As before, we assume that
$$
u^{\text{BL}}_B(t,x_h,z)= u_B \left( t,\frac{t}{\e},x_h,\frac{z}{\sqrt{\e\nu}} \right).
$$
The decomposition \eqref{CL:uB0_0} leads us to search for a corrector $u_B$ satisfying
$$
u_{B,h}= \sum_{k\in\Z^3} u_{B,h,k},
$$
where each term $u_{B,h,k}$ is an approximate solution of \eqref{eq:depart}, and
$$
u_{B,h,k|\zeta=0}(t,\tau, x_h) = \hat c_{B,h}(t,k) e^{-i \lambda_k\tau}  e^{ik_h'\cdot x_h}.
$$
The periodicity in time of the boundary condition prompts us to choose $u_{B,h,k}$ as a periodic function of $\tau$, with frequency $\lambda_k$. Also, it is classical to seek $u_{B,h,k}$ as an exponentially decaying function of $\zeta$; the rate of decay is then dictated by the equation. The precise expression of $u_{B,h,k}$ is the following (see \cite{MasmoudiCPAM}):

\noindent{$\bullet$ \it First case: $k_h\neq 0$.}

\noindent In this case, $u_{B,h,k}$ is an exact solution of \eqref{eq:ubl0}, and is equal to
\be
u_{B,h,k}(t,\tau, x,y,\zeta)=\sum_{\pm} w_{k}^\pm(t;\om) e^{-i\lambda_k\tau + ik_h'\cdot x_h-\eta_k^\pm\zeta}\label{def:uB}\ee
where
\begin{eqnarray*}
\eta_k^{\pm}&=&\sqrt{1 \mp \lambda_k}\frac{1\pm i}{\sqrt{2}},\\
w_k^\pm(t;\om)&=&\frac{1}{2}\begin{pmatrix}
                                       \hat c_{B,1}(t,k) \pm i  \hat c_{B,2}(t,k)\\\hat c_{B,2}(t,k) \mp i \hat c_{B,1}(t,k)
                                      \end{pmatrix}=\frac{ \hat c_{B,1}(t,k) \pm i \hat c_{B,2}(t,k)}2\begin{pmatrix}  1\\ \mp i\end{pmatrix}.
\end{eqnarray*}

The vertical part of the boundary layer is then given by
\be\label{def:uB3}
u_{B,3,k}(t,\tau, x,y,\zeta)=\sqrt{\e\nu}\sum_{\pm} \frac{1}{\eta_k^\pm}ik'_h\cdot w_{k}^\pm(t;\om) e^{-i\lambda_k\tau + ik_h'\cdot x_h-\eta_k^\pm\zeta}.
\ee

\noindent{$\bullet$ \it Second case: $k_h= 0$.}

\noindent In this case, the construction of the resonant boundary layers  in \cite{forcage-resonant} proves that there are indeed boundary layers, but which are of order $\sqrt{\nu t}$, and not $\sqrt{\e\nu}$ in general. Thus the size of the boundary layer depends (slowly) on time. 

First, notice that if $k_h=0,$ then $-\lambda_k=\sgn(k_3)=\pm 1$. As in the first case, we decompose $\hat c_{B,h}(t,k)$ onto the basis $(1,\pm i):$
$$
\hat c_{B,h}(t,k)=\frac{1}{2}\sum_\pm\left( \hat c_{B,1}(t,k) \mp i \hat c_{B,2}(t,k)\right)\begin{pmatrix}  1\\ \pm i\end{pmatrix}.
$$
As a consequence, we have
\begin{eqnarray*}
\sum_{k_3\in\Z^*}\hat c_{B,h}(t,0,k_3)e^{-i\lambda_k\tau}&=&\alpha_+(t)e^{i\tau} \begin{pmatrix}  1\\ i\end{pmatrix} + \alpha_-(t)e^{-i\tau} \begin{pmatrix}  1\\- i\end{pmatrix}\\
&+&\gamma_+(t)e^{i\tau} \begin{pmatrix}  1\\ -i\end{pmatrix} + \gamma_-(t)e^{-i\tau} \begin{pmatrix}  1\\ i\end{pmatrix},
\end{eqnarray*}
where
$$
\begin{aligned}
\alpha_\pm(t)= \sum_{k_3,\sgn(k_3)=\pm 1}\left( \hat c_{B,1}(t,0,k_3) \mp i \hat c_{B,2}(t,0,k_3)\right),\\
\gamma_\pm(t)= \sum_{k_3,\sgn(k_3)=\pm 1}\left( \hat c_{B,1}(t,0,k_3) \pm i \hat c_{B,2}(t,0,k_3)\right).
\end{aligned}
$$
The terms $ \gamma_\pm e^{\pm i\tau}(1,\mp i)$ give rise to a classical boundary layer term, namely
$$
\sum_\pm \gamma_\pm (t) e^{\pm i\tau-\eta^\pm \zeta} \begin{pmatrix}  1\\ \mp i\end{pmatrix},\quad \text{with } \eta^\pm=1\pm i.
$$
For the terms  $ \alpha_\pm e^{\pm i\tau}(1,\pm i)$,
we rather use the following Ansatz (see \cite{forcage-resonant})
\be\label{def:ustat}
\ustat(t,x_h,z)=\psi\left( \frac{z}{\sqrt{\nu t}} \right)\sum_{\pm}\alpha_\pm(t) e^{\pm i\frac{t}{\e}}\begin{pmatrix}
                                         1\\ \pm i
                                        \end{pmatrix}.
\ee
In order that $\ustat$ is an approximate solution of (the linear part of) equation \eqref{eq:depart}, the function $\psi$ must be such that
$$\begin{aligned}
   -\frac{X}{2}\psi'(X)-\psi''(X)=0,\\
\psi_{|X=0}=1,\quad \psi_{|X=+\infty}=0.
  \end{aligned}
$$
which yields
$$
\psi(X)=\frac{1}{\sqrt{\pi}}\int_X^\infty \exp\left( -\frac{u^2}{4} \right)\:du.
$$
With this definition, $\ustat(t)$ vanishes outside a layer of size $\sqrt{\nu t}$ localized near the bottom of the fluid. Hence $\ustat$ is an approximate solution of the linear part of equation \eqref{eq:depart}, and $\ustat_{|z=a}$ is exponentially small.
\vskip1mm

Now, set
$$
u_B(t,\tau,x_h,\zeta):=\sum_{\substack{k\in\Z^3,\\k_h\neq 0}} u_{B,h,k}\left( t,\tau, x_h,\zeta\right) + \sum_\pm \gamma_\pm (t) e^{\pm i\tau-\eta^\pm\zeta} \begin{pmatrix}  1\\ \mp i\end{pmatrix}.
$$
The complete boundary layer term at the bottom $\ubl_B$ is given by
$$
\ubl_B(t,x_h,z)=u_B\left( t,\frac{t}{\e}, x_h,\frac{z}{\sqrt{\e\nu}} \right)+ \ustat(t,x_h,z).
$$

We now give some estimates on the boundary layer terms constructed in this paragraph:
\begin{lemma}
 Let $u_B$ be defined by \eqref{def:uB}-\eqref{def:uB3} and $u^{\text{stat}}$ by \eqref{def:ustat}. Then the following estimates hold, for all $t>0$
$$\begin{aligned}
 \left\| u_{B,h}(t),\ \zeta\p_\zeta u_{B,h}(t)\right\|_{L^\infty([0,\infty)_\tau,L^2(\T^2\times [0,\infty)_\zeta))}\leq C \left( \sum_{\substack{k\in\Z^3,\\k_h\neq 0}} \left| \hat c_{B,h}(t,k) \right|^2  \frac{|k|}{|k_h|}  |k_3|^2  \right)^{\frac{1}{2}},\\
\left\| u_{B,h}(t),\ \zeta\p_\zeta u_{B,h}(t)\right\|_{L^\infty([0,\infty)_\tau\times\T^2\times [0,\infty)_\zeta)}\leq C \sum_{\substack{k\in\Z^3,\\k_h\neq 0}} \left| \hat c_{B,h}(t,k) \right|,\\
\left\| u_{B,3}(t),\ \zeta\p_\zeta u_{B,3}(t)\right\|_{L^\infty([0,\infty),L^2(\T^2\times [0,\infty)))}\leq C\sqrt{\e\nu} \left( \sum_{\substack{k\in\Z^3,\\k_h\neq 0}} \left| \hat c_{B,h}(t,k) \right|^2  \frac{|k|^3}{|k_h|} |k_3|^2  \right)^{\frac{1}{2}},\\
\left\| u_{B,3}(t),\ \zeta\p_\zeta u_{B,3}(t)\right\|_{L^\infty([0,\infty)\times\T^2\times [0,\infty))}\leq C\sqrt{\e\nu} \sum_{\substack{k\in\Z^3,\\k_h\neq 0}} |k|\,\left| \hat c_{B,h}(t,k) \right|,
\end{aligned}
$$
and
$$\begin{aligned}
\left\| \ustat(t),\ z\p_z \ustat(t) \right\|_{L^2(\U)}\leq C (\nu t)^{1/4}\sum_{k_3\in\Z^*}\left|\hat c_B(t,0,k_3)\right|,\\
\left\| \ustat(t),\ z\p_z \ustat(t) \right\|_{L^\infty(\U)}\leq C\sum_{k_3\in\Z^*}\left|\hat c_B(t,0,k_3)\right|.
\end{aligned}
$$

\label{lem:est_UB}
\end{lemma}
The proof of the above Lemma is left to the reader. Notice that according to the definition of $\eta_k^\pm$, we have
$$
C\frac{ |k_h|}{|k|}\leq \left| \eta_k^\pm\right|\leq 1\quad\forall k\in \Z^3.
$$

\begin{corollary}
Assume that there exists $N>0$ such that
$$
\hat c_{B,h}(t,k)=0\quad \text{if }|k|\geq N,\quad\forall t\geq 0.
$$ 
Then the boundary layer term $\ubl_B$ is an approximate solution of the linear part of equation \eqref{eq:depart}. Precisely, there exists a constant $C_N$, depending only on $N$, such that
\begin{multline*}
 \left\| \p_t \ubl_B + \frac{1}{\e} \ubl_B -\nu\p_z^2 \ubl_B -\Delta_h \ubl_B
 \right\|_{L^\infty([0,T], L^2(\U))}\\
\leq C_N \nu^{1/4} \sup_{t\in[0,T]} \left(\sum_{k\in\Z^3} \left(|\hat c_{B,h}(t,k)|^2 + |\p_t \hat c_{B,h}(t,k)|^2 \right)\right)^{1/2} 
\end{multline*}

\label{cor:est_UB}

\end{corollary}

\begin{proof}
By construction, $\ubl_B$ is an approximate solution of the linear part of equation \eqref{eq:depart}, with an error term equal to
$$
\left[ (\p_t-\Delta_h )u_B\right] \left(t,\frac{t}{\e},x_h,\frac{z}{\sqrt{\e\nu}}\right)+ \varphi\left( \frac{z}{\sqrt{\nu t}} \right)\sum_\pm \p_t\alpha^\pm(t) e^{\pm i\frac{t}{\e}},
$$
where $\p_t$ is the derivation operator with respect to the macroscopic time variable. Thanks to the assumption on the coefficients $\hat c_{B,h}$, we have
$$
\left[\int_0^T\int_{\U}\int_E\left|\p_t u_B\left(t,\frac{t}{\e},x_h,\frac{z}{\sqrt{\e\nu}}\right) \right|^2dm_0(\om)\:dz\:dx_h\:dt\right]^{1/2}\leq C_N\left[ (\e\nu)^{1/4}+\nu^{1/4} \right] \|\p_t c_{B,h} \|
$$
whereas the term $\Delta_h u_B$ is bounded in $L^\infty(E, L^2([0,T], H^{-1,0}))$ by $ C_{N}(\e\nu)^{1/4} \| c_{B,h} \| $.
At last, the error term due to $\ustat$ satisfies
$$
\left\| \varphi\left( \frac{z}{\sqrt{\nu t}} \right)\sum_\pm \p_t\alpha^\pm(t) e^{\pm i\frac{t}{\e}} \right\|_{L^\infty(E,L^2(\U))}\leq C_{N}\nu^{1/4}\|\p_t c_{B,h} \|.
$$\end{proof}

\section{The solution in the interior at main order}

\label{sec:interior}

This section is devoted to the construction of the first order interior terms in expansion \eqref{Ansatz}. At this stage, we merely know how to construct boundary layer terms which deal with the horizontal part of the boundary conditions \eqref{CL}; moreover, following the analysis in paragraph \ref{ssec:gal}, we expect $u^{\e,\nu}(t)$ to behave like $\exp(-tL/\e) w(t)$ in the interior. Thus the idea is to define a function
$$
\uint(t):=\exp\left( -\frac{t}{\e}L \right) w(t) + \vint\left(t,\frac{t}{\e}\right) + \delta \uint\left(t,\frac{t}{\e}\right),
$$
where $\vint$ and $\delta \uint$ are corrector terms, such that $\uint$ is an approximate solution of \eqref{eq:depart}. We also require that 
$$
\uint + u^{\text{BL}, \delta}_T + u_B^\text{BL}
$$
satisfies the boundary conditions \eqref{CL} at main order.

Let us now explain the role of the correctors $\vint$ and $\delta \uint$: it can be checked in the formulas of the previous section that the third components of the boundary layer terms do not vanish at the boundary: indeed, one has 
$$
u^{\text{BL}, \delta}_{T,3|z=a}=\cO (\beta \e\nu\| \sigma\|),\quad u^{\text{BL}}_{B,3|z=0}=\cO(\sqrt{\e\nu}\|w\|).
$$
The role of the corrector $\vint$ is precisely to lift these boundary conditions and to restore the zero-flux conditions at the bottom and at the surface.
Consequently, the term $\vint$ has fast oscillations (at frequencies of order $\e^{-1}$), and in general, $\vint$ is not an approximate solution of \eqref{eq:depart}. Filtering out the oscillations in the term
$$
\frac{\p \vint}{\p t} + \frac{1}{\e}e_3\wedge \vint 
$$
yields the source terms $S_B$ and $S_T$ in the equation satisfied by $w$ (see equation \eqref{eq:barw}). The remaining oscillating terms in the expression above are then taken care of through the addition of the corrector $\delta \uint$.

The organization of this section is as follows: first, we deal with the corrector $\vint$, by giving a precise definition and explaining how oscillations are filtered. Then we derive the equation on the function $w$; in general, this equation depends on the small parameter $\delta$, introduced when constructing the boundary layer terms at the surface. Thus in the  third paragraph, we identify the limit as $\delta \to 0$ of the function $w$, which yields the envelope equation. Eventually, the fourth and last paragraph is concerned with the definition of the corrector $\delta \uint.$

\subsection{Lifting the vertical boundary conditions}

In the rest of this section, we set
\begin{eqnarray*}
 c_{B,3}(t,\tau,x_h)&:=& - \frac{1}{\sqrt{\e\nu}}u_{B,3|\zeta=0}(t,\tau,x_h)\\
&=&-\sum_{\substack{k\in\Z^3,\\k_h\neq 0}}\sum_\pm\frac{ik'_h\cdot w_k^\pm}{\eta_k^\pm}e^{ik'_h\cdot x_h}e^{-i\lambda_k\tau},
\end{eqnarray*}
where 
$$
w_k^\pm(t)=-\frac{1}{2}\mean{N_k,w(t)}\begin{pmatrix}
                             n_1(k) \pm i  n_2(k)\\n_2(k)\mp i n_1(k)
                            \end{pmatrix}
$$
and $\eta_k^\pm$ was defined in the previous section.  In order to shorten the notation, we set
\begin{eqnarray*}
\hat c_{B,3}(t,\tau,k_h)&:=&\int_{\T^2} c_{B,3}(t,x_h) e^{-ik_h'\cdot x_h} dx_h\\
&=&-a_1a_2\sum_{k_3\in \Z}\sum_\pm\frac{ik'_h\cdot w_k^\pm(t)}{\eta_k^\pm}e^{-i\lambda_k\tau}.
\end{eqnarray*}
The function $w$ will be defined in the next section.

Similarly, we set
\begin{eqnarray*}
 c_{T,3}(t,\tau,x_h;\om)&:=&- \frac{1}{\beta{\e\nu}}u^\delta_{T,3|\zeta=0}(t,\tau,x_h;\om)\\
&=&\frac{1}{2}\sum_\pm \int_0^\infty  \left[ \dv_h \sigma \mp i\rot_h\sigma \right](t,\tau-s,x_h;\om)e^{-\delta s \pm is} ds.
\end{eqnarray*}
We also set $\hat c_{T,3}(t,\tau,k_h;\om)=\int_{\T^2}c_{T,3}(t,x_h;\om)e^{-ik_h'\cdot x_h}dx_h,$ so that
$$
\hat c_{T,3}(t,\tau, k_h)=\frac{1}{2}\sum_\pm
\int_0^\infty \hat \sigma^\pm (t,\tau-s, k_h) e^{-\delta s \pm i s}ds.$$

With the above definitions, the function $c_{B,3}$ is quasi-periodic with respect to the fast time variable $\tau$, whereas $c_{T,3}$ is random and stationary with respect to the fast time variable $\tau$.

$\bullet$  \textbf{Defintion of $\vint$. }We now define a function $\vint$ which is divergence free and such that
\be
\begin{aligned}
\vint_{3|z=0}(t,x_h)=\sqrt{\e\nu}c_{B,3}\left(t,\frac{t}{\e},x_h  \right),\\
\vint_{3|z=a}(t,x_h)=\e\nu\beta c_{T,3}\left(t,\frac{t}{\e},x_h  \right).
\end{aligned}\label{CB:uint}\ee
Of course, conditions \eqref{CB:uint} do not determine $ \vint$ unequivocally. A possible choice is
\begin{eqnarray}
\vint_3(t,\tau,x)&=&\frac{1}{a} \left[\e\nu\beta c_{T,3}\left(t,\tau,x_h\right) z + \sqrt{\e\nu} c_{B,3}\left(t,\tau,x_h\right)(a-z)\right]\label{def:v3},\\
\vint_h(t,\tau,x)&=&\frac{1}{a} \nabla_h \Delta_h^{-1}\left[\sqrt{\e\nu} c_{B,3}\left(t,\tau,x_h\right) -\e\nu\beta  c_{T,3}\left(t,\tau,x_h\right)\right].\label{def:vh}
\end{eqnarray}

In fact, since $c_{B,3}$ is an  almost periodic function,  a more convenient choice can be made, which is the so-called ``non-resonant'' choice in \cite{MasmoudiCPAM}. In this case, the equation on $\delta \uint$ is slightly more simple, since there is no source term due to $c_{B,3}$. However, we have chosen here not to distinguish  between stationary and almost periodic boundary conditions, and thus to work with the expressions \eqref{def:v3}, \eqref{def:vh}.


$\bullet$ \textbf{Filtering the oscillations.} We give here the statement and proof of a Lemma which will be useful in the construction of $\delta \uint$ and $w$.

\begin{lemma}Let $T>0$ be arbitrary.

Let $v \in L^\infty([0,T]\times [0,\infty)_\tau, L^2(\T^2\times E))$ such that $\p_\tau v \in  L^\infty([0,T]\times [0,\infty)_\tau, L^2(\T^2\times E))$ and
\begin{eqnarray}
\dv v=0,\label{condV1}\\ 
v_{3|z=0}(t,\tau,x_h)=\sqrt{\e\nu}c_{B,3}(t,\tau,x_h),\label{condV2}\\
v_{3|z=a}(t,\tau,x_h)=\beta{\e\nu}c_{T,3}(t,\tau,x_h)\label{condV3}.
\end{eqnarray}

Then as $\theta\to\infty$, the family
$$
S_\theta:=\frac{1}{\theta}\int_0^\theta\mathcal L(-\tau)\Proj\left[\p_\tau  v+  e_3\wedge v  \right]\:d\tau
$$
converges almost everywhere and in $L^\infty([0,T], L^2(\T^2\times[0,a]\times E))$, and its limit does not depend on the choice of the function $v$. Precisely, 
\be
\lim_{\theta\to\infty}S_\theta=:\bar S[c_{B,3}, c_{T,3}]= \frac{1}{\sqrt{{a}{a_1a_2}}}\sum_{k\in \Z^3}\frac{|k_h'|}{|k'|^2}  \cE_{-\lambda_k}\left[\sqrt{\e\nu}\hat c_{B,3}(k_h)- (-1)^{k_3}\beta{\e\nu} \hat c_{T,3}(k_h) \right]N_k.\label{lim:Stheta}
\ee
\label{lem:source}

\end{lemma}

\begin{remark}
In the above Lemma, the operator $\cE_\lambda$, which was originally defined for random stationary functions, has been extended to almost periodic functions: if
$$
c(\tau)=\sum_{\mu\in M} \alpha_\mu e^{i\mu\tau},
$$ 
with $\sum_\mu |\alpha_\mu|<\infty$, then
$$
\cE_\lambda[c]:=\lim_{\theta\to\infty} \frac{1}{\theta}\int_0^\theta e^{-\lambda\tau}c(\tau) \:d\tau=\mathbf 1_{\lambda=\mu}\alpha_\mu.
$$

\end{remark}


\begin{proof}Let $v^1, v^2$ be two solutions of \eqref{condV1}-\eqref{condV3}, and let $V=v^1-v^2$. Notice that $V\in L^\infty([0,T]\times[0,\infty)_\tau; L^2(E,\mathcal H)),$ and $\p_\tau V\in L^\infty([0,T]\times[0,\infty)_\tau; L^2(E\times \U)).$
We write$$
 \mathcal L(-\tau)\Proj\left[\p_\tau  V + e_3\wedge  V  \right] =  \mathcal L(-\tau)\left[\p_\tau  V + L V \right]= \frac{\p}{ \p \tau }\left[ \mathcal L(-\tau) V(\tau)\right].$$
Consequently,
$$
\frac{1}{\theta}\int_0^\theta\mathcal L(-\tau)\left[\p_\tau  V + \Proj\left( e_3\wedge  V\right)  \right] \:d\tau=\frac{\mathcal L(-\theta) V_{|\tau=\theta}- V_{|\tau=0}}{\theta}.
$$
The right-hand side of the above equality vanishes in $L^\infty([0,T]\times E,L^2(\U))$ as $\theta\to\infty$. Hence the limit is independent of the choice of $v$.
 
In order to complete the proof of the lemma, it is thus sufficient to show that the limit exists for the choice \eqref{def:v3}-\eqref{def:vh}, and to compute the limit in this case. First, we recall that for any function $F\in L^2(\U)$, we have
$$
\Proj F=\sum_{k\in\Z^3,k\neq 0}\mean{N_k, F} N_k.
$$It can be readily checked that if $k_h=0$, then $\mean{N_k,\p_\tau\vint}=0$. 
Thus for all $k=(k_h,k_3)\in\Z^3$ such that $k_h\neq 0$, we compute
\begin{eqnarray*}
\mean{N_k, \p_\tau \vint }&=&\frac{1}a\int_0^a \cos(k_3' z)\overline{n_h(k)}\cdot\frac{ik'_h}{|k'_h|^2}\left(  \e\nu\beta\p_\tau \hat c_{T,3}(\cdot,k_h)- \sqrt{\e\nu}\p_\tau \hat c_{B,3}(\cdot,k_h)\right)\:dz\\
&+&\frac1a\int_0^a \overline{n_3(k)}\sin( k'_3 z)\left( \sqrt{\e\nu}\p_\tau \hat c_{B,3}(\cdot,k_h)(a-z)+ \e\nu\beta\p_\tau  \hat c_{T,3}(\cdot,k_h)z\right)\:dz\\
&=&\overline{n_3(k)}\frac{\mathbf 1_{k_3\neq 0}}{ k'_3}\left[ \sqrt{\e\nu}\p_\tau \hat c_{B,3}(\cdot,k_h)- (-1)^{k_3} \e\nu\beta\p_\tau  \hat c_{T,3}(\cdot,k_h) \right]\\
&+& \mathbf1_{k_3=0}\;\overline{n_h(k)}\cdot\frac{ik'_h}{|k'_h|^2}\left(  \e\nu\beta\p_\tau  \hat c_{T,3}(\cdot,k_h)-\sqrt{\e\nu}\p_\tau \hat c_{B,3}(\cdot,k_h)\right).
\end{eqnarray*}
Notice that if $k_3=0$, then
$$
\overline{n_h(k)}\cdot k'_h=0;
$$
consequently, we have
$$
\mean{N_k, \p_\tau \vint }=-\frac{i}{\sqrt{aa_1a_2}}\frac{\mathbf 1_{k_3\neq 0}|k_h'|}{|k'|k_3'}\left[  \sqrt{\e\nu}\p_\tau \hat c_{B,3}(\cdot,k_h)- (-1)^{k_3}\e\nu\beta\p_\tau  \hat c_{T,3}(t,\tau,k_h;\om) \right].
$$
In a similar way,
\begin{eqnarray*}
 \mean{N_k, e_3\wedge \vint }&=&\frac1a\int_0^a \cos( k'_3 z)\overline{n_h(k)}\cdot \frac{i(k'_h)^{\bot}}{|k'_h|^2}\left( \e\nu\beta\hat c_{T,3}(\cdot,k_h)-  \sqrt{\e\nu}\hat c_{B,3}(\cdot,k_h)\right)\:dz\\
&=&\mathbf 1_{k_3=0}\overline{n_h(k)}\cdot \frac{i(k'_h)^{\bot}}{|k'_h|^2}\left( \e\nu\beta\hat c_{T,3}(\cdot,k_h)-  \sqrt{\e\nu}\hat c_{B,3}(\cdot,k_h)\right)\\
&=&\frac{\mathbf 1_{k_3=0}}{\sqrt{aa_1a_2}}\frac{1}{|k'_h|}\left( \sqrt{\e\nu}\hat c_{B,3}(\cdot,k_h)- \e\nu\beta \hat c_{T,3}(\cdot,k_h)\right).
\end{eqnarray*}
We deduce from the above calculations that
\begin{eqnarray}\label{expr:Sigma}
 &&\mathcal L(-\tau)\Proj(\p_\tau\vint + e_3\wedge \vint )\\\nonumber&=&-\frac{i}{\sqrt{aa_1a_2}}\sum_{k\in \Z^3}\frac{\mathbf 1_{k_3\neq 0}|k'_h|}{k'_3 |k'|}e^{i\lambda_k\tau}\left[ \sqrt{\e\nu} \p_\tau  \hat c_{B,3}- (-1)^{k_3}\e\nu\beta\p_\tau  \hat c_{T,3} \right](t,\tau,k_h;\om)N_k\\\nonumber
&+&\frac{1}{\sqrt{aa_1a_2}}\sum_{k\in \Z^3}\mathbf 1_{k_3=0}\frac{1}{|k'_h|}\left( \sqrt{\e\nu}\hat c_{B,3}(t,\tau,k_h;\om)-  \e\nu\beta\hat c_{T,3}(t,\tau,k_h;\om)\right)N_k.\end{eqnarray}

We decompose the sum in the right-hand side into two sums, one bearing on $k_h$ such that $|k_h|>A$, denoted by $S_{1,A}$, and the other on $|k_h|\leq A$, denoted by $S_{2,A}$, for some $A>0$ arbitrary. Using the fact that $\beta\sqrt{\e\nu}=\cO(1)$, we have
\begin{eqnarray*}
&&\| S_{1,A}(t,\tau)\|_{L^2}^2\\&\leq &C\e\nu\left\|  \sum_{|k_h|>A}\sum_{k_3\in\Z}\frac{\mathbf 1_{k_3\neq 0}|k_h'|}{ k'_3 |k'|}e^{i\lambda_k\tau}\left[  \p_\tau  \hat c_{B,3}- (-1)^{k_3}\beta\sqrt{\e\nu}\p_\tau  \hat c_{T,3} \right](t,\tau,k_h;\om)N_k \right\|^2_{L^2}\\
&+ &C\e\nu\left\| \sum_{|k_h|>A}\frac{1}{|k_h|}\left( \beta\sqrt{\e\nu}\hat c_{T,3}(t,\tau,k_h;\om)-  \hat c_{B,3}(t,\tau,k_h;\om)\right)N_{k_h,0}  \right\|^2_{L^2}\\
&\leq & C\e\nu\sum_{|k_h|>A} \left( |  \p_\tau  \hat c_{B,3} (t,\tau,k_h;\om)|^2 +  |  \p_\tau  \hat c_{T,3} (t,\tau,k_h;\om)|^2 \right)\\
&&+ C\e\nu\sum_{|k_h|>A} \left( |   \hat c_{B,3} (t,\tau,k_h;\om)|^2 +  |   \hat c_{T,3} (t,\tau,k_h;\om)|^2 \right).
\end{eqnarray*}
Since $c_B,c_T,\p_\tau c_B, \p_\tau c_T$ belong to $L^2(\T^2, L^\infty([0,\infty)\times [0,T]\times E))$, we deduce that the sum $S_{1,A}$ vanishes in $L^{\infty}([0,T]\times[0,\infty), L^2(\T^2\times[0,a]\times E))$ as $A\to \infty$. Thus we work with $A$ sufficiently large, but fixed, so that $S_{1,A}$ is arbitrarily small in $L^2$ norm, and we focus on $S_{2,A}$.

For $k\in\Z^3$ fixed, we have, according to Proposition \ref{prop:ergodic},
\begin{eqnarray*}
\frac{1}{\theta}\int_0^\theta e^{i\lambda_k\tau}\p_\tau  \hat c_{T,3}(t,\tau,k_h;\om)\:d\tau
&=&-i\lambda_k\frac{1}{\theta}\int_0^\theta e^{i\lambda_k\tau} \hat c_{T,3} (t,\tau,k_h;\om)\:d\tau\\
&+&\frac{1}{\theta}\left\{ e^{i\lambda_k \theta }   \hat c_{T,3}(t,\theta,k_h;\om) -  \hat c_{T,3} (t,0,k_h;\om) \right\}\\
&\underset{\theta\to\infty}{\longrightarrow}&-i\lambda_k\cE_{-\lambda_k}\left[  \hat c_{T,3}(t,k_h) \right](\om)
\end{eqnarray*}
in $L^\infty([0,\infty)_t, L^2(E)).$ The calculation for $c_{B,3}$ is identical.

Using Lebesgue's Theorem, we deduce that as $\theta\to\infty$
\be\label{limS2}
 \frac{1}{\theta}\int_0^\theta S_{2,A}(t,\tau)\:d\tau
\to  \frac{1}{\sqrt{aa_1a_2}}\sum_{|k_h|\leq A}\sum_{k_3\in\Z}\frac{|k_h'|}{|k'|^2} \cE_{-\lambda_k}\left[ \sqrt{ \e\nu }\hat c_{B,3}(t,k_h)- (-1)^{k_3} \e\nu\beta\hat c_{T,3}(t,k_h) \right]N_k,
\ee
and the convergence holds in $L^\infty([0,T], L^2(\U\times E)).$ Moreover, we have
\begin{eqnarray*}
 \sum_{k\in\Z^3 }\frac{|k_h'|^2}{|k'|^4} \left|\cE_{-\lambda_k} [\hat c_{T,3}(t,k_h)]\right|^2&\leq & C\sum_{k_3\in \Z^*}\frac{1}{1+|k_3|^2}\left\| \cE_{-\lambda_k} [c_{T,3}(t)]\right\|_{L^2(\T^2)}^2\\
&\leq & C \| c_{T,3}\|_{L^\infty([0,\infty)\times[0,\infty)\times E, L^2(\T^2))}^2.
\end{eqnarray*}
A similar estimate holds for $c_{B,3}$. 
Thus the right-hand side of \eqref{limS2}
converges in $L^2(\U\times E)$ as $A\to\infty$. Eventually, we infer \eqref{lim:Stheta}.
\end{proof}
\vskip2mm

$\bullet$\textbf{ Computation of the source terms.} For the sake of completeness, we now derive an expression of $\bar S[c_{B,3}, c_{T,3}]$ in terms of $w$ and $\sigma$.
We begin with $\cE_{-\lambda_k}[\hat c_{B,3}]$. Remembering the definition of $\hat c_{B,3}(t,k_h)$, we have
$$
\cE_{-\lambda_k}\left[ \hat c_{B,3}(t,k_h)  \right]= -a_1a_2 \sum_\pm \frac{ik_h'\cdot w_k^\pm(t)}{\eta_k^\pm}.
$$
Easy calculations lead to
$$
\cE_{-\lambda_k}\left[ \hat c_{B,3}(t,k_h)  \right]=\sqrt{\frac{a_1a_2 }{2a}}\mathbf 1_{k_h\neq 0}\mean{N_k, w(t)}|k_h'|\sum_\pm\frac{1\pm \lambda_k}{\sqrt{1\mp \lambda_k}}\;\frac{1\pm i}{{2}}
$$
Thus, we define the \textit{Ekman pumping term at the bottom of the fluid} by
\be\label{def:Sekman}
S_B(w):=\sum_{k\in\Z^3}\mean{N_k,w} A_k N_k,
\ee
where 
$$
A_k:=\frac{|k_h'|^2}{2\sqrt{2}a|k'|^2}\sum_\pm \frac{1\pm \lambda_k}{\sqrt{1\mp \lambda_k}}\;(1\pm i).
$$

There remains to compute the coefficients $\cE_{-\lambda_k}(\hat c_{T,3}(t,k_h))$; since the boundary condition $c_{T,3}$ depends on the small parameter $\delta$, the corresponding Ekman pumping term will depend on $\delta$ as well. The limit as $\delta$ vanishes of the corresponding source term will be computed in the next paragraph. By definition of $\cE_\lambda$, we have, for all $k_h\in\Z^2$, for all $\lambda\in\R$,
\begin{eqnarray*}
\cEl\left[\hat c_{T,3}(t,k_h)\right](\om)&=&\frac{1}{2}\sum_\pm\lim_{\theta\to\infty}\frac{1}{\theta}\int_0^\theta\int_0^\infty \hat \sigma^\pm(t,\tau-s,k_h;\om)e^{-\delta s -i\lambda \tau\pm is}\: ds\:d\tau\\
&=&\frac{1}{2}\sum_\pm\lim_{\theta\to\infty}\int_0^\infty\left( \frac 1 \theta \int_0^\theta\hat \sigma^\pm(t,\tau,k_h;\theta_{-s}\om) e^{-i\lambda\tau}\:d\tau \right) e^{-\delta s\pm is}\:ds,
\end{eqnarray*}
where 
$$
\hat \sigma^\pm(k_h)=(ik_h'\pm (k_h')^\bot)\cdot \hat \sigma(k_h) .
$$
Thanks to  Lebesgue's dominated convergence Theorem and Proposition \ref{prop:ergodic}, we infer, for all $\delta>0$,
\begin{eqnarray*}
\cEl\left[\hat c_{T,3}(t,k_h)\right](\om)&=&\frac{1}{2}\sum_\pm\int_0^\infty \cEl\left[ \hat \sigma^\pm(t,k_h) \right](\theta_{-s}\om)e^{-\delta s\pm is}\:ds\\
&=&\frac{1}{2}\sum_\pm\int_0^\infty \cEl\left[ \hat \sigma^\pm(t,k_h) \right](\om)e^{-\delta s\pm is-i\lambda s}\:ds\\
&=&\frac{1}{2}\sum_\pm\cEl\left[ \hat \sigma^\pm(t,k_h) \right](\om) \frac{-1}{-\delta + i(- \lambda \pm 1)} .
\end{eqnarray*}
Thus we define the \textit{Ekman pumping term at the top of the fluid} by
\be\label{def:SekmanT-delta}
S_T^\delta(\sigma)=\frac{1}{2} \frac{1}{\sqrt{aa_1a_2}}\sum_{k\in\Z^3}\sum_\pm\frac{(-1)^{k_3}|k_h'|}{|k'|^2} \frac{\cE_{-\lambda_k}\left[ \hat \sigma^\pm(k_h) \right]}{-\delta + i( \lambda_k\pm 1)}N_k.
\ee

Going back to Lemma \ref{lim:Stheta}, we deduce that
$$
\bar S[c_{B,3}, c_{T,3}]= \sqrt{\e\nu} S_B(w) + \e\nu\beta S^\delta_T(\sigma).
$$

\subsection{The envelope equation}
\label{ssec:envelope}
Now that the term $\vint$ is defined, there remains to construct $w$ and $\delta \uint$ such that  $\uint$ is an approximate solution of equation \eqref{eq:depart}. We recall that $\delta \uint, \vint$ are strongly oscillating  terms, small in $H^1$ norm. Consequently, setting $\buint(t,\tau) = \cL(\tau )w(t)$, we have
\begin{eqnarray*}
&&\p_t \uint + \uint\cdot \nabla\uint +  \frac{1}{\e} e_3\wedge \uint -\Delta_h\uint-\nu \p_z^2\uint\\
&\approx& \cL \left( \frac{t}{\e} \right)\p_t w + \buint\cdot \nabla\buint - \Delta_h\buint+ \frac{1}{\e}\left[ \p_\tau \delta \uint + L  \delta \uint\right]\left(t,\frac{t}{\e}  \right)\\&&+ \frac{1}{\e}\left[ \p_\tau \vint +e_3\wedge \vint\right]\left(t,\frac{t}{\e}  \right) + \nabla p^\text{int}\\
&=& \cL \left( \frac{t}{\e} \right)\left[ \p_t w + Q\left(\frac{t}{\e},w,w\right) - \Delta_{h} w  \right]+ \nabla p^\text{int}\\&+& \frac{1}{\e}\left[\cL \left( \tau \right)\p_\tau \left(\mathcal L\left( -\tau \right)  \delta \uint(t,\tau)\right)\right]_{|\tau=\frac{t}{\e}} + \Sigma\left( t,\frac{t}{\e} \right),
\end{eqnarray*}
where
$$
Q(\tau,w,w) = \mathcal L(-\tau) \Proj\left[ \nabla( \mathcal L(\tau) w  \otimes \mathcal L(\tau) w  ) \right].
$$
and $\Sigma$ is defined by
\be
\Sigma(t,\tau):= \frac{1}{\e } \left[\frac{\p }{\p \tau} \vint(t,\tau) +e_3\wedge\vint(t,\tau)\right].\label{def:Sigma}
\ee

Thus it is natural to choose $w$ and $\delta \uint$ such that for all $t,\tau$,
\be
\p_t w + Q(\tau,w,w) - \Delta_{h} w + \mathcal L\left( -\tau \right) \Proj \Sigma\left( t, \tau\right)  +\frac{1}{\e}\p_\tau \left[\mathcal L\left( -\tau \right)  \delta \uint(t,\tau) \right] =0.\label{eq:def_v}
\ee
The quantity $\cL(-\tau) \Proj \Sigma(t,\tau)$ has already been computed in Lemma \ref{lem:source} (see \eqref{expr:Sigma}). Since $w$ does not depend on $\tau$, the first idea is to average the above equation on a time interval $[0,\theta]$, and to pass to the limit as $\theta\to\infty$ in order to derive an equation for $w$.
Assuming that the term $\delta\uint$ is bounded uniformly in $\tau$, we have
$$
\lim_{\theta\to\infty }\int_0^\theta\frac{1}{\e}\p_\tau \left[\mathcal L\left( -\tau \right)  \delta \uint(t,\tau) \right]\: d\tau=0. 
$$
On the other hand, we have already proved in Lemma \ref{lem:source} that 
\begin{eqnarray*}
 \lim_{\theta\to\infty}\frac{1}{\theta}\int_0^\theta\mathcal L\left( -\tau \right) \Proj \Sigma\left( t, \tau\right) \:d\tau &=& \frac{1}{\e}\bar S[c_{B,3}, c_{T,3}]\\&=&\sqrt{\frac{\nu}{\e}} S_B[w] + \nu\beta S_T^\delta(\sigma)\quad\text{in }L^\infty_\text{loc}([0,\infty)_t,L^2(\U\times E)).
\end{eqnarray*}

Moreover, we have
$$
Q(\tau,w,w)=\sum_{k,l,m\in\Z^3}e^{i(-\lambda_l - \lambda_m+ \lambda_k)\tau}\mean{N_l,w}\mean{N_m,w}\mean{N_k, (N_l\cdot \nabla) N_m} N_k,
$$
and it is proved in \cite{CheminDGG} that if $w$ is sufficiently smooth,
$$
\frac{1}{\theta}\int_0^\theta Q(\tau,w,w)\rightharpoonup\bar Q(w,w)
$$
in the distributional sense, where $\bar Q$ is defined by \eqref{def:Q}. Hence, for all $\delta>0$, we define $w^\delta $ as the unique solution in $ L^\infty(E, \mathcal C([0,\infty),\mathcal H \cap H^{0,1} ))\cap  L^\infty(E, L^2_{\text{loc}}([0,\infty), H^{1,0}))$ of the equation
\be\label{eq:envelope1}\begin{aligned}
    &\p_t w^\delta + \bar Q(w^\delta,w^\delta) -\Delta_h w^\delta + \sqrt{\frac{\nu}{\e}}S_B(w^\delta) +\nu\beta S_T^\delta(\sigma)=0,\\
&w^\delta_{|t=0}=u_0\in \cH \cap H^{0,1}. \end{aligned}\ee
We refer to Proposition 6.5 p. 145 in \cite{CheminDGG} and to the comments following Proposition \ref{prop:eqlim} in the Introduction of this paper for existence and uniqueness results about equation \eqref{eq:envelope1}. Notice that if $\sigma\in L^\infty([0,T)\times [0,\infty)_\tau\times E, L^2(\T^2))$ only has a finite number of horizontal modes,  then $S_T^\delta(\sigma)\in L^\infty([0,T]\times E, H^{0,1}).$
Moreover, the fact that $\Re(A_k)\geq 0$ in the definition of $S_B$ implies that   the Ekman pumping due to the Dirichlet condition at $z=0$ induces a damping term in the envelope equation.

\vskip1mm

$\bullet$ The idea is then to pass to the limit in $S_T^\delta(\sigma)$ as $\delta\to 0$ when $\sigma$ satisfies {\bf (H1)-(H2)}, using \eqref{nonres_cEL}. Let us admit for the time being that the last property of Proposition \ref{prop:ergodic} holds, i.e.
\be
\exists \eta>0, \ \forall \lambda\in[-1-\eta, -1+ \eta]\cup[1-\eta, 1+ \eta],\ \cEl(\sigma)=0.\label{nonres_barsigma1}
\ee
Property \eqref{nonres_barsigma1} entails that the sum in the right-hand side of \eqref{def:SekmanT-delta} bears only on the triplets $(k_1,k_2,k_3)$ such that 
$$
|\lambda_k-1|\geq \eta,\ |\lambda_k+1|\geq \eta,
$$
which entails
$$
|k_3|\leq C(\eta) |k_h|.
$$

Consequently, since $\sigma$ only has a finite number of horizontal modes,  we deduce that the sum in the definition of $ S_{T}^{\delta}(\sigma)$ is finite. Hence $ S_{T}^{\delta}(\sigma)$ converges as $\delta\to 0$ in $L^{\infty}([0,\infty)\times E; L^2(\T^2\times[0,a)))$ towards
\be
S_T(\sigma):=-\sqrt{\frac{a_1a_2}{a}}\sum_{\substack{k\in\Z^3,\\ k_h\neq 0}}\frac{(-1)^{k_3}}{|k'_h|}\left( \lambda_k k_h' +i(k_h')^\bot \right)\cdot \cE_{-\lambda_k}\left[ \hat \sigma(k_h) \right]N_k.\label{def:barS_T}
\ee

The same property holds when $\sigma $ has an infinite number of horizontal Fourier modes, provided $\sigma$ is sufficiently smooth with respect to the horizontal variable $x_{h}$ and satisfies \textbf{(H1)-(H2)}.

Thus for all $T_{0}>0$, the source term $S_{T}^\delta(\sigma)$ remains bounded in $L^\infty((0,T_0)\times E,H^{0,1})$ as $\delta\to 0$; whence $w^\delta$ is bounded, uniformly in $\delta$, in $L^\infty(E,\mathcal C([0,T_{0}],\mathcal H \cap H^{0,1} ) \cap L^2([0,T_{0}], H^{1,0}))$. Moreover, let $w$ be the unique solution in $L^\infty(E,\mathcal C([0,\infty),\mathcal H \cap H^{0,1} ))\cap L^\infty(E,L^2_{\text{loc}}([0,\infty), H^{1,0}))$ of

\be\begin{aligned}
    &\p_t w + \bar Q(w,w) -\Delta_h w+  \sqrt\frac{\nu}{\e}S_B(w) +{\nu}{\beta}S_T(\sigma)=0,\label{enveloppe_lim}\\
&w_{|t=0}=u_0.
   \end{aligned}\ee
A standard energy estimate leads to the following error bound, for all $T>0$,
\begin{multline}
 ||w-w_{\delta}||_{L^\infty([0,T]\times E, L^2)} + ||\nabla_h(w-w_{\delta})||_{L^\infty(E,L^2([0,T]\times \U))}\\\leq C{\nu}{\beta} ||S_T(\sigma)-S_T^\delta(\sigma) ||_{L^\infty(E,L^2([0,T]\times \U))}.
\label{in:w_wdelta}\end{multline}

Thus, when constructing the approximate solution in the next section, we will use the function $w^\delta$, but we will keep in mind that $w^\delta$ converges towards $w$ as $\delta$ vanishes.

\vskip1mm


$\bullet$ \label{nonres_barsigma}Let us now turn to the proof of property \eqref{nonres_barsigma1} (which is the same as \eqref{nonres_cEL}). Using \textbf{(H2)}, we choose $\eta_0>0$ such that 
$$
[-1-\eta_0, -1 + \eta_0]\subset V_-,\quad[1-\eta_0, 1 + \eta_0]\subset V_+.
$$
For $\lambda\in\R$ arbitrary, and for $\theta>0$, we have
\begin{eqnarray*}
&&\left\|\frac{1}{\theta}\int_0^\theta \sigma(\tau,\om) e^{-i\lambda\tau}\:d\tau\right\|_{L^\infty([0,T]\times\T^2\times E)}\\&=& \left\|\frac{1}{\theta}\int_0^\theta (\sigma-\sigma_\alpha + \sigma_\alpha)(\tau,\om) e^{-i\lambda\tau}\:d\tau\right\|_{L^\infty([0,T]\times\T^2\times E)}\\
&\leq & ||\sigma-\sigma_\alpha||_{L^\infty((0,\theta)\times[0,T]\times  E\times\T^2)}\\
&& + \frac{1}{\theta}\left\|\int_0^\theta\int_\R e^{-\alpha |\mu|+i\mu\tau-i\lambda\tau}\cFa \sigma(\mu)\:d\mu\:d\tau \right\|_{L^\infty([0,T]\times\T^2\times E)}\\
&\leq & ||\sigma-\sigma_\alpha||_{L^\infty((0,\theta)\times[0,T]\times  E\times \T^2)}\\
&& + \left\|\int_\R e^{-\alpha |\mu|}\frac{e^{i(\mu-\lambda)\theta}-1}{i(\mu-\lambda)\theta}\cFa \sigma(\mu)\:d\mu\right\|_{L^\infty(\T^2)}\\
&\leq &  ||\sigma-\sigma_\alpha||_{L^\infty((0,\theta)\times[0,T]\times  E\times \T^2)}\\
&& + \left(\sup_{\mu\in V_-\cup V_+}\left\| \cFa \sigma(\mu)\right\|_{L^\infty([0,T]\times E\times\T^2)} \right)\left( \left| V_+ \right| + \left| V_- \right| \right)\\
&& + \int_{\R\setminus ( V_-\cup V_+)}e^{-\alpha |\mu|}\left|\frac{e^{i(\mu-\lambda)\theta}-1}{i(\mu-\lambda)\theta}\right| \; \left\|\cFa \sigma(\mu)\right\|_{L^\infty([0,T]\times E\times\T^2)}\:d\mu\:d\tau .
\end{eqnarray*}
Let us now evaluate the last integral when $\lambda$ is close to $\pm 1$, say for instance
$$
|\lambda-1|\leq \frac{\eta_0}{2}.
$$
Then if $\mu\in \R\setminus ( V_-\cup V_+)$, we have $|\mu-1|\geq \eta_0$, and thus
$$
|\mu-\lambda|\geq \frac{\eta_0}{2}.
$$
In particular,
$$
\left|\frac{e^{i(\mu-\lambda)\theta}-1}{i(\mu-\lambda)\theta}\right|\leq \frac{2}{|\mu-\lambda|\theta}\leq \frac{C}{\theta}.
$$
Hence, for all $\theta>0$, for $\lambda$ such that $|\lambda\pm 1|\leq \eta_0/2$, the following inequality holds for all $\alpha>0$
\begin{eqnarray*}
&&\left\|\frac{1}{\theta}\int_0^\theta \sigma(\tau,\om) e^{-i\lambda\tau}\:d\tau\right\|_{L^\infty([0,T]\times \T^2, L^2(E))}\\&\leq &  ||\sigma-\sigma_\alpha||_{L^\infty([0,\theta]\times[0,T] \times E\times \T^2)} + \frac{C}{\theta}\sup_{\alpha>0} \left\| \cFa \sigma(\mu) \right\|_{L^\infty([0,T]\times E\times \T^2, L^1(\R_\lambda))}\\&+& \sup_{\mu\in V_-\cup V_+}\left\| \cFa \sigma(\mu)\right\|_{L^\infty([0,T]\times E\times\T^2)} \left( \left| V_+ \right| + \left| V_- \right| \right).
\end{eqnarray*}
In the above inequality, we first take $\theta$ large enough, so that the left-hand side is close to $\|\bar\sigma(\lambda)\|$, and $C\sup_\alpha \| \cFa\sigma\|/\theta$ is small. Then we let $\alpha$ go to zero, with $\theta$ fixed; we deduce that
$$
\cEl[\sigma]=0\quad\forall \lambda\ \text{such that } d(\lambda,\pm 1)\leq \frac{\eta_0}{2}.
$$

\subsection{Definition of $\delta \uint$}

Once $w$ (or $w^\delta$) and $\vint$ are defined, there only remains to obtain an equation on $\delta \uint$. As stated before, $\delta \uint$ is chosen so that equality \eqref{eq:def_v} holds for all $\tau\geq 0$. According to the above computations, this amounts to taking $\delta \uint$ such that
\begin{eqnarray}\nonumber
\frac{\p}{\p\tau}\left[ \mathcal L(-\tau) \delta\uint(\tau) \right]&=&\e\bar Q(w,w) - \e Q(\tau,w,w) +\bar S[c_{B,3},c_{T,3}] -\e\mathcal L(-\tau) \Proj\Sigma(t,\tau),\\\nonumber
 \mathcal L(-\tau) \delta\uint(\tau) &=&\e\int_0^\tau \left[ \bar Q(w,w)- Q(s,w,w)\right] \:ds \nonumber\\&&\nonumber+ \int_0^\tau \left[ \bar S[c_{B,3},c_{T,3}] - \e\mathcal L(-s) \Proj\Sigma(t,s)\right]\:ds\\\label{def:deltauint}
\delta\uint(\tau)&=&\e \mathcal L(\tau)\int_0^\tau \left[ \bar Q(w,w)- Q(s,w,w)\right] \:ds \\&&+  \mathcal L(\tau) \int_0^\tau \left[  \bar S[c_{B,3},c_{T,3}]  - \e\mathcal L(-s) \Proj\Sigma(t,s)\right]\:ds.\nonumber
\end{eqnarray}
Equivalently, $\delta \uint$ satifies the equation
$$
\p_\tau \delta \uint + L\delta \uint = \e\mathcal L(\tau) \left[ \bar Q(w,w)- \e Q(\tau,w,w) \right]+\mathcal L(\tau)  \bar S[c_{B,3},c_{T,3}]  - \e \Proj\Sigma(t,\tau). $$
We now derive a bound on the coefficients of $\delta\uint$:


\begin{lemma}Let $T>0$, $N>0$, and let  $w\in L^\infty(E,\mathcal C([0,T],\mathcal H ))$ such that 
$$
\mean{N_k,w(t)}=0\quad\forall k,|k|>N,\ \forall t\in[0,T].
$$

Let $\Sigma$ be given by \eqref{def:Sigma}, and $\delta \uint$ by \eqref{def:deltauint}. Then for all $k\in\Z^3$, for all $\eta>0$,  there exists a constant $C_{\eta,k}$ such that for all $\tau\geq 0$,  for all $\e,\nu,\beta>0$ such that $\nu=\mathcal O(\e)$ and $\sqrt{\e\nu}\beta=\cO(1)$
$$
\left\| \mean{N_k,\delta\uint(t,\tau)} \right\|_{L^\infty([0,T], L^2(E))}\leq (\e+ \sqrt{\e\nu})(C_{\eta,k}+ \eta\tau).$$
\label{lem:bound_deltauint}
\end{lemma}

\begin{remark}
The above Lemma is stated with a function $w$ having only a finite number of Fourier modes, which is not the case for the solution of \eqref{eq:envelope1} in general. However, when constructing the approximate solution in the next section, we will consider regularizations of the solution $w$ of the envelope equation \eqref{eq:lim}, so that this issue is in fact unimportant.

\end{remark}

\begin{proof}
We begin with the derivation of a bound for the term
\begin{eqnarray*}
 &&\int_0^\tau \left[ \bar Q(w,w)- Q(s,w,w)\right] \:ds\\
&=&-\sum_{\substack{k,l,m\\ \lambda_l+\lambda_m\neq \lambda_k}}\alpha_{l,m,k}\mean{N_m,w}\mean{N_l,w}\left( \int_0^\tau  e^{i(\lambda_k-\lambda_l-\lambda_m)s} ds \right)N_k.
\end{eqnarray*}
Notice that the set $(l,m)\in\Z^3\times\Z^3$ such that $\mean{N_m,w}\mean{N_l,w}\neq 0$ is finite, and included in $B_N\times B_N$.
Moreover, if $(l,m)\in B_N\times B_N$ and $\lambda_l+\lambda_m\neq \lambda_k$, then there exists a constant $\alpha_{N,k}>0$ such that
$$
|\lambda_l+\lambda_m-\lambda_k|\geq \alpha_{N,k}.
$$
As a consequence, we have
\begin{eqnarray*}
 &&\left| \mean{N_k,\int_0^\tau \left[ \bar Q(w(t),w(t))- Q(s,w(t),w(t))\right] \:ds} \right|\leq  \frac{1}{\alpha_{N,k}} \|w\|_{L^\infty((0,T)\times \T^2\times[0,a]\times E)}^2.
\end{eqnarray*}

In a similar way, we now derive a bound on the second term in \eqref{def:deltauint}. According to Lemma~\ref{lem:source}, we have, for all $k\in\Z^3$,
$$
\frac{1}{\tau}\int_0^\tau \mean{N_k,\mathcal L(-s) \Proj\Sigma(t,s)}\:ds \to \frac{1}{\e}\mean{N_k,\bar S[c_{B,3},c_{T,3}] }
$$
as $\tau\to\infty$, in $L^\infty([0,T],L^2( E))$. Let $\tau_{\eta,k}>0$ such that if $\tau\geq \tau_{\eta,k}$, then
$$
\left\| \frac{1}{\tau}\int_0^\tau \mean{N_k,\mathcal L(-s) \Proj\Sigma(t,s)}\:ds-\frac{1}{\e}\mean{N_k,\bar S[c_{B,3},c_{T,3}] } \right\|_{L^\infty([0,T],L^2( E))}\leq \eta.
$$
Now, for $\tau<\tau_{\eta,k}$, we have
\begin{eqnarray*}
&&\left\| \mean{N_k,\int_0^\tau \left[  \frac{1}{\e}\bar S[c_{B,3},c_{T,3}]  - \mathcal L(-s) \Proj\Sigma(t,s)\right]ds} \right\|_{L^\infty([0,T],L^2( E))}\\&\leq &\tau_{\eta,k}\left\| \mean{N_k,\frac{1}{\e}\bar S[c_{B,3},c_{T,3}] } \right\|_{L^\infty([0,T],L^2( E))}\\&+& \sqrt{\tau_{\eta,k}}\int_0^{\tau_{\eta,k}}\left\| \mean{N_k,\Sigma(\cdot,s)} \right\|_{L^\infty([0,T],L^2( E))}ds\\&\leq &C_{\eta,k}.
\end{eqnarray*}
Gathering all the estimates, we infer the inequality announced in Lemma \ref{lem:bound_deltauint}.
\end{proof}


\section{Proof of convergence}

\label{sec:CVproof}
This section is devoted to the proof of the convergence result in Theorem \ref{thm:conv}. In the previous sections, we have already defined boundary layer terms and interior terms at the main order. Unfortunately, the sum of those first order terms is not a sufficiently good approximation of $u^{\e,\nu}$. Hence the first step of the proof is to define additional correctors, and thus to build an adequate approximate solution. We then derive some technical estimates on the various terms of the approximate solution, and eventually we prove the convergence thanks to an energy estimate.

\subsection{Building an approximate solution}

\label{sec:approx}

The approximate solution is obtained as the sum of some interior terms and some boundary layer terms; although we have to construct several correctors in order to obtain a good approximation of the function $u^{\e,\nu}$, we emphasize that all terms vanish in $L^2$ norm, except the solution $w^\delta$ of the approximated envelope equation \eqref{eq:envelope1}. In this paragraph, we build the correctors step by step, using the general constructions of the previous sections. At each step, we will give some bounds on the corresponding term; these bounds will be proved in the next paragraph.

\vskip2mm

\noindent{\it $\bullet$ First step. The interior term at the main order.}

We have seen that the interior term at main order is given as the solution of the envelope equation \eqref{eq:envelope1}, and that when the parameter $\delta$ vanishes, the envelope equation becomes \eqref{enveloppe_lim}. However, we are not able to construct the boundary layer terms at the top for $\delta=0$, and thus we must keep the approximated solution of the envelope equation, namely $w^\delta$. Moreover, when constructing the corrector terms $\ubl, \delta\uint, \vint$, we will need some high regularity estimates in space and time on $w^\delta$, which are in general not available for $w^\delta $ or $w$. Thus we introduce a \textbf{regularization} of $w^\delta$ with respect to the time variable, and we \textbf{truncate the large frequencies} in $w^\delta$. 

Let $\chi\in \mathcal D(\R)$ be a cut-off function such that 
\begin{gather*}
 \chi(t)=0\quad \forall t\in[0,\infty),\quad
\chi(t)=0\quad \forall t\in(-\infty,-1],\\
\chi(t)\geq 0\quad\forall t\in\R,\ \int_{\R}\chi=1.
\end{gather*}

For $n\in\N^*$, set $\chi_n:=n^{-1}\chi(\cdot/n),$ and define, for $n,N>0$, 
$$
w_{n,N}^\delta:=\Proj_N\left[w^\delta\ast_t\chi_n\right]=(\Proj_N w^\delta)\ast_t \chi_n,
$$
where $\Proj_N$ stands for the projection onto the vector space generated by $N_k$ for $|k|\leq N.$
The convolution in time is well-defined thanks to the assumptions on the support of $\chi$. We have clearly
$$ \begin{aligned}  \lim_{n,N\to\infty}\sup_{\delta>0}\| w^\delta-\wnd\|_{L^\infty([0,T]\times E, L^2)}=0,\\
\lim_{n,N\to\infty}\sup_{\delta>0}\| w^\delta-\wnd\|_{L^\infty(E, L^2([0,T], H^{1,0}))}=0.
  \end{aligned}
$$
Moreover, the following result holds, and will be proved in the next paragraph:
\begin{lemma}
The function $\wnd$ is an approximate solution of \eqref{eq:envelope1}, with an error term $r_{n,N}^\delta$ which vanishes in $L^2([0,T], H^{-1,0})$ as $n,N\to \infty$, uniformly in $\delta$.

\label{lem:est:wnd}

\end{lemma}

Hence we work with $\wnd$ instead of $w$ from now on; for all $k,s>0$, there exists a constant $C_{n,N}(k,s)$ such that
$$
\|\wnd\|_{L^\infty(E,W^{k,\infty}([0,T], H^s(\U))}\leq C_{n,N}(k,s).
$$
In the sequel, we denote by  $C_{n,N}$ all constants depending on $n$ and $N$ (and possibly $T, u_0$ and $\sigma$), but not on $\delta$.
\vskip2mm


\noindent{\it $\bullet$ Second step. The boundary layer terms at the first order.}

The boundary layer terms at main order, $\ubl_B$ and $\ubl_{T}$, are defined in Section \ref{sec:BL}, where the function $w$ is replaced by $\wnd$. Thus $\ubl_{B}$ depends in fact on the parameters $n, N$ and $\delta$, and $\ubl_T$ depends on $\delta$.
Using the results of Proposition \ref{prop:estuT0} and  the previous step, the following estimates can be proved:
\begin{lemma}
We recall that $\nu=\cO(\e) $ and $\beta \sqrt{\e\nu}=\cO(1)$.
Setting
\begin{eqnarray*}
\ubl(t,x_h,z)&:=&\ubl_{B}(t,x_h,z)+\ubl_{T}(t,x_h,z)\\
&=&u_{B}\left(t,\frac t\e,x_{h},\frac z{\sqrt{\e\nu}} \right) + u^\text{BL,res}(t,x_{h},z) + u_{T}\left(t,\frac t\e,x_{h},\frac {a-z}{\sqrt{\e\nu}} \right)
\end{eqnarray*}
we have
\begin{eqnarray}\label{est:ubl}
 &&\left\| \ubl,\  z\p_z\ubl ,\ (z-a)\p_z \ubl \right\|_{L^\infty([0,T]\times\T^2\times[0,a]\times E)} \leq C_{n,N} ,\\
\nonumber&&\left\| \ubl \right\|_{L^\infty([0,T]\times E, H^{1,0})}\leq C_{n,N} \nu^{1/4} ,\\
&&\left\|   z\p_z\ubl ,\ (z-a)\p_z \ubl \right\|_{L^\infty([0,T]\times E, L^2(\U))}\leq C_{n,N}\nu^{1/4}
\nonumber.
\end{eqnarray}

Moreover, $\ubl$
 is an approximate solution of the linear part of equation \eqref{eq:depart}, with an error term bounded in $L^\infty([0,T]\times E,L^2(\U))$ by
$$
 C_{n,N}\nu^{1/4} + C \frac{\delta}{\sqrt{\e}}.
$$ 
 
\label{lem:est:ubl-fin}
\end{lemma}
The above Lemma follows immediately from Lemma \ref{lem:def-uTBL}, Proposition \ref{prop:estuT0}, Lemma \ref{lem:est_UB} and Corollary \ref{cor:est_UB}.

\vskip2mm


\noindent{\it $\bullet$ Third step. The interior corrector terms $\vint$ and $\delta \uint$.}

We now define the correctors $\vint$ and $\delta \uint$ as in \eqref{def:v3}-\eqref{def:vh} and \eqref{def:deltauint} respectively, taking $w=\wnd$ in \eqref{def:deltauint}. Notice  that the boundary conditions $c_{B,3}$ and $c_{T,3}$ are of order one in $L^\infty$. More precisely, using the fact that $\wnd$ has a finite number of Fourier modes on the one hand, and  {\bf (H1)-(H2)} on the other, we deduce that
$$
 \| \vint\|_{L^\infty([0,T]\times[0,\infty)\times\T^2\times[0,a])}\leq C\left(\sqrt{\e\nu}\| \wnd\|_{L^\infty([0,T], H^3)} + {\nu\e}{\beta}\right)\leq C_{n,N}\sqrt{\nu\e};
$$
moreover,  according to Lemma \ref{lem:bound_deltauint},
$$
\forall \eta>0,\ \forall k\in\Z^3,\ \exists C_{\eta,k}>0,\ \left\| \mean{N_k,\delta\uint\left( t,\frac{t}{\e} \right)} \right\|_{L^\infty([0,T],L^2( E))}\leq \eta + C_{\eta,k}\e.
$$
Thus we set, for $K>0$ arbitrary,
$$
\delta \uint_K:= \Proj_K\delta\uint= \sum_{|k|\leq K}\mean{N_k,\delta\uint} N_k.
$$
 According to the above convergence result, for all $K\in \N$, we have
$$
\left\|\delta \uint_K \left( t,\frac{t}{\e} \right)\right\|_{L^\infty([0,T],L^2( E, W^{1,\infty}(\U)))}\to 0\quad\text{as }\e,\nu\to 0.
$$
Moreover, there exists a constant $C_{n,N,K}$ such that
$$
\left\|\delta \uint_K \left( t,\frac{t}{\e} \right)\right\|_{L^\infty([0,T]\times E, W^{1,\infty}(\U))}\leq C_{n,N,K}.
$$
In the rest of the paper, we set
\be\label{def:uint}
\uint(t):= \mathcal L\left( \frac{t}{\e} \right)\wnd(t) + \vint\left( t,\frac{t}{\e} \right) + \delta \uint_K\left( t,\frac{t}{\e} \right);
\ee
the following lemma holds:
\begin{lemma}
Let $r_{n,N}^{\delta}$ be the remainder term in the equation on $\wnd$ (see Lemma \ref{lem:est:wnd}). Then
the function $\uint$ satisfies
$$
\p_t \uint +\frac{1}{\e}e_3\wedge \uint+ \uint\cdot \nabla \uint-\Delta_h\uint -\nu\p_z^2\uint + \nabla p=\cL\left( \frac{t}{\e} \right)r_{n,N}^\delta+\wrm_1 +\wrm_2 + \wrm_{3},
$$
where $\wrm_1=o(1)$ in $L^2([0,T]\times E\times \T^2 \times [0,a])$, $\wrm_2=o(1)$ in $L^2([0,T]\times E, H^{-1,0}),$ and 
$$
\forall n,N,\quad\lim_{K\to\infty}\sup_{\e,\nu,\beta,\delta}\| \wrm_{3}\|_{L^\infty(E,L^2([0,T]\times\T^2\times[0,a])}=0.
$$
Moreover, 
$$
\uint_{|t=0}=u_0 + o(1)\quad\text{in } L^\infty(E, L^2(\times\U)),
$$
and there exists a constant $C_{n,N,K}$ such that
\be
 \| \uint\|_{L^\infty([0,T]\times E, W^{1,\infty}(\U))}\leq C_{n,N,K}.\label{in:est-uint}
\ee
\label{lem:est-uint}
\end{lemma}

In the above Lemma and in the rest of the paper,  the $o(1)$ means that for all $n,N,K$, the corresponding expression vanishes as $\e,\nu\to 0$, uniformly in $\delta$.
\vskip2mm


\noindent{\it $\bullet$ Fourth step. The boundary layer term at the second order.}

At this stage, we have exhibited a function $\uint$ (resp. $\ubl$) which is an approximate solution of the evolution equation \eqref{eq:depart} (resp. of its linear part); moreover, the boundary layer term $\ubl$ and the corrector $\vint$ have been built so that the boundary conditions are satisfied at the leading order. Precisely, we have
$$
\begin{aligned}
\ubl_{h|z=0}(t) + \uint_{h|z=0}(t)&=\vint_{h|z=0}(t,t/\e) + \delta \uint_{K,h|z=0}(t,t/\e) + u_{T,h|\zeta=\frac{a}{\sqrt{\e\nu}}}(t,t/\e),\\
  \p_z\left(\ubl_{h|z=a}(t) + \uint_{h|z=a}(t)\right)&=\beta \sigma(t,t/\e)+\frac{1}{\sqrt{\e\nu}}\p_\zeta u_{B,h|\zeta=\frac{a}{\sqrt{\e\nu}}}(t,t/\e) + \p_z u^\text{BL,res}_{h|z=a}(t),\\
\ubl_{3|z=0}(t) + \uint_{3|z=0}(t)&=u_{T,3|\zeta=\frac{a}{\sqrt{\e\nu}}}(t,t/\e),\\
\ubl_{3|z=a}(t) + \uint_{3|z=a}(t)&=u_{B,3|\zeta=\frac{a}{\sqrt{\e\nu}}}(t,t/\e).
\end{aligned}
$$
The  terms $u_{T|\zeta=\frac{a}{\sqrt{\e\nu}}}$, $ \p_\zeta u_{B, h|\zeta=\frac{a}{\sqrt{\e\nu}}}$, $ u_{B, 3|\zeta=\frac{a}{\sqrt{\e\nu}}}$ and $\ustat_{|z=a}$ are exponentially small, thus satisfy the assumptions of Lemma A.2 in the Appendix; they will be taken care of at the very last step. But in general, setting $\tilde c_{B,h}:=\vint_{h|z=0} + \delta\uint_{K,h|z=0}$, the quantity $\e^{-1}\tilde c_{B,h}$ does not vanish. Thus, we define another boundary layer term in order to restore the Dirichlet boundary condition at $z=0$. We now have to make precise which parts are almost periodic or random stationary in $\tilde c_{B,h}(t,\tau)$. We have
$$
\vint_{h|z=0}=\vint_{h}= -\frac{\beta{\e\nu}}{a}\nabla_h \Delta_h^{-1}(c_{T,3}) + \frac{\sqrt{\e\nu}}{a}\nabla_h \Delta_h^{-1}(c_{B,3}).
$$
The first term in the right-hand side is clearly random and stationary, whereas the second one is almost periodic. Concerning the term $\delta \uint_K$, the situation is not so clear. Using \eqref{def:deltauint}, we write
$$
\delta \uint_K(t,\tau)=\sum_{|k|\leq K}e^{-i\lambda_k\tau} \delta b_k(t,\tau) N_k,
$$
where
\begin{eqnarray*}
 \delta b_k(t,\tau) &:=&\e\mean{N_k, \int_0^\tau \left(\bar Q(\wnd,\wnd) - Q(s,\wnd,\wnd)\right)\:ds}\\
&+& \mean{N_k, \int_0^\tau  \bar S[c_{B,3}, c_{T,3}] - \e\mathcal L(-s) \Proj \Sigma(t,s)}.
\end{eqnarray*}
According to Lemma \ref{lem:bound_deltauint},
$$
\sup_{t\in[0,T]}\left\|\delta b_k\left( t,\frac{t}{\e} \right)\right\|_{L^2(E)}=o(1),
$$
and
$$
\sup_{t\in[0,T]}\left\|\frac{\p}{\p t}\delta b_k\left( t,\frac{t}{\e} \right)\right\|_{L^\infty(E)}\leq C_{n,N}.
$$
Thus we forget the fact that $\delta b_k$ depends on the microscopic time variable $\tau$, and we merely treat $\delta \uint_K$ as an almost periodic function. Hence we use the construction of section \ref{sec:BL} (see in particular Remark \ref{rem:BL-Dir-stat} for the random stationary part), and we denote by  $\delta\ubl$ the boundary layer term thus obtained. By definition,
$$
\delta \ubl_{h|z=0}=-\tilde c_{B,h},
$$
and
$$
\p_t \delta \ubl + \frac{1}{\e} e_3\wedge \delta \ubl - \nu \p_z^2 \delta \ubl=o(1)\quad \text{in } L^\infty([0,T]\times E\times \U).
$$
Using the same kind of estimates as in Lemma \ref{lem:est_UB}, we deduce that
$$
\left\| \delta \ubl_h\right\|_{L^2([0,T]\times E\times \U)}=o((\nu)^{1/4}).
$$

\vskip2mm


\noindent{\it $\bullet$ Fifth step.  The ``stopping'' corrector.}

Let us now examine the remaining boundary conditions.
\begin{itemize}
 \item[$\triangleright$] Horizontal component at $z=0$: this term is the simplest of all. We have
$$
\delta_{B,h}(t):=\left(\uint_h(t)+ \ubl_h(t)+ \delta\ubl_h(t)\right)_{|z=a}=u_{T,h|\zeta=\frac{a}{\sqrt{\e\nu}}}(t,t/\e),
$$
and thus, using the same arguments as in Proposition \ref{prop:estuT0}, we prove that there exists a constant $C$ such that
$$\begin{aligned}
   \|\delta_{B,h}(t) \|_{H^3(\T^2)}\leq C \exp\left( -\frac{a}{\sqrt{\e\nu}} \right)\\
\|\p_t\delta_{B,h}(t) \|_{H^3(\T^2)}\leq \frac{C}{\e}\exp\left( -\frac{a}{\sqrt{\e\nu}} \right).
  \end{aligned}
$$
Since $\e^{-k}\exp\left( -{a}/{\sqrt{\e\nu}} \right)=o(1)$ for all $k\in\N^*$,  $\delta_{B,h}$ satisfies the conditions of Lemma A.2 in the Appendix.

\item[$\triangleright$] Vertical component at $z=0$: we compute
$$
 \delta_{B,3}(t):=\left(\uint_3(t)+ \ubl_3(t)+ \delta\ubl_3(t)\right)_{|z=0}=u_{T,3|\zeta=\frac{a}{\sqrt{\e\nu}}}(t,t/\e) + \delta\ubl_{3|z=0}(t).
$$
It is easily proved that $u_{T,3|\zeta=a/\sqrt{\e\nu}}(t,t/\e) $ satisfies the hypotheses of Lemma A.2, provided $\sigma$ is sufficiently smooth. Concerning $\delta\ubl_{3}$, we have, according to the assumptions on $\sigma,$
$$\begin{aligned}
   \|\delta\ubl_{3|z=0} \|_{L^\infty([0,T], L^2(E,H^3(\T^2))}\leq o(\sqrt{\e\nu}) + C_{n,N,K}{(\nu\e)^{3/2}}{\beta}\leq o(\e),\\
\|\p_t\delta\ubl_{3|z=0} \|_{L^\infty([0,T], L^2(E,H^3(\T^2)))}=o(1).
  \end{aligned}
$$
Thus $\delta_{B,3}$ satisfies the conditions of Lemma A.2. 
\item[$\triangleright$] Horizontal component at $z=a$:
\begin{eqnarray*}
 \delta_{T,h}(t)&=&\p_z\left(\uint_h(t)+ \ubl_h(t)+ \delta\ubl_h(t)\right)_{|z=a}- \frac{1}{\beta}\sigma\left( t,\frac{t}{\e} \right)\\
&=&\frac{1}{\sqrt{\e\nu}}\p_\zeta u_{B,h|\zeta=\frac{a}{\sqrt{\e\nu}}}(t,t/\e) + \p_z \ustat_{h|z=a}(t)+ \p_z  \delta\ubl_{h|z=a}(t).
\end{eqnarray*}
For all $s>0$, we have
$$\begin{aligned}
   \left\|\p_z \ustat_{h|z=a} \right\|_{L^\infty([0,T]\times E, H^s(\T^2))}\leq C_{n,N} \frac{1}{\sqrt{\nu T}}\exp\left( -\frac{a^2}{4\nu T} \right)=o(\e),\\
\left\|\p_t\p_z \ustat_{h|z=a}(t) \right\|_{L^\infty([0,T]\times E, H^s(\T^2))}\leq C_{n,N} \frac{1}{\nu^{3/2} }\exp\left( -\frac{a^2}{4\nu T} \right)=o(\e).
  \end{aligned}
$$
(Remember that $\nu=\mathcal O(\e)$.)
Thus all terms of the right-hand side are exponentially small as $\e$ vanishes, and  $\delta_{T,h}$ satisfies the conditions of Lemma A.2.

\item[$\triangleright$] Vertical component at $z=a$: let
\begin{eqnarray*}
 \delta_{T,3}(t)&:=&\left(\uint_3(t)+ \ubl_3(t)+ \delta\ubl_3(t)\right)_{|z=a}\\&=&u_{B,3|\zeta=\frac{a}{\sqrt{\e\nu}}}(t,t/\e) + \delta\ubl_{3|\zeta=\frac{a}{\sqrt{\e\nu}}}(t).
\end{eqnarray*}
Once again, $ \delta_{T,3}$ is exponentially small in all $H^s$ norms, and thus matches the conditions of Lemma A.2.

\end{itemize}
We thus define $u^{\text{stop}}$, given by Lemma A.2, so that
$$\begin{aligned}
u^{\text{stop}}_{h|z=0}=-\delta_{B,h},\quad&\p_zu^{\text{stop}}_{h|z=a}=-\delta_{T,h}\\
u^{\text{stop}}_{3|z=0}=-\delta_{B,3},\quad&u^{\text{stop}}_{3|z=a}=-\delta_{T,3},
\end{aligned}
$$
and such that $u^{\text{stop}}$ is an approximate solution of the linear part of equation \eqref{eq:depart}, with an error term which is $o(1)$ in $L^2$. Notice that the corrector $u^{\text{stop}}$ itself is $o(\e)$ in $L^2$.

\vskip2mm

We now define
\begin{eqnarray}\label{def:uapp}
 u^{\text{app}}&:=& \uint+ \ubl+ \delta \ubl + u^{\text{stop}}\\
&=& \uint+ u^{\text{rem}}.
\end{eqnarray}
By construction, the remainder $u^{\text{rem}}$ is $o(1)$ in $L^\infty([0,T], L^2(E\times \U)$ and
$
u^{\text{app}}
$ satisfies conditions \eqref{CL}. The goal of the next paragraph is to prove that $u^{\text{app}}$ is an approximate solution of \eqref{eq:depart}, which allows us to  conclude in paragraph \ref{ssec:energy} that $\ug- u^{\text{app}}$ vanishes thanks to an energy estimate.

\subsection{Estimates  on the approximate solution}
We start by proving the lemmas stated in the previous paragraph.

$\bullet$\textit{ Proof of Lemma \ref{lem:est:wnd} (Estimates on $\wnd$).}
 
Remembering \eqref{eq:envelope1}, it is easily checked that $\wnd$ satisfies 
$$
\p_t \wnd + \Proj_N(\bar Q(w^\delta, w^\delta))\ast \chi_n - \Delta_h \wnd + \sqrt{\frac{\nu}{\e}} S_B(\wnd) +\nu\beta \Proj_N S_T^\delta (\sigma\ast \chi_n)=0.
$$
Thus $\wnd$ is an approximate solution of \eqref{eq:envelope1}, with an error term $r_{n,N}^\delta$ equal to
\begin{eqnarray}
r_{n,N}^\delta&=& \bar Q(\wnd,\wnd)-\Proj_N\bar Q(w^\delta,w^\delta)\ast\chi_n + 
{\nu}{\beta}(S_T^\delta(\sigma) - \Proj_N S_T^\delta (\sigma \ast \chi_n))
\nonumber\\
&=&\label{def:rnd}\left[ (\Proj- \Proj_N)\bar Q(w^\delta, w^\delta) \right]\ast \chi_n+ \left[ \bar Q(\Proj_N w^\delta, \Proj_N w^\delta )- \bar Q(w^\delta, w^\delta) \right]\ast \chi_n \\
&&\nonumber+\left[\bar Q(\wnd,\wnd)- \bar Q(\Proj_N w^\delta, \Proj_N w^\delta )\ast \chi_n\right]  + 
{\nu}{\beta}\Proj_NS_T^\delta[\sigma-\sigma\ast\chi_n] + \nu\beta (\Proj-\Proj_N ) S_T^\delta (\sigma).
\end{eqnarray}
In order to evaluate $r_{n,N}^\delta$, we need continuity estimates on the quadratic term $\bar Q$. 
We recall that $\bar Q$ is bilinear continuous from
$$
L^\infty([0,T], H^{0,1})\times L^2([0,T], H^{1,0})\quad\text{into}\quad L^2([0,T], H^{-1,0}).
$$
(see Proposition 6.6 in \cite{CheminDGG} for a proof of this non trivial fact). Moreover, for $a,b\in H^1\cap \cH,$
it can be proved, using the methods of \cite{CheminDGG}, that there exists a constant $C>0$ such that
\begin{eqnarray}
\| \bar Q(a,b)\|_{H^{-1,0}}&\leq &C \| a\|^{1/2}_{L^2}\| a\|^{1/2}_{H^{1,0}} \| b\|^{1/2}_{L^2}\| b\|^{1/2}_{H^{1,0}}  \nonumber\\
&&+C\| \p_3 a\| _{L^2 }\| b\|^{1/2}_{L^2}\| b\|^{1/2}_{H^{1,0}}+ C\| \p_3 b\| _{L^2 }\| a\|^{1/2}_{L^2}\| a\|^{1/2}_{H^{1,0}}\nonumber\\
&\leq & C\left( \| a\|_{H^{1,0}}\| b\|_{H^{0,1}} + \| a\|_{H^{0,1}}\| b\|_{H^{1,0}}  \right)\label{cont-Q}.
\end{eqnarray}
It is easily deduced from the above inequality that the  two terms in line \eqref{def:rnd} converge towards zero; 
on the other hand, the regularity of $\sigma$ entails that $S_T^\delta[\sigma-\sigma\ast\chi_n]$ vahishes in $L^2$ as $n\to\infty$, uniformly in $\delta$, together with $ (\Proj-\Proj_N ) S_T^\delta (\sigma)$.
We thus focus on the last term in the expression of $r_{n,N}^\delta$, which we write 
\begin{eqnarray*}
&&\bar Q(\wnd(t),\wnd(t))-\bar Q(\Proj_N w^\delta, \Proj_N w^\delta )\ast \chi_n\\&=&\int_{\R}\bar Q(\wnd(t),\Proj_Nw^\delta(u))\chi_n(t-u)\:du - \int_{\R}\bar Q(\Proj_Nw^\delta(u),\Proj_Nw^\delta(u))\chi_n(t-u)\:du \\
&=&\int_{\R}\bar Q(\wnd(t)-\Proj_Nw^\delta(u),\Proj_Nw^\delta(u))\chi_n(t-u)\:du,
\end{eqnarray*}
and thus, using inequality \eqref{cont-Q} together with the $L^\infty([0,T], H^{0,1})$ bound on $w^\delta$, we infer
\begin{eqnarray*}
&&\left\| \bar Q(\wnd(t),\wnd(t))-\bar Q(w^\delta,w^\delta)\ast\chi_n (t)\right\|_{H^{-1,0}}\\
&\leq & C\int_{\R} \left\|  \wnd(t)-\Proj_Nw^\delta(u) \right\|_{H^{0,1}}\left\|\Proj_Nw^\delta(u)\right\|_{H^{1,0}}\chi_n(t-u)\:du\\
&+& C\int_{\R} \left\|  \wnd(t)-\Proj_Nw^\delta(u) \right\|_{H^{1,0}}\left\|\Proj_Nw^\delta(u)\right\|_{H^{0,1}}\chi_n(t-u)\:du.
\end{eqnarray*}
Eventually, we get
\begin{eqnarray*}
&&\left\| \bar Q(\wnd(t),\wnd(t))-\bar Q(w^\delta,w^\delta)\ast\chi_n (t)\right\|_{L^\infty(E,L^2([0,T],H^{-1,0}))}\\&\leq & C \sup_{|h|\leq \frac{1}{n}}\left\| w^\delta- \tau_hw^\delta \right\|_{L^\infty(E,L^2([0,T], H^{1,0}))}+  C \sup_{|h|\leq \frac{1}{n}}\left\| w^\delta- \tau_hw^\delta \right\|_{L^\infty([0,T]\times E, H^{0,1})},
\end{eqnarray*}
where $\tau_h w:(t,x)\mapsto w(t+h,x).$ The right-hand side of the above inequality vanishes as $n\to\infty$, uniformly in $\delta$.

Thus $r_{n,N}^\delta $ vanishes as $n,N\to\infty$ in $L^2([0,T]\times E, H^{-1,0})$, uniformly in $\delta$. $\square$

\vskip2mm
Hence we have proved that $\wnd$ is an approximate solution of \eqref{eq:envelope1}. We now tackle the bounds on $\uint$.

\vskip1mm

$\bullet$ \textit{Proof of Lemma \ref{lem:est-uint} (Estimates on $\uint$).}
 First of all, the estimate \eqref{in:est-uint} is easily deduced from the previous bounds on $\wnd,\vint$ and $\delta \uint$. Thus the main point is to check that the assertions on $\wrm_i$, $i=1,2,3$, hold true.

We begin with the term $\wrm_3$, which is due to the truncation of the large frequencies in $\delta\uint$; precisely, we have
$$
\wrm_3(t):=\frac{1}{\e}(\Proj-\Proj_K)\left[ \p_\tau \delta\uint + L \delta\uint \right]\left( t,\frac{t}{\e} \right).
$$
Remembering the definition of $\delta\uint$ (see \eqref{def:deltauint}), we infer
\begin{eqnarray*}
&&\| \wrm_3\|_{L^\infty(E,L^2([0,T]\times\T^2\times[0,a])}\\
&\leq & \left\| (\Proj-\Proj_K)\left[\bar Q(\wnd,\wnd)\right]  \right\|_{L^\infty(E,L^2([0,T]\times\U))}
\\&+& \left\| (\Proj-\Proj_K)\left[Q(s,\wnd,\wnd)\right]  \right\|_{L^\infty([0,\infty)_s\times E,L^2([0,T]\times\U))}
\\&+&\frac{1}{\e}\left\| (\Proj-\Proj_K)\left[\bar S[c_{B,3}, c_{T,3}]\right]  \right\|_{L^\infty(E,L^2([0,T]\times\U))}
\\&+&\left\| (\Proj-\Proj_K)\Sigma\right\|_{L^\infty([0,\infty)\times E,L^2([0,T]\times\U))}.
\end{eqnarray*}
If $\nu=\mathcal O(\e)$, and $\sqrt{\nu\e}\beta=\mathcal O(1)$, all terms vanish  as $K\to\infty$ uniformly in $\e,\nu,\delta$. Thus the condition on $\wrm_3$ is satisfied.

On the other hand, we have defined $\vint$ and $\delta\uint$ so that $\uint$ is an approximate solution of equation \eqref{eq:depart}, with an error term which we now evaluate in $L^2([0,T]\times \U \times E) + L^2([0,T]\times E, H^{-1,0})$. Apart from the one mentioned above, which is due to the truncation of the large spatial frequencies in $\delta\uint$, the error term is equal to
$$\begin{aligned}
\mathcal L\left(\frac{t}{\e}\right)r_{n,N}^\delta(t)+\left[(\p_t -\Delta_h -\nu \p_z^2)(\delta \uint_K+\vint)\right]\left(t,\frac{t}{\e} \right)
\\
+\left[ \uint\cdot \nabla \right](\delta \uint_K + \vint)\left( t,\frac{t}{\e} \right) + \left[(\delta \uint_K + \vint)\left( t,\frac{t}{\e} \right) \cdot \nabla \right] \mathcal L\left( \frac{t}{\e} \right)\wnd(t).
\end{aligned}
$$
Let $
\wrm_2(t):=-\Delta_h \vint\left( t,t/\e \right).
$
Then $\wrm_2$ is bounded in $L^2([0,T]\times E, H^{-1,0})$ by
$$
\sqrt{\e\nu}\| c_{B,3}\|_{L^\infty([0,T]\times [0,\infty)_\tau\times E, H^1(\T^2))} + \e\nu\beta \| c_{T,3}\|_{L^\infty([0,T]\times [0,\infty)_\tau\times E, H^1(\T^2))}=o(1).
$$
Keeping aside $\mathcal L\left(t/\e\right)r_{n,N}^\delta(t)$, the remaining
  error terms are bounded in $L^2([0,T]\times \T^2 \times [0,a]\times E)$ by
\begin{eqnarray*}
&& \| \p_t\delta \uint_K\|_{L^\infty([0,T]_t\times \left[0,\frac{T}{\e}\right]_\tau, L^2(E\times \U))}\\
 &+ &\| \p_t\vint\|_{L^\infty([0,\infty)_\tau, L^2([0,T]\times \T^2 \times [0,a]\times E)}\\
&+& \|\delta \uint_K  \|_{L^\infty([0,T]_t\times \left[0,\frac{T}{\e}\right]_\tau, L^2(E, H^2))} \\
&+& \| \uint\|_{L^\infty} \|\delta \uint_K + \vint \|_{L^\infty([0,T]_t\times \left[0,\frac{T}{\e}\right]_\tau, L^2(E, H^1)) }\\
&+&\| \uint\|_{L^\infty(E, L^2([0,T], H^1))}\|\delta \uint_K + \vint\|_{L^2(E, L^\infty([0,T]_t\times \left[0,\frac{T}{\e}\right]_\tau\times \U))}\\
&= & o(1) .
  \end{eqnarray*}
Above, we have used the fact that $\wnd$, and whence $\vint$, $\delta \uint_K$, are smooth with respect to the time variable $t$.
$\square$

\vskip2mm

$\bullet$ At this stage, we know that $\uint$ is an approximate solution of \eqref{eq:depart}, and that $u^\text{rem}$ is an approximate solution of the linear part of \eqref{eq:depart}, such that additionally $u^\text{rem}=o(1)$. There remains to prove that the function $u^{\text{app}}=\uint + u^\text{rem}$ is an approximate solution of equation \eqref{eq:depart}. 
The core of the proof  lies in the following Lemma:
\begin{lemma}[Non linear estimate on the remainder term]
For all $n,N, $ as $\e,\nu\to 0$ with $\nu=\mathcal O(\e)$ and $\beta\sqrt{\e\nu}=\mathcal O(1),$ we have
$$
\sup_{\delta>0}\left\| \uint\cdot \nabla  u^{\text{rem}} + u^{\text{rem}}\cdot \nabla \uint + u^{\text{rem}}\cdot \nabla u^{\text{rem}}\right\|_{L^2([0,T]\times \T^2\times[0,a]\times E)}\to 0.
$$

\label{lem:non_lin_est}
\end{lemma}

\begin{proof}
 First, we have
\begin{eqnarray*}
&&\left\| \left(u^{\text{rem}}\cdot \nabla\right) \uint \right\|_{L^2([0,T]\times \T^2\times[0,a]\times E)}\\&\leq & \left\|u^{\text{rem}}\right\|_{L^2([0,T]\times \T^2\times[0,a]\times E)} \left\|  \uint \right\|_{L^\infty([0,T]\times E, W^{1,\infty})}\\
&\leq & C_{n,N,K}\left( \|\ubl \|_{L^2} +\|\delta \ubl\|_{L^2} + \| u^{\text{stop}}\|_{L^2} \right).
\end{eqnarray*}
The right-hand side vanishes thanks to the estimates of the previous paragraph.

The other terms are slightly more complicated. We write
\begin{eqnarray*}
 \uint\cdot \nabla  u^{\text{rem}} + u^{\text{rem}}\cdot \nabla u^{\text{rem}}&=&u^{\text{app}}\cdot \nabla u^{\text{rem}}\\
&=&u^{\text{app}}\cdot \nabla u^{\text{stop}}+u^{\text{app}}\cdot \nabla \left( \ubl + \delta \ubl \right).
\end{eqnarray*}

The first term in the right-hand side is bounded in $L^2([0,T]\times E\times \U)$ by
$$
\|u^{\text{app}} \|_{L^\infty}\| u^{\text{stop}}\|_{L^2([0,T]\times E, H^1)}\leq C_{n,N,K}\e.
$$
We thus focus on the second term, which we further split into
$$
u^{\text{app}}_h\cdot \nabla_h \left( \ubl + \delta \ubl \right) + u^{\text{app}}_3\p_z \left( \ubl + \delta \ubl \right).
$$
We have
\begin{eqnarray*}
&& \left\|  u^{\text{app}}_h\cdot \nabla_h \left( \ubl + \delta \ubl \right)\right\|_{L^2([0,T]\times E\times \U)}\\
&\leq & \|  u^{\text{app}}\|_{L^\infty([0,T]\times E\times \U)} \|\ubl + \delta \ubl  \|_{L^2([0,T]\times E, H^{1,0})}\\
&\leq & C_{n,N,K}\nu^{1/4 } .
\end{eqnarray*}

We split the other term as follows
\begin{eqnarray*}
\left\| u^{\text{app}}_3\p_z \left( \ubl + \delta \ubl \right) \right\|_{L^2(\U)}^2
&=&\int_{\T^2}\int_0^{a/2}\left|u^{\text{app}}_3\p_z \left( u_B + \delta \ubl \right) \right|^2\\
&+&\int_{\T^2}\int_0^{a/2}\left|u^{\text{app}}_3\p_z  u_T  \right|^2\\
&+& \int_{\T^2}\int_{a/2}^a\left|u^{\text{app}}_3\p_z \left( u_B + \delta \ubl \right) \right|^2\\
&+&\int_{\T^2}\int_{a/2}^a\left|u^{\text{app}}_3\p_z u_T  \right|^2.
\end{eqnarray*}
For $z\geq a/2$, $t>0$, we have
$$
\left|\p_z \left( u_B + \delta \ubl \right)(t)\right|^2\leq C_{n,N}\left[(\e\nu)^{-1} \exp\left( -\frac{ca}{\sqrt{\e\nu}} \right)+ \frac{1}{\nu t}\exp\left( -\frac{ca}{\sqrt{\nu t} }\right) \right]
$$
and thus
\begin{eqnarray*}
&&\int_0^T\int_{\T^2}\int_{a/2}^a\left|u^{\text{app}}_3\p_z \left( u_B + \delta \ubl \right) \right|^2\\
&\leq & C_{n,N}\left[(\e\nu)^{-1} \exp\left( -\frac{ca}{\sqrt{\e\nu}} \right)+ \exp\left( -\frac{ca}{\sqrt{\nu T}} \right) \right].
\end{eqnarray*}
Similarly,
\begin{eqnarray*}
 &&\int_0^T\int_{\T^2}\int_0^{a/2}\left|u^{\text{app}}_3\p_z  u_T  \right|^2\\
&\leq & C_{n,N}\beta^2 \exp\left( -\frac{ca}{\sqrt{\e\nu}} \right)\leq C_{n,N} (\e\nu)^{-1} \exp\left( -\frac{ca}{\sqrt{\e\nu}} \right).
\end{eqnarray*}
We now evaluate the two remaining terms. The idea is the following: since $u^{\text{app}}_3$ vanishes at the boundary, we have
$$\begin{aligned}
   u^{\text{app}}_3(z)\approx C z\quad\text{for } z=o(1),\\
\text{and }u^{\text{app}}_3(z)\approx C (z-a)\quad\text{for } z-a=o(1),
  \end{aligned}
$$
and $z\p_z u_B$, $(z-a)\p_z u_T$ are evaluated in \eqref{est:ubl}. Moreover, we can split $\uapp$ into
\begin{eqnarray*}
 u^{\text{app}}(t)&=& \left[\mathcal L\left( \frac{t}{\e} \right)\wnd(t) + \delta \uint_K\left(t,\frac{t}{\e} \right) \right] +
\left[ \vint\left( t,\frac{t}{\e} \right) + \ubl(t)\right] \\&+& \left[ \delta\ubl(t) + u^{\text{stop}}(t) \right].
\end{eqnarray*}
By definition of $\vint$ and $u^{\text{stop}}$, the vertical components of each of the three terms in brackets vanish at $z=0$ and $z=a$; additionally, the first term is bounded in $L^\infty([0,T]\times E, W^{1,\infty})$ by a constant $C_{n,N,K}$, while the (vertical components of the) second and  third ones are respectively of order
$$
   C_{n,N}\left(\sqrt{\e\nu} +(\e\nu)^{3/4}\right) \quad\text{and}\quad
o((\e\nu)^{3/4})+ o(\e)
 $$
in $L^\infty([0,T]\times E, H^{1,0}).$ Once again, the formulation $u^{\text{stop}}=o(\e)$ must be understood as 
$$
\forall n, N ,K,\ \lim_{\e,\nu\to 0} \sup_{\delta>0}\e^{-1}\|u^{\text{stop} }\|=0.
$$

As a consequence, we have
\begin{eqnarray*}
 &&\int_{\T^2}\int_0^{a/2}\left|u^{\text{app}}_3(t)\p_z \left( u_B + \delta \ubl \right)(t) \right|^2 \\
&\leq & \left\|  z^{-1}\left[\left(\mathcal L\left( \frac{t}{\e} \right)\wnd\right)_3(t) + \delta \uint_{K,3}\left(t,\frac{t}{\e}\right)\right]\right\|_{L^\infty}^2\left\| z\p_z \left( u_B + \delta \ubl \right)(t) \right\|_{L^2}^2\\
&+& \left\|  z^{-1}\left[\vint_3\left( t,\frac{t}{\e} \right) + \ubl_3(t)+\delta\ubl_3(t) + u^{\text{stop}}_3(t)\right]\right\|_{L^2}^2\left\| z\p_z \left( u_B + \delta \ubl \right)(t) \right\|_{L^\infty}^2.
\end{eqnarray*}
Using Hardy's inequality together with the divergence-free property, we infer that
\begin{eqnarray*}
 &&\int_{\T^2}\int_0^{a/2}\left|u^{\text{app}}_3(t)\p_z \left( u_B + \delta \ubl \right)(t) \right|^2 \\
&\leq & C_{n,N,K}\nu^{1/2}\left\|  \p_z\cL \left( \frac{t}{\e} \right)\wnd + \delta \uint_{K}\left(t,\frac{t}{\e}\right)\right\|_{L^\infty}^2\\
&+&C_{n,N,K}\left\| \p_z\left[\vint_3\left( t,\frac{t}{\e} \right) + \ubl_3(t)+\delta\ubl_3(t) + u^{\text{stop}}_3(t)\right]\right\|_{L^2}^2\\
&\leq & C_{n,N,K}\nu^{1/2}\left\|  \mathcal L\left( \frac{t}{\e} \right)\wnd(t) + \delta \uint_{K}\left(t,\frac{t}{\e}\right)\right\|_{W^{1,\infty}}^2\\
&+&C_{n,N,K}\left\| \vint_h\left( t,\frac{t}{\e} \right) + \ubl_h(t)+\delta\ubl_h(t) + u^{\text{stop}}_h(t)\right\|_{H^{1,0}}^2\\
&\leq & o(1).
\end{eqnarray*}
The term
$$
\int_{\T^2}\int_{a/2}^a\left|u^{\text{app}}_3(t)\p_z u_T (t) \right|^2
$$
is treated in a similar way. Gathering all the terms, we deduce  the convergence result stated in Lemma \ref{lem:non_lin_est}.
\end{proof}

In the rest of this section, following the notations introduced in Lemma \ref{lem:est-uint}, we denote by $\wrm_1$ any term which satisfies  
\be\label{wrm1}
\forall n,N,K,\quad\lim_{\e\to 0}\sup_{\delta>0}\left\|  \wrm_1 \right\|_{L^2([0,T]\times E\times \T^2\times[0,a])}=0,
\ee
by $\wrm_2$ any term which satisfies 
\be\label{wrm2}
\forall n,N,K,\quad\lim_{\e\to 0}\sup_{\delta>0}\left\|  \wrm_2 \right\|_{L^2([0,T]\times E, H^{-1,0})}= 0,
\ee
and by $\wrm_3$ any term which satisfies 
\be\label{wrm3}
\forall n,N,\quad\lim_{K\to\infty}\sup_{\e,\nu,\beta,\delta}\| \wrm_{3}\|_{L^\infty([0,\infty)\times E,L^2([0,T]\times\T^2\times[0,a])}=0.
\ee

According to Lemmas \ref{lem:est:ubl-fin}, \ref{lem:est-uint} and \ref{lem:non_lin_est}, $\uapp$ satisfies an equation of the type
\begin{multline}\label{eq:uapp}
    \p_t \uapp + \uapp \cdot \nabla \uapp + \frac{1}{\e} e_3\wedge \uapp - \Delta_h \uapp -\nu \p_z^2\uapp\\
=\nabla p + \cL\left( \frac{t}{\e} \right)r_{n,N}^\delta+ \wrm_1 + \wrm_2 + \wrm_3+ \mathcal O\left(\frac{\delta}{\sqrt{\e}}\right)_{L^2},
 \end{multline}
We recall that the remainder $r_{n,N}^\delta$, which was defined by \eqref{def:rnd}, satisfies
$$
\lim_{n,N\to\infty}\sup_{\e,\nu,\delta}\| r_{n,N}^\delta\|_{L^2([0,T]\times E, H^{-1,0})}=0.
$$
Equation \eqref{eq:uapp} is supplemented with the boundary conditions \eqref{CL} and the initial condition 
$$
\uapp_{|t=0}=w_0 +\delta w_0^{1} + \delta w_0^{2},
$$
where $\delta w_0^1$ and $\delta w_0^2$ are such that
$$\begin{aligned}
\lim_{n,N\to\infty}\sup_{\delta,\e,\nu}\|\delta w_0^{1}\|_{L^\infty(E,L^2(\U))}=0,\\
\text{and}\quad\forall n,N,\quad\lim_{\e,\nu\to 0}\sup_{\delta>0}\| \delta w_0^{2}\|_{L^\infty(E,L^2(\U))}=0.\end{aligned}$$In order to avoid too heavy notation, we will simply write
$$
\uapp_{|t=0}=w_0 + o(1).
$$

\subsection{Energy estimate}

\label{ssec:energy}

 We now evaluate the difference between $\ug$ and $\uapp$ thanks to an energy estimate.
The function $\ug-\uapp$ is a solution of
\begin{eqnarray*}
&& \p_t(\ug-\uapp) +\frac{1}{\e} e_3\wedge (\ug-\uapp)-\Delta_h(\ug-\uapp) -\nu \p_z^2(\ug-\uapp)\\
&=&\nabla p'+ \wrm_1 + \wrm_2+\wrm_3 - \cL\left(\frac{t}{\e}  \right) r_{n,N}^\delta + \mathcal O\left(\frac{\delta}{\sqrt{\e}}\right)_{L^2}\\&-& (\ug\cdot \nabla)(\ug-\uapp) -\left[(\ug-\uapp)\cdot \nabla \right] \uapp.
\end{eqnarray*}

Taking the scalar product of the above equation by $\ug-\uapp$ and using the Cauchy-Schwarz inequality, we deduce that for all $t>0$, for almost every $\om \in E$,
\begin{eqnarray*}
&&\frac{1}{2}\frac{d}{dt}\| \ug(t,\om)-\uapp(t,\om)\|_{L^2}^2 +\frac{1}{2} \|\ug(t,\om)-\uapp (t,\om)\|_{H^{1,0}}^2 \\
&\leq & \int_{\T^2\times[0,a]}\left| \left[\left((\ug(t,\om)-\uapp(t,\om))\cdot\nabla  \right) \uapp (t,\om)\right]\cdot(\ug(t,\om)-\uapp(t,\om)) \right|\\
&+& \left\| \wrm_1 (t,\om)\right\|_{L^2(\U)}^2 + \left\| \wrm_2(t,\om) \right\|_{H^{-1,0}}^2 +  \left\| \wrm_3 (t,\om)\right\|_{L^2(\U)}^2 \\&+&C \left\|r_{n,N}^\delta(t)  \right\|_{H^{-1,0}}^2 + C \frac{\delta^2}{\e} + C \| \ug(t,\om)-\uapp(t,\om)\|_{L^2}^2 .
\end{eqnarray*}
In the above inequality, we have dropped the term  $\nu\|\p_z(\ug-\uapp)\|_{L^2}^2$ in the left-hand side.
We now evaluate the term
$$
\int_{\T^2\times[0,a]}\left|\left((\ug-\uapp)\cdot \nabla \right)  \uapp \cdot(\ug-\uapp) \right|.
$$
First, let us write
$$
\uapp=\left[ \uint + u^{\text{stop}} \right] + \left[ \ubl + \delta \ubl \right].
$$
The function $\uint + u^{\text{stop}} $ is bounded in $L^\infty([0,T]\times E, W^{1,\infty}(\U)$ by a constant $C_{n,N}$; similarly, $\nabla_h(\ubl + \delta\ubl)$ is bounded in $L^\infty([0,T]\times E\times \U).$ As a consequence, we have
\begin{eqnarray*}
&&\int_{\T^2\times[0,a]}\left| (\ug-\uapp)\cdot \nabla \left[ \uint + u^{\text{stop}} \right] \cdot(\ug-\uapp) \right|\\&+& \int_{\T^2\times[0,a]}\left| (\ug_h-\uapp_h)\cdot \nabla_h \left[ \ubl + \delta \ubl \right] \cdot(\ug-\uapp) \right|\\
&\leq &C_{n,N,K}\|\ug-\uapp \|_{L^2([T^2\times [0,a])}^2.
\end{eqnarray*}
There remains to derive a bound for the term
$$
\int_{\T^2\times[0,a]}\left| (\ug_3-\uapp_3)\p_z \left[ \ubl + \delta \ubl \right] \cdot(\ug-\uapp) \right|;
$$
the calculations are quite similar to those of Lemma \ref{lem:non_lin_est}. We first split the integral on $[0,a]$ into two integrals, one bearing on $[0,a/2]$ and the other on $[a/2,a]$. The term $\ubl_T$ (resp. $\ubl_B+\delta \ubl$) is exponentially small on $[0,a/2]$ (resp. on $[a/2,a]$), and thus we neglect it in the final estimate. Moreover, we have for instance
\begin{eqnarray*}
&& \int_0^{a/2}\int_{\T^2}\left| (\ug_3-\uapp_3)\p_z \left[ \ubl_B + \delta \ubl \right] \cdot(\ug-\uapp) \right|\\
&\leq & \left\| \frac{1}{z}(\ug_3-\uapp_3) \right\|_{L^2}\left\| z\p_z \left[ \ubl_B + \delta \ubl \right] \right\|_{L^\infty}\left\|\ug-\uapp \right\|_{L^2}\\
&\leq & C \| \p_z(\ug_3-\uapp_3) \|_{L^2(\U)}\left\|\ug-\uapp \right\|_{L^2(\U)}\\
&\leq & C \| \ug-\uapp\|_{H^{1,0}}\left\|\ug-\uapp \right\|_{L^2(\U)}.
\end{eqnarray*}
Eventually, we infer that
\begin{eqnarray*}
 &&\int_{\T^2\times[0,a]}\left| (\ug_3-\uapp_3)\p_z \left[ \ubl + \delta \ubl \right] \cdot(\ug-\uapp) \right|\\
&\leq & C \left\|\ug-\uapp \right\|_{L^2(\U)}^2 + C \| \ug-\uapp\|_{H^{1,0}}\left\|\ug-\uapp \right\|_{L^2(\U)}.
\end{eqnarray*}

Gathering all the above estimates and integrating on $E$, we deduce that
\begin{eqnarray*}
&&\frac{\p}{\p t} \left\| \ug- \uapp\right\|_{L^2(E\times \U)}^2 + \left\| \ug- \uapp\right\|_{L^2(E,H^{1,0})}^2\\&\leq& C \left\| \ug- \uapp\right\|_{L^2(E\times \U)}^2 \\
&+&\left\| r_{n,N}^\delta  \right\|_{L^2(E,H^{-1,0})}^2 + \frac{C\delta^2}{\e} \\
&+&\left\| \wrm_1 \right\|_{L^2(E\times \U)}^2 + \left\| \wrm_2 \right\|_{L^2(E,H^{-1,0})}^2 + \left\| \wrm_3 \right\|_{L^2(E\times \U)}^2 .
\end{eqnarray*}
Using Gronwall's Lemma, we infer that for all $t\in[0,T]$,
\begin{eqnarray}\label{energy}
&& \left\| (\ug- \uapp)(t)\right\|_{L^2(E\times \U)}^2 + \int_0^t \left\| \ug- \uapp\right\|_{L^2(E,H^{1,0})}^2\\
&\leq &C\left[\left\| \wrm_1 \right\|_{L^2([0,T]\times E\times \U)}^2  +\left\| \wrm_2 \right\|_{L^2([0,T]\times E,H^{-1,0})}^2+ \left\| \wrm_3 \right\|_{L^2([0,T]\times E\times \U)}^2 \right]\nonumber\\ &+& C\left[\left\| r_{n,N}^\delta  \right\|_{L^2([0,T]\times E,H^{-1,0})}^2+\frac{\delta^2}{\e}\right] \nonumber.
\end{eqnarray}

\vskip2mm
$\bullet$ We are now ready to prove Theorem \ref{thm:conv}. Let us write
\begin{eqnarray*}
 \ug(t)-\mathcal L\left( \frac{t}{\e} \right)w(t) &=&\left[\ug-\uapp\right](t)+\left[  \uapp(t) -\mathcal L\left( \frac{t}{\e} \right)\wnd (t) \right]\\&&+\left[\mathcal L\left( \frac{t}{\e} \right)[\wnd-w^\delta](t) \right] +\mathcal L\left( \frac{t}{\e} \right)[w^\delta-w](t),
\end{eqnarray*}
where
\begin{itemize}
 \item[$\triangleright$]  the term $\ug-\uapp$ satisfies the energy estimate \eqref{energy};
\item[$\triangleright$]  the term $\uapp(t) -\mathcal L\left( \frac{t}{\e} \right)\wnd (t)$ is equal to $u^\text{rem}+ \vint + \delta \uint_K,$ and thus vanishes as $\e,\nu\to 0$ in $L^\infty([0,T], L^2(E, H^{1,0}))$,  uniformly in $\delta>0,$ and for all $n,N,K;$

\item[$\triangleright$]  the term  $\wnd-w^\delta$ vanishes as $n,N\to\infty$ uniformly in $\delta,\e,\nu$ according to the first step in paragraph \ref{sec:approx};
\item[$\triangleright$]  the term $w^\delta -w$ vanishes as $\delta\to 0,$ uniformly in $\e,\nu$, according to \eqref{in:w_wdelta}.
\end{itemize}

Let $\eta>0$ be arbitrary. We first take $n_0,N_0$ large enough so that for all $\delta>0$, $\e,\nu,\beta>0$,
$$\begin{aligned}
   \|r_{n_0,N_0}^\delta  \|_{L^\infty([0,T]\times E,H^{-1,0})}^2 \leq \eta,\\
\|\wnd-w^\delta\|_{L^\infty([0,T]\times E, L^2)}^2, \|\wnd-w^\delta\|_{L^\infty( E, L^2([0,T], H^{1,0}))}^2 \leq \eta.
  \end{aligned}
$$
Thanks to \eqref{wrm3}, we now choose $K>0$ large enough so that for all $\e,\nu,\beta,\delta$,
$$
\left\| \wrm_3 \right\|_{L^2([0,T]\times E\times \U)}^2 \leq \eta.
$$
Remembering properties \eqref{wrm1}-\eqref{wrm2}, we deduce that there exists $\e_0,\nu_0>0$ such that for all $\delta$, for all $\e<\e_0,\nu<\nu_0$ with $\nu\leq C\e$ and $\beta\sqrt{\e\nu}\leq C$,
$$
\begin{aligned}
 \left\| \wrm_1 \right\|_{L^2([0,T]\times E\times \U)}^2  \leq \eta,\\
\left\| \wrm_2 \right\|_{L^2([0,T]\times E,H^{-1,0})}^2\leq \eta,\\
\left\| \uapp(t) -\mathcal L\left( \frac{t}{\e} \right)w_{n_0,N_0}^\delta (t)\right\|_{L^\infty([0,T], L^2(E, H^{1,0}))}^2\leq \eta.
\end{aligned}
$$
At this stage, we have, for all $\delta>0$, for all $\e,\nu,\beta$ such that $0<\e<\e_0$ and $\nu=\mathcal O(\e),$ $\sqrt{\nu\e}\beta=\mathcal O(1),$
\begin{multline*}
 \left\| \ug(t)- \cL \left( \frac{t}{\e} \right)w(t)\right\|_{L^2(E\times \U)}^2 + \int_0^t \left\| \ug(s)- \cL \left( \frac{s}{\e} \right)w(s)\right\|_{L^2(E,H^{1,0})}^2\:ds\\\leq C\eta + C \|w^\delta -w \|_{L^\infty([0,T], L^2(E\times \U))}^2 + C \|w^\delta -w \|_{L^2([0,T]\times E, H^{1,0})}^2 + \frac{C\delta^2}{\e} .
\end{multline*}
We now let $\delta\to 0$ in the right-hand side, and we obtain
$$
\left\| \ug(t)- \cL \left( \frac{t}{\e} \right)w(t)\right\|_{L^2(E\times T^2\times [0,a])}^2 + \int_0^t \left\| \ug(s)- \cL \left( \frac{s}{\e} \right)w(s)\right\|_{L^2(E,H^{1,0})}^2\:ds\leq C\eta
$$
for $\e,\nu$ small enough. The convergence result is thus proved.

\section{Mean behaviour at the limit}

\label{sec:mean_bhv}
This section is devoted to the proof of Proposition \ref{prop:meanbhv}. Let us recall what the issue is: in general, the source term $ S_T(\sigma)$ in \eqref{eq:lim} is a random function, and thus so is $w$. Hence, our goal is to derive an equation, or a system of equations, on $\bE[w]$. We emphasize that such a derivation is not always possible, because of the nonlinear term $\bar Q(w,w)$. However, we shall prove that the vertical average of $w_h$, denoted by $\bar w_h$, is always a deterministic function. Moreover, if the torus is nonresonant (see \eqref{hyp:tor_nonres}), then $w-\bar w$ solves a linear equation, and thus in this particular case we can derive an equation for $\bE[w-\bar w]$.

Our first result is the following:
\begin{lemma}Assume that the group transformation $(\theta_\tau)_{\tau\in\R}$ is ergodic.
Let $u_0\in \cH\cap H^1$, and let $w$ be the solution of \eqref{eq:lim}. Set
$$
\bar w_h=\frac{1}{a}\int_0^a w_h.
$$
Then $\bar w_h$ is the unique solution in $\mathcal C([0,\infty), L^2(\T^2))\cap L^2_\text{loc}([0,\infty), H^1(\T^2))$ of the two-dimensional Navier-Stokes equation
\be
\left\{
\begin{array}{l}
\ds\p_t\bar w_h + \bar w_h \cdot \nabla \bar w_h-\Delta_h\bar w_h + \frac1{a\sqrt{2}}\sqrt{\frac{\nu}{\e}} \bar w_h+\nu\beta \bE\left[ S_T(\sigma)\right]_h=0,\\
\bar w_{h|t=0}=\frac{1}{a}\int_0^a w_{0,h}.
\end{array}
\right.\label{eq:barw}
\ee
In particular, $\bar w_h$ is a deterministic function.

\end{lemma}

\begin{proof}
Let us recall that if
$$
\phi=\sum_{k\in\Z^3}\hat \phi(k) N_k\in\cH,
$$ 
then
$$
P_h(\phi):=\frac{1}{a}\int_0^a\phi_h=\sum_{k_h\in\Z^2} \hat \phi(k_h,0)e^{ik_h'\cdot x_h}n_h(k_h,0).
$$
Thus we have to project equation \eqref{eq:lim} onto the horizontal modes, which correspond to $k_3=0$. 
It is easily checked that
$$
P_h\left( S_{B}( w)\right)= S_{B,h}(\bar w_h)= \frac{1}{\sqrt{2}a}\bar w_h,
$$
and we recall (see \cite{MasmoudiCPAM} and Proposition 6.2 in \cite{CheminDGG}) that there exists a function $\bar p\in L^2(\T^2)$ such that for all $w\in H^1\cap \cH$
$$
P_h(\bar Q(w,w) )=(\bar w_h \cdot \nabla_h)\bar w_h + \nabla_h \bar p.
$$
Thus we only have to prove that
$$
P_h( S_{T}(\sigma))= \bE\left[ S_{T,h}(\sigma)\right],$$
almost surely in $E$. We use the following fact, of which we postpone the proof: if $\lambda\in\R$,  then
\be
\bE[\cEl[\sigma]]= \left\{ \begin{array}{ll}
                                          \bE[ \sigma]&\text{ if }\lambda=0,\\
0&\text{ else. }
                                         \end{array}
 \right.\label{esp:sigma1}
\ee
Moreover, if $\lambda=0$, then
\be
\cEl[\sigma]=\bE[ \sigma]\quad\text{almost surely.}\label{esp:sigma2}
\ee
Note also that $\lambda_k=0$ if and only if $k_3=0$. Remembering \eqref{def:barS_T}, we deduce from \eqref{esp:sigma1} and \eqref{esp:sigma2} that

\begin{eqnarray*}
 \bE[ S_{T,h}(\sigma)]&=&-\frac{i}{\sqrt{aa_1a_2}} \sum_{k_h\in\Z^2}\frac{1}{|k'_h|^2}(k'_h)^\bot\cdot \bE[\hat \sigma(k_h)] e^{ik_h'\cdot x_h}\begin{pmatrix}
                                                                                                      ik_2'\\-ik_1'
                                                                                                     \end{pmatrix}
\\
&=&-\frac{i}{\sqrt{aa_1a_2}}\sum_{k_h\in\Z^2}\frac{1}{|k'_h|^2}(k'_h)^\bot\cdot \cE_0[\hat \sigma(k_h)] e^{ik_h'\cdot x_h}\begin{pmatrix}
                                                                                                      ik_2'\\-ik_1'
                                                                                                     \end{pmatrix}\\
&=&P_h[ S_T(\sigma)].
\end{eqnarray*}

Thus the lemma is proved, pending the derivation of \eqref{esp:sigma1} and \eqref{esp:sigma2}. Concerning \eqref{esp:sigma1}, the invariance of the probability measure $m_0$ with respect to $\theta_\tau$ entails that
$$
\bE\left[\cEl[\sigma]\right]=\bE[\sigma]\lim_{\theta\to\infty}\frac{1}{\theta}\int_0^\theta e^{-i\lambda\tau}\;d\tau,
$$
and \eqref{esp:sigma1} follows easily. Equality \eqref{esp:sigma2} is a consequence of Birkhoff's ergodic theorem (see \cite{sinai}).
\end{proof}

The first point in Proposition \ref{prop:meanbhv} follows easily from the above Lemma (together with Theorem \ref{thm:conv}), by simply noticing that the sequence
$$
\exp\left( -\frac{t}{\e} L \right)w(t)=\sum_k e^{-i\lambda_k\frac{t}{\e}}\mean{N_k, w(t)} N_k
$$
weakly converges in $L^2([0,T]\times \U\times E)$ towards
$$
\sum_{\substack{k\in\Z^3,\\\lambda_k=0}}\mean{N_k, w(t)} N_k=\bar w(t)=(\bar w_h(t),0).
$$

\begin{remark}
Notice that
$$
\rot_h P_h[ S_T(\sigma)]=-\sqrt{\frac{a_1a_2}{a}}bE \left[ \rot_h \sigma \right].
$$ 
Hence we recover the result of \cite{DesjardinsGrenier}: the vorticity $\phi:=\rot_h \bar w_h$ is a solution of
$$
\p_t \phi + \bar w_h\cdot \nabla_h \phi -\Delta_h \phi +\frac{1}{a\sqrt{2}}\sqrt{\frac{\nu}{\e}}\phi={\nu\beta}\sqrt{\frac{a_1a_2}{a}} \bE \left[ \rot_h \sigma \right].
$$

\end{remark}

From now on, we assume that the torus is nonresonant (see \eqref{hyp:tor_nonres}). Consequently, with $\bar w=(\bar w_h,0),$ we have
$$
\bar Q(w-\bar w, w-\bar w)=0.
$$
Moreover, using \eqref{esp:sigma1}-\eqref{esp:sigma2}, it is easily checked that
$$
\bE\left[ S_{T,3}(\sigma) \right]=0.
$$
Setting $u= w-\bar w$, we deduce that $u$ solves a linear equation, namely
$$
\p_t u + 2\bar Q(u,\bar w) - \Delta_h u + \sqrt{\frac{\nu}{\e}} S_B(u)+ \nu\beta  S_T(\sigma) -\nu\beta \bE[S_T(\sigma)]=0.
$$
Since $\bar w$ is deterministic, we have
$$
\bE\left[ \bar Q(u,\bar w)  \right]=\bar Q(\bE[u],\bar w) .
$$
Hence we can further decompose $u$ into $\tilde w+ \tilde u$, where $\tilde w$ is deterministic and does not depend on $\sigma$, and $\tilde u$ is random with zero average. The precise result is stated in the following lemma, from which Proposition \ref{prop:meanbhv} follows immediately:

\begin{lemma}
Assume that the hypotheses of Proposition \ref{prop:meanbhv} hold. Then 
$$
w=\bar w + \tilde w+ \tilde u
$$
where:
\begin{itemize}
 \item the function $\bar w$ is deterministic and satisfies \eqref{eq:barw};
\item the function $\tilde w$ is deterministic and satisfies
$$
\left\{
\begin{array}{l}
\ds\p_t\tilde w +2 \bar Q(\bar w, \tilde w)-\Delta_h\tilde w +\sqrt{\frac{\nu}{\e}} S_B(\tilde w) =0,\\
\tilde w_{|t=0}=u_0-\bar w_{|t=0};
\end{array}
\right.
$$
\item the function $\tilde u$ is random, with zero average, and satisfies
$$
\left\{
\begin{array}{l}
\ds\p_t\tilde u + 2\bar Q(\bar w, \tilde u)-\Delta_h\tilde u + \sqrt{\frac{\nu}{\e}} S_B(\tilde u)+ \nu\beta S_T(\sigma) - \nu\beta\bE[S_T(\sigma) ] =0,\\
\tilde u_{|t=0}=0.
\end{array}
\right.
$$
\end{itemize}

\end{lemma}

\section*{Appendix A: convergence of the family $\sigma_\alpha$}~

{\sc Lemma A.1} {\it 
Let $T>0$. Assume that $\sigma\in L^\infty([0,T]\times E, \mathcal C(\R))\cap L^\infty([0,T]\times \R_\tau \times E)$. Then for all $T'>0$, $$
\sigma_\alpha- \sigma \to 0 \quad \text{ in }L^\infty((0,T)\times (0,T')\times E)\ \text{ as }\alpha \to 0.
$$}

\begin{proof}
By definition of $\sigma_\alpha$, we have
\begin{eqnarray*}
 \sigma_\alpha(t,\tau,\om)&=&\frac{1}{2\pi}\int_{\R\times\R} \exp(-\alpha |\lambda|-\alpha |s|) e^{i \lambda (\tau-s)}\sigma(t,s,\om)\:ds\:d\lambda\\
&=&\frac{1}{2\pi}\int_{\R}\exp(-\alpha |s|)\frac{2 \alpha}{\alpha^2 + (\tau-s)^2}\sigma(t,s,\om)\:ds\\
&=&\frac{1}{\pi}\int_{\R}\exp(-\alpha |\tau + \alpha s|)\frac{1}{1+ s^2}\sigma(t,\tau + \alpha s,\om)\:ds.
\end{eqnarray*}
Consequently,
\begin{eqnarray*}
\sigma(t,\tau,\om) -\sigma_\alpha(t,\tau,\om)&=&\frac{1}{\pi}\int_{\R}\exp(-\alpha |\tau + \alpha s|)\frac{1}{1+ s^2}\left[ \sigma(t,\tau,\om)-\sigma(t,\tau + \alpha s,\om) \right]\:ds\\
&&+ \frac{1}{\pi}\sigma(t,\tau,\om)\int_{\R}\left[ 1- \exp(-\alpha |\tau + \alpha s|)\right]\frac{1}{1+ s^2}\:ds.
\end{eqnarray*}
The convergence result of Lemma A.1 follows easily.
\end{proof}

\section*{Appendix B: proof of Proposition \ref{prop:ergodic}}

Let $\lambda\in\R$ be arbitrary, and let $\phi\in L^2(E)$.

Consider the probability space
$$
E_\lambda:=E \times [0,2\pi),\quad P_\lambda:= P \otimes \frac{d\mu}{2\pi},
$$
where $\mu$ is the standard Lebesgue measure on $[0,2\pi]$. Let us define the following group of transformations, acting on $(E_\lambda, P_\lambda)$
$$
\mathcal T^\lambda_\tau(\om,\varphi):= (\theta_\tau\om , \varphi -\lambda\tau\ \mathrm{mod} 2\pi),\quad \tau\in\R.
$$
Then it is easily checked that $\mathcal T^\lambda_\tau$ is measure-preserving for all $\tau\in\R$. And if $T>0$, we have, for all $\varphi\in[0,2\pi]$,
\begin{eqnarray*}
\int_0^T \Phi(\theta_\tau\om) e^{-i\lambda\tau}\:d\tau&=& e^{-i\varphi}\int_0^T \Phi(\theta_\tau\om) e^{i\varphi-i\lambda\tau}\:d\tau\\
&=&e^{-i\varphi}\int_0^T \Psi\left( \mathcal T^\lambda_\tau(\om,\varphi) \right)\:d\tau,
\end{eqnarray*}
where the function $\Psi\in L^2(E_\lambda)$ is defined by $$\Psi\left( \om,\varphi \right):=\Phi(\om) e^{i\varphi}.$$

Hence, according to Birkhoff's ergodic theorem (see \cite{sinai}), there exists a function $\Psi^\lambda\in L^2(E_\lambda) $, invariant by the group of transformations $\left(\mathcal T^\lambda_\tau\right)_{\tau\in\R}$, such that
$$
\frac{1}{T}\int_0^T \Phi(\theta_\tau\om) e^{-i\lambda\tau}\:d\tau\to e^{-i\varphi}\Psi^\lambda(\om,\varphi),
$$
$P_\lambda$ - almost surely in $E_{\lambda}$ and in $L^2(E_{\lambda})$. Moreover, the function
$$
(\om,\Phi)\mapsto e^{-i\varphi}\Psi^\lambda(\om,\varphi) 
$$
clearly does not depend on $\varphi$ by construction. Hence, we set
$$
\Phi^{\lambda}(\om):=e^{-i\varphi}\Psi^\lambda(\om,\varphi)\quad \forall(\om,\varphi)\in E_\lambda,
$$
and we have proved that
$$
\frac{1}{T}\int_0^T \Phi(\theta_\tau\om) e^{-i\lambda\tau}\:d\tau\to \Phi^{\lambda}(\om)
$$
almost surely in $\om$ and in $L^2(E)$.

Now, since $\Psi^\lambda$ is invariant by the group $\left(\mathcal T^\lambda_\tau\right)_{\tau\in\R}$ and $\Phi^\lambda$ does not depend on $\varphi$, we have, almost surely in $\om$,
\begin{eqnarray*}
 \Phi^{\lambda}(\theta_\tau\om)&=&e^{-i\varphi}\Psi^\lambda(\theta_\tau\om,\varphi)\\
&=&e^{-i(\varphi-\lambda\tau)}\Psi^\lambda(\theta_\tau\om,\varphi-\lambda\tau\ \mathrm{mod} 2\pi)\\
&=&e^{-i(\varphi-\lambda\tau)}\Psi^\lambda\left( \mathcal T^\lambda_\tau(\om,\varphi) \right)\\
&=&e^{-i(\varphi-\lambda\tau)}\Psi^\lambda\left(\om,\varphi \right)\\
&=&e^{i\lambda\tau}\Phi^{\lambda}(\om).
\end{eqnarray*}
This completes the proof of Proposition \ref{prop:ergodic}.

\section*{Appendix C: the stopping Lemma}~

{\sc Lemma A.2} (Stopping condition) {\it Let $T_0>0$, and let $\delta_B,\delta_T\in L^\infty([0,T_0], H^3(\T^2))$ be two families  such that $$\int (\delta_{T,3}-\delta_{B,3}) dx_h=0$$
and such that as $\e\to 0,$ for $A=T,B$,
$$\frac1\e \| \delta _A\| _{ L^\infty([0,T],H^1(\T^2))} \to 0,\quad  \| \delta _A\| _{ L^\infty([0,T],H^3(\T^2))} \to 0\hbox{ and } \|\p _t  \delta_A \| _{ L^\infty([0,T],H^1(\T^2))} \to 0.$$
Then there exists a family $u^\text{stop} \in L^\infty([0,T], L^2(\U))$ with $\nabla\cdot w=0$ such that
$$
u^\text{stop} _{|z=0} =\delta_B, \quad u^\text{stop} _{3|z=a} =\delta_{T,3} \hbox{ and } \p_z u^\text{stop} _{h|z=a} =\delta_{T,h}
$$
and such that  as $\e\to 0,$
$$
\frac{1}{\e}\| u^\text{stop}  \|_{L^\infty([0,T],L^2)} \to 0,\;
\left\| \p_t u^\text{stop}  + \frac1\e Lu^\text{stop}  -\nu\p_{zz} u^\text{stop} -\Delta_h u^\text{stop}  \right\| _{L^\infty([0,T],L^2)} \to 0.
$$}
\vskip2mm

For a proof of the above Lemma, see \cite{forcage-resonant}.

\section*{Acknowledgements}
I wish to express my deepest gratitude to Laure Saint-Raymond, for  helpful discussions during the preparation of this work, and to the D\'epartement de math\'ematiques et applications (\'Ecole normale sup\'erieure, Paris), for its hospitality during the academic year 2007-2008.

\bibliography{windstress_stationary}

\providecommand{\bysame}{\leavevmode\hbox to3em{\hrulefill}\thinspace}
\providecommand{\MR}{\relax\ifhmode\unskip\space\fi MR }
\providecommand{\MRhref}[2]{%
  \href{http://www.ams.org/mathscinet-getitem?mr=#1}{#2}
}
\providecommand{\href}[2]{#2}
\begin{thebibliography}{10}

\bibitem{AP}
Yves Achdou and Olivier Pironneau, \emph{Domain decomposition and wall laws},
  C. R. Acad. Sci. Paris S\'er. I Math. \textbf{320} (1995), no.~5, 541--547.

\bibitem{Allaire}
Gr\'egoire Allaire, \emph{Homogenization and two-scale convergence}, SIAM J.
  Math. Anal. (1992), no.~23, 1482--1518.

\bibitem{BMN}
A.~Babin, A.~Mahalov, and B.~Nicolaenko, \emph{Integrability and regularity of
  {$3$}{D} {E}uler and equations for uniformly rotating fluids}, Comput. Math.
  Appl. \textbf{31} (1996), no.~9, 35--42.

\bibitem{CheminDGG}
J.-Y. Chemin, B.~Desjardins, I.~Gallagher, and E.~Grenier, \emph{Mathematical
  geophysics}, Oxford Lecture Series in Mathematics and its Applications,
  vol.~32, The Clarendon Press Oxford University Press, Oxford, 2006, An
  introduction to rotating fluids and the Navier-Stokes equations.

\bibitem{forcage-resonant}
A.-L. Dalibard and L.~Saint-Raymond, \emph{Mathematical study of resonant
  wind-driven oceanic motions}, preprint, 2008.

\bibitem{DesjardinsGrenier}
B.~Desjardins and E.~Grenier, \emph{On the homogeneous model of wind-driven
  ocean circulation}, SIAM J. Appl. Math. \textbf{60} (2000), no.~1, 43--60
  (electronic).

\bibitem{SaintRayGallagherHbk}
I.~Gallagher and L.~Saint-Raymond, \emph{On the influence of the earth's
  rotation on geophysical flows}, Handbook of Mathematical Fluid Dynamics
  (Susan Friedlander and Denis Serre, eds.), vol.~4, Elsevier, 2007.

\bibitem{dgv}
David G{\'e}rard-Varet, \emph{Highly rotating fluids in rough domains}, J.
  Math. Pures Appl. (9) \textbf{82} (2003), no.~11, 1453--1498.

\bibitem{Gill}
A.E. Gill, \emph{Atmosphere-{O}cean dynamics}, International Geophysics series,
  vol.~4, 1982.

\bibitem{grenier}
E.~Grenier, \emph{Oscillatory perturbations of the {N}avier-{S}tokes
  equations}, J. Math. Pures Appl. (9) \textbf{76} (1997), no.~6, 477--498.

\bibitem{MasmoudiGrenier}
E.~Grenier and N.~Masmoudi, \emph{Ekman layers of rotating fluids, the case of
  well prepared initial data}, Comm. Partial Differential Equations \textbf{22}
  (1997), no.~5-6, 953--975.

\bibitem{ladyparabolic}
O.~A. Ladyzhenskaya, V.~A. Solonnikov, and N.~N. Ural'tseva, \emph{Linear and
  {Q}uasilinear {E}quations of {P}arabolic type}, American Mathematical
  Society, 1968.

\bibitem{Lannes}
David Lannes, \emph{Nonlinear geometrical optics for oscillatory wave trains
  with a continuous oscillatory spectrum}, Adv. Differential Equations
  \textbf{6} (2001), no.~6, 731--768.

\bibitem{LT}
C.~David Levermore, Marcel Oliver, and Edriss~S. Titi, \emph{Global
  well-posedness for models of shallow water in a basin with a varying bottom},
  Indiana Univ. Math. J. \textbf{45} (1996), no.~2, 479--510.

\bibitem{MasmoudiRousset}
N.~Masmoudi and F.~Rousset, \emph{Stability of oscillating boundary layers in
  rotating fluids}, preprint, 2007.

\bibitem{Masmoudi1}
Nader Masmoudi, \emph{The {E}uler limit of the {N}avier-{S}tokes equations, and
  rotating fluids with boundary}, Arch. Rational Mech. Anal. \textbf{142}
  (1998), no.~4, 375--394.

\bibitem{MasmoudiCPAM}
\bysame, \emph{Ekman layers of rotating fluids: the case of general initial
  data}, Comm. Pure Appl. Math. \textbf{53} (2000), no.~4, 432--483.

\bibitem{Ng}
Gabriel N'Guetseng, \emph{A general convergence result for a functional related
  to the theory of homogenization}, SIAM J. Math. Anal. (1989), no.~20,
  608--623.

\bibitem{PapanicolaouVaradhan}
G.~Papanicolaou and S.~R.~S. Varadhan, \emph{Boundary value problems with
  rapidly oscillating random coefficients}, Rigorous results in Statistical
  Mechanics and Quantum Field Theory (J.~Fritz, J.~L. Lebaritz, and D.~Szasz,
  eds.), Proc. Colloq. Random Fields, vol.~10, Coll. Math. Soc. Janos Bolyai,
  1979, pp.~835--873.

\bibitem{Pedlosky1}
J.~Pedlosky, \emph{Geophysical fluid dynamics}, Springer, 1979.

\bibitem{Pedlosky2}
\bysame, \emph{Ocean {C}irculation theory}, Springer, 1996.

\bibitem{Rousset}
F.~Rousset, \emph{Stability of large {E}kman boundary layers in rotating
  fluids}, Arch. Ration. Mech. Anal. \textbf{172} (2004), no.~2, 213--245.

\bibitem{schochet}
Steven Schochet, \emph{Fast singular limits of hyperbolic {PDE}s}, J.
  Differential Equations \textbf{114} (1994), no.~2, 476--512.

\bibitem{sinai}
Ya.~G. Sinai, \emph{Introduction to ergodic theory}, Princeton University
  Press, Princeton, N.J., 1976, Translated by V. Scheffer, Mathematical Notes,
  18.

\end{thebibliography}

\end{document}